\documentclass[11pt,letterpaper]{article}
\usepackage[margin=0.9in]{geometry}
\usepackage[utf8]{inputenc} 
\usepackage[T1]{fontenc}    
\usepackage{dsfont}
\usepackage{multirow}
\usepackage{mathrsfs}
\usepackage{wrapfig}
\usepackage[T1]{fontenc}
\usepackage{float}
\usepackage{bm}

\usepackage{booktabs}       
\usepackage{amsfonts}       
\usepackage{nicefrac}       
\usepackage{microtype}      
\usepackage{enumerate}
\usepackage{lipsum}
\usepackage{pifont}
\usepackage{tcolorbox}
\usepackage{mathtools}
\usepackage{cuted}
\usepackage{float}
\usepackage{mdframed}
\usepackage[dvipsnames,table,xcdraw]{xcolor}
\usepackage{bbm}
\usepackage{varioref}
\usepackage{hyperref}
                                
\usepackage{amssymb} 
\usepackage{rotating}
\usepackage{graphicx}
\usepackage{subcaption}
\usepackage{caption}
\usepackage{xcolor}
\usepackage{pifont}
\usepackage[normalem]{ulem}

\usepackage{amsthm}
\usepackage{algorithm}
\usepackage{algpseudocode}
\usepackage{amsmath}
\usepackage{tikz}
\usetikzlibrary{arrows.meta, positioning}
\usetikzlibrary{decorations.pathreplacing, calc}
\usetikzlibrary{shapes, shapes.geometric}

\newtheorem{theorem}{Theorem}[section]
\newtheorem{lemma}{Lemma}[section]
\newtheorem{remark}{Remark}[section]
\newtheorem{proposition}{Proposition}[section]
\newtheorem{definition}{Definition}[section]
\newtheorem{corollary}{Corollary}[section]
\theoremstyle{plain}

\newcommand{\abbr}[1]{\textup{(}#1\textup{)}}

\title{\Large
Smoothing-Enabled Randomized Stochastic Gradient Schemes for Solving Nonconvex Nonsmooth Potential Games under Uncertainty
}
\author{Zhuoyu Xiao\footnote{Department of Industrial and Operations Engineering, University of Michigan, Ann Arbor, MI, 48109. Email: zyxiao@umich.edu.}}

\newmdenv[
  backgroundcolor=gray!20,
  linecolor=white,
  linewidth=0pt,
  innertopmargin=10pt,
  innerbottommargin=10pt,
  innerleftmargin=10pt,
  innerrightmargin=10pt,
]{custombox}

\begin{document}

\maketitle
\thispagestyle{empty}

\begin{abstract}
The state of the art in solving nonconvex nonsmooth games under uncertainty remains in its infancy. Existing studies primarily rely on stringent growth and sharpness conditions or local convexity-like properties, making the development of alternative algorithms desirable. In this work, we study the Clarke-Nash equilibria (CNE) computation of a class of stochastic $N$-player nonconvex nonsmooth games characterized by a potential function. We first consider the nonconvex smooth setting and develop a randomized stochastic gradient (RSG) scheme. The RSG scheme achieves the optimal sample complexity of $\mathcal{O}(N^{2}\epsilon^{-4})$ for reaching a point whose expected residual has norm at most $\epsilon$. Building on this result, we introduce a randomized smoothed RSG (RS-RSG) scheme for solving stochastic potential games afflicted by nonconvexity and nonsmoothness. We show that RS-RSG asymptotically converges to an equilibrium of the smoothed game with sample complexity $\mathcal{O}(L^{4}_{\max}n^{3/2}_{\max}N^{3}\eta^{-1}\epsilon^{-4})$,  where $\eta>0$ is the smoothing parameter. Under Lipschitz continuity of the Clarke subdifferentials, we show that the expected residual evaluated at the smoothed equilibrium is $\mathcal{O}(\eta^{2})$. In addition, we discuss the biased RSG and RS-RSG variants and demonstrate the effectiveness of the biased RS-RSG scheme on a class of stochastic potential hierarchical games where the exact lower-level solution is unavailable in finite time. Collectively, our results provide a new pathway that goes beyond classical conditions for solving stochastic nonconvex nonsmooth games. Some preliminary numerics are also provided.
\end{abstract}


\section{Introduction}\label{Sec-1}

The Nash equilibrium (NE) \cite{nash-1951} serves as a central solution concept in noncooperative multi-agent decision-making, capturing situations in which no participant can reduce its cost or improve its payoff by unilaterally altering its strategy. This framework naturally arises in engineered fields and economic environments \cite{hobbs-pang-2007, metzler-hobbs-pang-2003, scutari-palomar-facchinei-pang-2010, scutari-palomar-pang-facchinei-2009} where rational agents interact through coupled constraints or shared resources. The study of NE has therefore become essential for characterizing system-level outcomes that emerge from decentralized behavior. In recent years, the increasing focus on large-scale optimization and machine learning has further expanded the role of Nash equilibria, such as convex–concave saddle-point minimax problems \cite{kovalev-gasnikov-2022, lin-jin-jordan-2020, wang-li-2020} in high-dimensional regimes. Many efforts in computing equilibria in continuous-strategy games with both smooth and nonsmooth objectives have relied on gradient-response and best-response methods, along with their stochastic variants, largely under the setting where each player’s problem is convex \cite{facchinei-pang-2009, koshal-nedic-shanbhag-2016, lei-shanbhag-2020, lei-shanbhag-2022, lei-shanbhag-pang-sen-2020}. In this paper, we consider the setting where player-specific functions are possibly nonconvex, nonsmooth, and expectation-valued. Specifically, our focus lies on the stochastic $N$-player noncooperative game
\begin{equation*}\label{games}
    \min_{x_i \in X_i} ~ f_i(x_i, x_{-i}) \triangleq \mathbb{E} \Big[ \tilde{f}_i(x_i, x_{-i}, \xi) \Big], ~ \forall i\in [N], \tag{$\mathrm{G}$}
\end{equation*}
where the random variable $\xi: \Omega \to \mathbb{R}^m$ is defined on the probability space $(\Omega, \mathcal{F}, \mathbb{P})$ and $\Xi \triangleq \left\{ \xi(\omega) \mid \omega \in \Omega \right\}$, each $X_{i} \subseteq \mathbb{R}^{n_{i}}$ is convex and closed, $x_{-i} = (x_j)_{j \ne i} \in X_{-i} \triangleq {\prod_{j \ne i}}\ X_j$, $x = (x_{i})_{i=1}^{N} \in X \triangleq {\prod_{j =1}^{N}}\ X_j \subseteq \mathbb{R}^{n}$ with $n \triangleq \sum_{i=1}^{N} n_{i}$, and each $\Tilde{f}_{i}(\bullet, x_{-i}, \xi)$ is a possibly nonconvex and nonsmooth real-valued function for given $x_{-i}$ and $\xi$.

\subsection{Motivating Examples}\label{two-motivating-examples}

Most existing studies on stochastic game-theoretic problems focus on convex and smooth regimes. However, nonconvexity and nonsmoothness are pervasive in practical stochastic game-theoretic models and remain largely unexplored. Below we provide two examples to illustrate how widespread this problem class is in economic and engineered applications.

\subsubsection{Stochastic Nash-Cournot Games with Piecewise-Linear Concave Costs}

Nash--Cournot models provide a canonical framework for studying strategic production in oligopolistic market equilibrium. Classical analysis often relies on increasing convex production costs, which leads to tractable variational inequality formulations. However, such assumptions may fail to capture economies of scale \cite{osullivan-sheffrin-swan-2003} and decreasing marginal costs. The class of piecewise-linear concave cost functions \cite{muu-nguyen-quy-2008} provides a flexible alternative, but also introduces nonconvexity and nonsmoothness, thereby motivating the study of efficient computational schemes for the resulting games. 

To this end, we consider the stochastic generalization of the problems considered in \cite{muu-nguyen-quy-2008}, where the $i$th player solves the following parameterized stochastic optimization problem for given $x_{-i}\in X_{-i}$:
\begin{equation}\label{NC-i}
    \min_{x_{i}\in X_{i}} ~ f_{i}(x_{i}, x_{-i}) \triangleq \mathbb{E} \left[ \Tilde{c}_{i}(\xi) \right] g_{i}(x_{i}) - \mathbb{E} \left[ \Tilde{p}(\bar{x}, \xi) \right] x_{i}, ~ \forall i\in [N], \tag{NC-PL}
\end{equation}
where $\mathbb{E}[\Tilde{c}_{i}(\xi)]g_{i}(x_{i})$ denotes the expected private cost of the $i$th player, and the linear inverse demand function is $\Tilde{p}(\bar{x}, \xi) = a(\xi) - b(\xi)\bar{x}$ for random variables $a(\xi)$ and $b(\xi)$, with $\bar{x} \triangleq \sum_{i=1}^{N} x_{i}$. Suppose that $\mathbb{E}[\Tilde{c}_{i}(\xi)] > 0$ and the increasing piecewise-linear concave cost function $g_{i}(x_{i}): \mathbb{R}_{+} \to \mathbb{R}_{+}$ is given by 
\begin{equation*}
    g_{i}(x_{i}) \triangleq \begin{cases}
        a_{0} x_{i}, \;\; &\textrm{if} \;\; 0 \leq x_{i} \leq d_{1}, \\
        a_{1} x_{i} + b_{1}, \;\; &\textrm{if} \;\; d_{1} \leq x_{i} \leq d_{2}, \\
        \qquad \vdots \\
        a_{m} x_{i} + b_{m}, \;\; &\textrm{if} \;\; d_{m} \leq x_{i},
    \end{cases}
\end{equation*}
where we assume that $a_{0} > a_{1} > \cdots > a_{m} > 0$ and $0 < b_{1} < b_{2} < \cdots < b_{m}$. An illustrative example is provided in Figure \ref{piecewise-linear-concave-cost}, where $g_i(x_i)$ is seen to be a nonconvex nonsmooth cost function.

\begin{figure}
    \centering
    \begin{tikzpicture}[
    >=Stealth, 
    xscale=0.18, 
    yscale=0.12
]

    \draw[->, thick] (-3, 0) -- (44, 0) node[below left] {$x_{i}$};
    \draw[->, thick] (0, -3) -- (0, 34) node[above] {$g_{i}(x_{i})$};

    \coordinate (Start) at (-1.5, -3); 
    \coordinate (O)   at (0, 0);
    \coordinate (D1)  at (4, 8);   
    \coordinate (D2)  at (12, 16); 
    \coordinate (D3)  at (32, 26); 
    \coordinate (End) at (41, 28.2); 

    \draw[thin, gray] (D1) -- (4, 0) node[below, black] {$d_{1}$};
    \draw[thin, gray] (D2) -- (12, 0) node[below, black] {$d_{2}$};
    \draw[thin, gray] (D3) -- (32, 0) node[below, black] {$d_{3}$};

    \draw[thick] (Start) -- (O);
    \draw[very thick] (O) -- (D1) -- (D2) -- (D3) -- (End);

    \node[below right] at (1.8, 5) {$a_{0} x_{i}$};
    \node[above left] at (10, 13) {$a_{1} x_{i} + b_{1}$};
    \node[above left] at (23, 21.5) {$a_{2} x_{i} + b_{2}$};
    \node[above] at (36, 29) {$a_{3} x_{i} + b_{3}$};
\end{tikzpicture}
\caption{A piecewise-linear concave cost function.}
\label{piecewise-linear-concave-cost}
\end{figure}

\subsubsection{Stochastic Nonconvex Nonsmooth Hierarchical Games}

Stochastic hierarchical games \cite{cui-shanbhag-2023, cui-shanbhag-staudigl-2025, lei-shanbhag-chen-2026} provide a framework for modeling multi-agent decision problems in which each player’s objective is shaped not only by rival strategies, but also by the solution of a lower-level equilibrium problem. Such models arise naturally in stochastic bilevel programming \cite{chen-sun-yin-2021, hong-wai-wang-yang-2023}, stochastic MPECs \cite{cui-shanbhag-yousefian-2023, patriksson-wynter-1999}, and stochastic multi-leader multi-follower games \cite{aussel-svensson-2020, demiguel-xu-2009}. Existing computational efforts, however, have largely focused on regimes where the implicit player-specific objectives are convex and the associated mappings satisfy monotonicity property. These classical assumptions are crucial for variational characterizations and convergence guarantees, but they exclude many applications involving nonconvex objectives or nonmonotone mappings. This motivates the development of algorithms for nonconvex and nonmonotone regimes.

To this end, we consider a class of stochastic $N$-player two-stage nonconvex hierarchical games with uncertainty in the lower-level equilibrium problem. Suppose that the $i$th leader's problem is defined as
\begin{equation*}
    \min_{x_{i}\in X_{i}} ~ f_{\mathrm{H}, i}(x_{i}, x_{-i}) \triangleq \mathbb{E} \Big[ \Tilde{h}_{i}(x_{i}, y_{i}(x_{i}), \xi) \Big] + \mathbb{E} \Big[ \Tilde{m}_{i}(x_{i}, x_{-i}, \xi) \Big], ~ \forall i\in [N],
\end{equation*}
where $f_{\mathrm{H}, i}(\bullet, x_{-i})$ is possibly nonconvex, and 
$y_{i}(x_{i})$ denotes a solution to the follower's lower-level stochastic variational inequality $\mathrm{SVI}\:( \mathbb{E}[\Tilde{F}_{i}(x_{i}, \bullet, \xi)], Y_{i} )$ parameterized by the leader's decision $x_i$, i.e.,
\begin{equation*}
    y_{i}(x_{i}) \in \mathrm{SOL}\:( \mathbb{E}[\Tilde{F}_{i}(x_{i}, \bullet, \xi)], Y_{i} ).
\end{equation*}
The implicit dependence of the leader's objective on the lower-level solution map $y_i(\bullet)$ introduces nonsmoothness, thereby giving rise to a stochastic nonconvex nonsmooth hierarchical game.

\subsection{Related Work}\label{Sec-1.1}

This paper is closely related to several influential works in stochastic, nonconvex, and nonsmooth optimization and game-theoretic problems, as well as in nonmonotone variational inequalities. We review them below.

~

\noindent\textbf{SA and RSG in stochastic programming.} The stochastic approximation (SA) method was first introduced by Robbins and Monro \cite{robbins-monro-1951} for stochastic strongly convex problems, with acceleration later presented in \cite{polyak-juditsky-1992} by Polyak and Juditsky. Nemirovski et al. \cite{nemirovski-juditsky-lan-shapiro-2009} proposed the mirror descent SA scheme for solving general nonsmooth convex stochastic programming. For unconstrained and constrained stochastic nonconvex problems, Ghadimi et al.\ \cite{ghadimi-lan-2013, ghadimi-lan-zhang-2016} proposed the RSG scheme with the sample complexity of $\mathcal{O}(\epsilon^{-4})$ and its zeroth-order variant. In Arjevani et al. \cite{arjevani-carmon-duchi-foster-srebro-woodworth-2023}, they showed that the $\mathcal{O}(\epsilon^{-4})$ complexity bound is optimal for first-order methods in stochastic nonconvex optimization under the standard unbiasedness and bounded variance assumptions.

~

\noindent\textbf{Randomized smoothing and zeroth-order methods.} In this work, we study a class of stochastic games with player-specific Lipschitz continuous objectives that are possibly nonconvex and nonsmooth. One line of research for deterministic nonconvex nonsmooth optimization with Lipschitz continuous objectives relies on gradient-sampling methods \cite{burke-curtis-lewis-overton-simoes-2020, burke-lewis-overton-2005, kiwiel-2007, kiwiel-2010}. Another avenue leverages zeroth-order approaches reliant on randomized smoothing \cite{lin-zheng-jordon-2022, marrinan-shanbhag-yousefian-2026}, applicable in both deterministic and stochastic regimes. Specifically, Marrinan et al.\ \cite{marrinan-shanbhag-yousefian-2026} extended the work by Lin et al. \cite{lin-zheng-jordon-2022} to constrained settings and developed a variance-reduced zeroth-order gradient scheme variant. The sample complexities for driving the residual of the $\eta$-smoothed problem to at most $\epsilon$ are $\mathcal{O}(n^{3/2}(L^{4}+L^{3}\eta^{-1})\epsilon^{-4})$ and $\mathcal{O}(n^{3/2}(L^{5}+L^{3}\eta^{-2})\epsilon^{-4})$ in \cite{lin-zheng-jordon-2022} and \cite{marrinan-shanbhag-yousefian-2026}, respectively, where $L$ is the Lipschitz continuity constant and $n$ is the problem dimension. The zeroth-order methods have also been applied to solve stochastic MPECs \cite{cui-shanbhag-yousefian-2023} and hierarchical federated optimization \cite{qiu-shanbhag-yousefian-2023} recently.

~

\noindent\textbf{Nonconvex games and nonmonotone VIs.} For certain nonconvex games, NE may fail to exist \cite{pang-scutari-2011, pang-scutari-2013}. Two weaker solution concepts of quasi-Nash equilibria (QNE) \cite{pang-scutari-2011} and Clarke-Nash equilibria (CNE) \cite{cui-pang-2021} were introduced recently. From the computational standpoint, there are two approaches for solving continuous-strategy noncooperative games. The first approach concerns gradient–response (GR) schemes based on variational inequality (VI) \cite{facchinei-pang-2003} formulations. While many studies assume monotonicity \cite{facchinei-pang-2009, koshal-nedic-shanbhag-2013, lei-shanbhag-2022, yousefian-nedic-shanbhag-2016}, recent works consider nonmonotone VIs under the Minty condition and its variant \cite{alacaoglu-kim-wright-2024, diakonikolas-daskalakis-jordan-2021, hsieh-iutzeler-malick-mertikopoulos-2020, huang-zhang-2024, vankov-nedic-sankar-2023}, pseudomonotonicity and its variants \cite{dang-lan-2015, iusem-jofre-oliveira-thompson-2017, iusem-jofre-oliveira-thompson-2019, kannan-shanbhag-2019, kotsalis-lan-li-2022, yousefian-nedic-shanbhag-2017}, and cohypomonotonicity \cite{cai-alacaoglu-diakonikolas-2023, cai-oikonomou-zheng-2022, cai-zheng-2022, lee-kim-2021}. Xiao and Shanbhag \cite{xiao-shanbhag-2025-gr} recently established last-iterate guarantees for solving stochastic nonconvex smooth games under certain growth and sharpness conditions. However, these growth and sharpness conditions are relatively stringent and hard to verify sometimes. Moreover, existing works largely concern the smooth setting, whereas subgradient and smoothing-enabled gradient methods for the nonconvex nonsmooth regime remain largely unexplored. The second approach concerns best-response (BR) procedures. The synchronous and asynchronous BR schemes have been developed for convex games \cite{cui-shanbhag-2023, lei-shanbhag-2020, lei-shanbhag-2022, lei-shanbhag-pang-sen-2020, xiao-shanbhag-2026-br}, and for nonconvex games via surrogations \cite{cui-pang-2021, pang-razaviyayn-2016, razaviyayn-2014, xiao-shanbhag-2026-br}. However, studies on synchronous BR schemes typically rely on a contraction assumption, implicitly requiring that the player-specific objectives satisfy convexity or local convexity-like properties.

\vspace{8pt}

\begin{tcolorbox}
    \textbf{Question.} Can we develop efficient smoothing-enabled stochastic gradient schemes with convergence and rate guarantees that go beyond classical growth, sharpness, or local convexity-like conditions for solving nonconvex nonsmooth games under uncertainty?
\end{tcolorbox}

\subsection{Main Contributions and Outline}\label{Sec-1.2}

Our main contributions are articulated next.
\begin{itemize}
    \item\textbf{Potentiality-based GR schemes.} There are three common approaches employed for GR and BR schemes in computing equilibria: contraction approach, potentiality approach, and the VI-based approach. The first two are leveraged within the BR framework \cite{cui-shanbhag-2023, cui-pang-2021, lei-shanbhag-2020, lei-shanbhag-2022, lei-shanbhag-pang-sen-2020, pang-razaviyayn-2016, razaviyayn-2014, xiao-shanbhag-2026-br}, while the GR schemes rely on the VI-based assumption \cite{facchinei-pang-2009, kannan-shanbhag-2012, lei-shanbhag-2022, xiao-shanbhag-2025-gr, yousefian-nedic-shanbhag-2016}. To the best of our knowledge, our work is the first one to investigate the gradient-type schemes under the potentiality condition. See Figure \ref{gr-br-overview} for an illustration. In this work, we first consider the stochastic nonconvex smooth games and show our RSG scheme achieves the optimal $\mathcal{O}(N^{2}\epsilon^{-4})$ sample complexities to reach a point whose expected residual has norm at most $\epsilon$, where $N$ is the number of players. The merit of our RSG scheme is that it does not rely on the stringent growth and sharpness conditions in \cite{xiao-shanbhag-2025-gr} when contending with nonconvexity.

    \begin{figure}[h!]
    \centering
    \begin{tikzpicture}[>=latex, font=\small]

    \node[draw, ellipse, minimum width=2.8cm, minimum height=1.4cm, align=center] (BR) at (-2.9, 2.7) {\textbf{BR}};
    \node[draw, ellipse, minimum width=2.8cm, minimum height=1.4cm, align=center] (GR) at (2.9, 2.7) {\textbf{GR}};

    \node[draw, ellipse, minimum width=3.4cm, minimum height=1.4cm, align=center] (CON) at (-5.8, 0.0) {\textbf{Contraction}};
    \node[draw, ellipse, minimum width=3.4cm, minimum height=1.4cm, align=center] (POT) at (0.0, 0.0) {\textbf{Potentiality}};
    \node[draw, ellipse, minimum width=3.4cm, minimum height=1.4cm, align=center] (VI)  at (5.8, 0.0) {\textbf{VI Assump.}};

    \draw[thick] (BR) -- (CON)
        node[pos=0.7, above] {\scriptsize \cite{cui-pang-2021, lei-shanbhag-2022, lei-shanbhag-pang-sen-2020, pang-razaviyayn-2016, razaviyayn-2014, xiao-shanbhag-2026-br}};
    \draw[thick] (BR) -- (POT)
        node[pos=0.7, above] {\scriptsize \cite{cui-shanbhag-2023, cui-pang-2021, lei-shanbhag-2020, pang-razaviyayn-2016, razaviyayn-2014, xiao-shanbhag-2026-br}};
    \draw[thick] (GR) -- (VI)
        node[pos=0.7, above] {\scriptsize \cite{facchinei-pang-2009, kannan-shanbhag-2012, lei-shanbhag-2022, xiao-shanbhag-2025-gr, yousefian-nedic-shanbhag-2016}};

    \draw[dashed,thick] (GR) -- (POT)
        node[pos=0.6, above, fill=white, inner sep=1pt] {\scriptsize \textbf{This Work}};

    \end{tikzpicture}
    \caption{An overview of common assumptions for BR and GR schemes.}
    \label{gr-br-overview}
    \end{figure}
    
    \item\textbf{Nonconvex nonsmooth extension and CNE computation.} In this paper, we also consider stochastic potential games with Lipschitz continuous objectives and develop a randomized smoothed RSG (RS-RSG) scheme. We show that RS-RSG scheme converges asymptotically to an equilibrium of the smoothed games with sample complexity $\mathcal{O}(L^{4}_{\max}n^{3/2}_{\max}N^{3}\eta^{-1}\epsilon^{-4})$. The set of Clarke–Nash equilibria (CNE) can be captured by the solution set of a generalized variational inequality (GVI). We show that, under Lipschitz continuity of the Clarke subdifferentials, the expected residual of GVI evaluated at the smoothed equilibrium is $\mathcal{O}(\eta^{2})$, where $\eta > 0$ denotes the randomized smoothing parameter. To the best of our knowledge, our work is the first to establish convergence guarantees for a CNE-seeking algorithm in the nonconvex nonsmooth setting.
    
    \item\textbf{Biased RSG/RS-RSG schemes and stochastic hierarchical games.} The unbiasedness assumption is crucial in the convergence analysis of stochastic gradient schemes. However, ensuring unbiased gradient estimation can be particularly difficult in some contemporary stochastic optimization problems such as distributionally robust optimization (DRO) \cite{kuhn-shafiee-wiesemann-2025} and stochastic bilevel optimization \cite{chen-sun-yin-2021, hong-wai-wang-yang-2023}. Motivated by this, we study the biased RSG scheme and show that it drives the expected residual to zero if the bias sequence is summable, with iteration complexity $\mathcal{O}(N\epsilon^{-2})$ and sample complexity $\mathcal{O}(N^{4}\epsilon^{-4})$. We also evaluate the efficiency of the biased RS-RSG scheme on a class of stochastic hierarchical games for which the exact lower-level solution information is unavailable in finite time.
\end{itemize}

\begin{table}[htbp]
\centering
\normalsize
\setlength{\tabcolsep}{8pt}
\renewcommand{\arraystretch}{2.0}
\begin{tabular}{c ccccc}
\toprule
\toprule
 & \textbf{RSG} & \textbf{b-RSG} & \textbf{RS-RSG} & \textbf{b-RS-RSG} \\
\midrule
\midrule
\shortstack{\textbf{iteration}\\\textbf{complexity}} & $\mathcal{O}(\epsilon^{-2})$ & $\mathcal{O}(N\epsilon^{-2})$ & $\mathcal{O}(L^{3}_{\max} n_{\max} N \eta^{-1} \epsilon^{-2})$ & $\mathcal{O}(L^{3}_{\max}n^{7/2}_{\max}N^{2}\eta^{-4}\epsilon^{-2})$ \\
\shortstack{\textbf{sample}\\\textbf{complexity}} & $\mathcal{O}(N^{2}\epsilon^{-4})$ & $\mathcal{O}(N^{4}\epsilon^{-4})$ & $\mathcal{O}(L^{4}_{\max}n^{3/2}_{\max}N^{3}\eta^{-1}\epsilon^{-4})$ & $\mathcal{O}(L^{4}_{\max} n^{13/2}_{\max} N^{5}\eta^{-7}\epsilon^{-4})$ \\
\bottomrule
\bottomrule
\end{tabular}
\caption{A summary of complexities.}
\label{rsg-errors}
\end{table}

We summarize the complexities in Table \ref{rsg-errors}. The remainder of this paper is organized as follows. In Section \ref{Sec-2}, we provide the mathematical preliminaries. The main RSG algorithm as well as its biased variant for the smooth case are described in Section \ref{Sec-3}. In Section \ref{Sec-4}, we discuss the RS-RSG scheme and the equilibria approximation result. In Section \ref{Sec-5}, we demonstrate the biased RS-RSG scheme on a challenging class of stochastic potential hierarchical games. Preliminary numerics and final conclusions are provided in Sections \ref{Sec-6} and \ref{Sec-7}, respectively.

\:

\noindent\textbf{Notations.} We denote the Euclidean projection of $x$ onto set $X$ by $\Pi_{X}[x]$. The symbol $\mathbb{B}(x; \delta)$ denotes the closed ball of radius $\delta>0$ centered at $x\in \mathbb{R}^{n}$. The interior and convex hull of set $X$ are denoted by $\mathrm{int}(X)$ and $\mathrm{conv}(X)$, respectively. Let $A$ and $B$ be two nonempty subsets of $\mathbb{R}^{n}$. The distance from a vector $x\in \mathbb{R}^{n}$ to $A$ is defined as $\mathrm{dist}\: (x, A) \triangleq \inf_{y\in A} \| y-x \|$. The one-sided deviation of $A$ from $B$ is defined as $\mathbb{D}(A, B) \triangleq \sup_{x\in A} \mathrm{dist}\: (x, B)$. Given a Lipschitz continuous function $f: \mathbb{R}^{n} \to \mathbb{R}$, let $\partial^{C}f(x)$ denote the Clarke subdifferential of $f$ at $x\in \mathrm{dom}\:(f)$. We denote the Fr{\'e}chet normal cone to $X$ at $x$ by $\mathcal{N}_{X}(x)$. We refer the reader to the monographs \cite{cui-pang-2021, rockafellar-wets-1998} for their definitions and properties.

\section{Preliminaries}\label{Sec-2}

\subsection{Randomized Smoothing}\label{Sec-2.1}

Let $f: X\subseteq \mathbb{R}^{n} \to \mathbb{R}$ be a locally Lipschitz continuous function. Recall that the randomized smoothing \cite{steklov-1907} of $f(x)$ with $\eta > 0$ is defined as 
\begin{equation*}
    f^{\eta}(x) \triangleq \mathbb{E}_{u\in \mathbb{B}}[f(x+\eta u)],
\end{equation*}
where $\mathbb{B}$ denotes the unit ball and $u$ is uniformly distributed over $\mathbb{B}$. We recall some basic properties of randomized smoothing in the following lemma, where we denote the surface of $\mathbb{B}$ by $\mathbb{S}$ and the Minkowski sum of $X$ and $\eta \mathbb{B}$ by $X_{\eta} \triangleq X+\eta \mathbb{B}$.

\begin{lemma}[\mbox{\cite[Lemma 2.4]{marrinan-shanbhag-yousefian-2026}}]\label{RS-property-grad}
    Consider $f: X\subseteq \mathbb{R}^{n}\to \mathbb{R}$ and its randomized smoothing $f^{\eta}$, where $\eta>0$. Then the following hold.\\
    \emph{(i)} $f^{\eta}$ is $C^{1}$ over $X$ and $\nabla_{x} f^{\eta}(x) = (\tfrac{n}{2\eta})\mathbb{E}_{v\in \eta\mathbb{S}}[ (f(x+v)-f(x-v))\tfrac{v}{\|v\|}]$ holds for all $x\in X$.

    \noindent Suppose $f$ is $L_{0}$-Lipschitz continuous on $X_{\eta}$. For any $x, y\in X$, \emph{(ii)-(v)} hold.

    \noindent\emph{(ii)} $|f^{\eta}(x)-f^{\eta}(y)|\leq L_{0}\|x-y\|$. \emph{(iii)} $|f^{\eta}(x)-f(x)|\leq L_{0}\eta$.
    
    \noindent\emph{(iv)} $\|\nabla_{x}f^{\eta}(x)-\nabla_{x}f^{\eta}(y)\|\leq \tfrac{L_{0}\sqrt{n}}{\eta}\|x-y\|$.

    \noindent\emph{(v)} $f$ is $L_{1}$-smooth on $X_{\eta} \implies  \forall x \in X_{\eta}$, $\| \nabla_{x}f^{\eta}(x) - \nabla_{x}f(x) \|\leq \eta L_{1} n$.

    \noindent\emph{(vi)} For any $x\in X$, we have $\mathbb{E}_{\mathbf{v}\in \eta\mathbb{S}}[\|g(x, \mathbf{v})\|^{2}]\leq 16\sqrt{2\pi}L^{2}_{0}n$ where $g(x, \mathbf{v}) \triangleq \left( \frac{n(f(x+\mathbf{v})-f(x-\mathbf{v}))\mathbf{v}}{2\eta\|\mathbf{v}\|} \right)$.
\end{lemma}

The $\delta$-Clarke subdifferential \cite{goldstein-1977} of a Lipschitz continuous function $f: \mathbb{R}^{n}\to \mathbb{R}$ is defined as
\begin{equation}\label{delta-Clarke-subdifferential}
    \partial^{C}_{\delta} f(x) \triangleq \mbox{conv} \left\{ \zeta \;\vert\; \zeta\in \partial^{C} f(y),\; \|x-y\|\leq \delta \right\} = \mbox{conv} \left( \bigcup_{ \, \|y-x\| \le \delta} \partial^C f(y) \right).
\end{equation}
The following lemma is useful for our approximation analysis, showing the inclusion relation between the smoothed gradient and the $\delta$-Clarke subdifferential.

\begin{lemma}[\mbox{\cite[Proposition 2.5]{marrinan-shanbhag-yousefian-2026}}]\label{RS-property-inclusion}
    Consider the optimization problem $\min_{x\in X} f(x)$ where $f$ is locally Lipschitz continuous and $X \subseteq \mathbb{R}^{n}$ is closed, convex, and bounded. \emph{(i)} For any $\eta>0$ and any $x\in \mathbb{R}^{n}$, $\nabla f^{\eta}(x)\in \partial^{C}_{\eta} f(x)$. \emph{(ii)} For any $\eta>0$ and any $x\in X$, if $0\in \nabla f^{\eta}(x) + \mathcal{N}_{X}(x)$ holds, we have $0\in \partial^{C}_{\eta} f(x) + \mathcal{N}_{X}(x)$.
\end{lemma}

\subsection{Clarke-Nash Equilibria and Generalized Variational Inequalities}\label{Sec-2.2}

Inspired by Clarke stationarity \cite[Definition 6.1.4]{cui-pang-2021}, Clarke-Nash equilibrium (CNE) were introduced. We recall the definition of CNE below.

\begin{definition}[\mbox{\cite[Definition 11.1.1]{cui-pang-2021}}]\label{CNE-def}
    Consider the $N$-player game \eqref{games} where for any $i \in [N]$, the $i$th player-specific function $f_{i}(\bullet, x_{-i})$ is Lipschitz continuous for given $x_{-i}$. We say $x^{*} = (x^{*}_{i})_{i=1}^{N}$ is a Clarke-Nash equilibrium \abbr{CNE} if for any $i\in [N]$, $x^{*}_{i}$ is a Clarke stationary point, given $x^{*}_{-i}$, i.e.,
    \begin{equation}
        0\in \partial^{C}_{x_{i}} f_{i}(x^{*}_{i}, x^{*}_{-i}) + \mathcal{N}_{X_{i}}(x^{*}_{i}) \tag{CNE},
    \end{equation}
    where $\partial^{C}_{x_{i}} f_{i}(x^{*}_{i}, x^{*}_{-i})$ is the Clarke subdifferential of $f_{i}(\bullet, x^{*}_{-i})$ at $x^{*}_{i}$.
\end{definition}

By the Clarke regularity \cite[Definition 4.3.4]{cui-pang-2021} of convex functions, the CNE coincides with the standard NE in convex games. If for any $i\in [N]$, $f_{i}(\bullet, x_{-i})$ is continuously differentiable on an open set $\mathcal{O}_{i}$ such that $X_{i}\subseteq \mathcal{O}_{i}$ for given $x_{-i}$, the CNE reduces to
\begin{equation*}
    \nabla_{x_{i}} f_{i}(x^{\ast}_{i}, x^{\ast}_{-i})^{\top}(x_{i}-x^{\ast}_{i})\geq 0,~ \forall x_{i}\in X_{i},
\end{equation*}
which can be captured by the solution set of $\mathrm{SVI}\:(X, F)$, where $X \triangleq \prod_{i=1}^{N}X_{i}$ and $F(x) \triangleq (\nabla_{x_{i}}f_{i}(x))_{i=1}^{N}$. The equilibria existence under smoothness from the SVI perspective is established in \cite{ravat-shanbhag-2011}. 

Next we consider a generalization of variational inequalities, namely the generalized variational inequality \cite{fang-peterson-1982}, denoted by $\mathrm{GVI}\:(X, \mathbf{F})$, where $\mathbf{F}: X \rightrightarrows \mathbb{R}^n$ is a set-valued map. The problem consists of finding a pair $(x^{*}, y^{*})$ such that $x^{*}\in X$, $y^{*}\in \mathbf{F}(x^{*})$ and
\begin{equation*}
    (y^{*})^{\top}(x-x^{*})\geq 0, ~\forall x\in X.
\end{equation*}
By Definition \ref{CNE-def}, we know that the solution set of $\mathrm{GVI}\:(X,\mathbf{F})$ characterizes CNE. Indeed, we may take $\mathbf{F}(x)\triangleq (\partial^{C}_{x_i} f_i(x_i,x_{-i}))_{i=1}^{N}$. The following proposition establishes the CNE existence from the perspective of $\mathrm{GVI}\:(X,\mathbf{F})$, under the upper semicontinuity \cite[Definition 1.5.1]{cui-pang-2021} assumption on $\mathbf{F}$.

\begin{proposition}\label{QNE-CNE-existence}
    Consider the $N$-player game \eqref{games}. Let each $X_{i}$ be a compact convex set. If each function $f_{i}(\bullet, x_{-i})$ is Lipschitz continuous for $i\in [N]$, then a CNE exists.
\end{proposition}
\begin{proof}
    The proof proceeds similarly to \cite[Proposition 11.2.2]{cui-pang-2021} by considering the set-valued map $\mathbf{F}(x)\triangleq (\partial^{C}_{x_i} f_i(x_i,x_{-i}))_{i=1}^N$. The required upper semicontinuity in \cite[Proposition 11.2.2]{cui-pang-2021} of Clarke subdifferentials follows from \cite[Proposition 4.3.1]{cui-pang-2021} directly.
\end{proof}

\section{Stochastic Nonconvex Smooth Potential Games}\label{Sec-3}

In this section, we consider the stochastic nonconvex smooth game:
\begin{equation*}\label{smooth-games}
    \min_{x_i \in X_i} ~ f_i(x_i, x_{-i}) \triangleq \mathbb{E} \Big[ \tilde{f}_i(x_i, x_{-i}, \xi) \Big], ~ \forall i\in [N], \tag{$\mathrm{G^{S}}$}
\end{equation*}
where each $\Tilde{f}_{i}(\bullet, x_{-i}, \xi)$ is smooth for any $x_{-i}$ and $\xi$.

Recall that an $N$-player game is said to be a potential game \cite{monderer-shapley-1996} if there exists a potential function $P: X\triangleq \Pi_{i=1}^{N} X_{i} \to \mathbb{R}$ such that for any $x^{1}_{i}, x^{2}_{i}\in X_{i}$ and any $x_{-i}\in X_{-i}$, we have
\begin{equation}\label{potential-function}
    P(x^{1}_{i}, x_{-i}) - P(x^{2}_{i}, x_{-i}) = f_{i}(x^{1}_{i}, x_{-i}) - f_{i}(x^{2}_{i}, x_{-i}),~\forall i\in [N].
\end{equation}
The following proposition shows that under smoothness, a potential game can be equivalently viewed as an optimization problem. In fact, our claim is equivalent to \cite[Lemma 4.4]{monderer-shapley-1996} and \cite[Definition 3.1.1]{sandholm-2010}, which are often used as an alternative definition of potential games.

\begin{proposition}\label{potentiality-integrability}
    Suppose that the smooth game \eqref{smooth-games} admits a potential function $P$. Then $P$ is continuously differentiable and we have that $\nabla_{x} P(x) = (\nabla_{x_{i}} f_{i}(x))_{i = 1}^{N}$.
\end{proposition}

\begin{proof}
    For any $i\in [N]$ and given $x_{-i}\in X_{-i}$, we define the function $r_{i}(\bullet; x_{-i}) \triangleq P(\bullet, x_{-i}) - f_{i}(\bullet, x_{-i})$, treating $x_{-i}$ as a parameter. By the potential game definition, we have that $r_{i}(x^{1}_{i}; x_{-i}) \equiv r_{i}(x^{2}_{i}; x_{-i})$ for any $x^{1}_{i}, x^{2}_{i}\in X_{i}$, it follows that $r_{i}(\bullet; x_{-i})$ is a constant on $X_{i}$. Then we have that (i) $r_{i}(\bullet; x_{-i})$ is continuously differentiable; and (ii) $\nabla_{x_{i}} r_{i}(x_{i}; x_{-i}) = 0$. 
    
   By (i), we know that $P(\bullet, x_{-i})$ is continuously differentiable. Indeed, since both $r_i(\bullet; x_{-i})$ and $f_i(\bullet, x_{-i})$ are continuously differentiable, their sum must be also continuously differentiable. It follows that $\nabla_{x_{i}} P(x_{i}, x_{-i})$ exists hence $\nabla_{x} P(x) = (\nabla_{x_{i}} P(x_{i}, x_{-i}))_{i=1}^{N}$ also exists, implying that $P$ is continuously diffferentiable. By (ii), we may derive that $0 = \nabla_{x_{i}} r_{i}(x_{i}; x_{-i}) = \nabla_{x_{i}} P(x_{i}, x_{-i}) - \nabla_{x_{i}} f_{i}(x_{i}, x_{-i})$ for any $i\in [N]$, implying that $\nabla_{x} P(x) = (\nabla_{x_{i}} P(x))_{i = 1}^{N} = (\nabla_{x_{i}} f_{i}(x))_{i = 1}^{N}$, as desired.
\end{proof}

\begin{remark}\em
    The above observation forms the key idea of this paper.  In contrast to existing literature, where potentiality is primarily leveraged for the developments and analyses of best-response schemes, we exploit a different implication of this property. Specifically, under smoothness and potentiality, a stochastic synchronous gradient-response scheme admits an equivalent optimization-theoretic interpretation as a stochastic gradient method for the associated constrained stochastic nonconvex smooth optimization. This intrinsic connection provides the possibility of establishing convergence and attaining optimal complexity guarantees in the game-theoretic setting.
\end{remark}

\subsection{Randomized Stochastic Gradient Scheme}\label{Sec-3.1}

We impose Assumption $\mathrm{A}$ throughout this subsection.

\vspace{5pt}

\begin{custombox}
\noindent\textbf{Assumption A.}\\
(A1) (Potentiality) The smooth game \eqref{smooth-games} admits an $L$-smooth potential function $P$ attaining its maximum $P_{\max}$ and minimum $P_{\min}$ over $X$.\\
(A2) (Unbiasedness) For any $x^{k}$, we have $\mathbb{E}[\nabla_{x_{i}}\Tilde{f}_{i}(x^{k}_{i}, x^{k}_{-i}, \xi) \mid x^{k}] = \nabla_{x_{i}}f_{i}(x^{k}_{i}, x^{k}_{-i})$ for any $i\in [N]$.\\
(A3) (Bounded second moment) For any $x^{k}$, we have $\mathbb{E}[ \|\nabla_{x_{i}}\Tilde{f}_{i}(x^{k}_{i}, x^{k}_{-i}, \xi)-\nabla_{x_{i}}f_{i}(x^{k}_{i}, x^{k}_{-i})\|^{2} \mid x^{k}]\leq \sigma^{2}$ for some $\sigma>0$ and any $i\in [N]$.
\end{custombox}

\vspace{5pt}

Inspired by Ghadimi et al. \cite{ghadimi-lan-zhang-2016}, we propose the  RSG scheme for solving the stochastic nonconvex smooth game \eqref{smooth-games}. While both \cite{ghadimi-lan-zhang-2016} and our work employ mini-batch sampling and randomized output settings, the i.i.d. samples $\{\xi^{k}_{i, l}\}_{l=1}^{S_{k}}$ at iteration $k$ here may be different for different players since we are considering an $N$-player game.

\begin{algorithm}
\caption{~RSG scheme} \label{RSG}
\begin{algorithmic}[1]
\State \textbf{Input}: starting point $x^{0}\in X$, iteration limit $T$, stepsizes $\{\gamma_{k}\}_{k\geq 0}$, batch sizes $\{S_{k}\}_{k\geq 0}$, and probability mass function $P_{R}$ supported on $\{ 1, \dots, T \}$.
\State Let $R$ be a random variable with probability mass function $P_{R}$.
\For{$k = 0, 1, \dots, R-1$}
    \For{$i = 1, \dots, N$}
        \State (i) Pick $S_{k}$ i.i.d. realizations $\{\xi^{k}_{i, l}\}_{l=1}^{S_{k}}$ of the random variable $\xi$;
        \State (ii) Set the minibatch gradient as $\Tilde{g}_{i}(x^{k}, \xi_{i}^{k}) \triangleq \frac{1}{S_{k}}\sum_{l=1}^{S_{k}} \nabla_{x_{i}}\Tilde{f}_{i}(x^{k}_{i}, x^{k}_{-i}, \xi^{k}_{i, l})$;
        \State (iii) Update $x^{k+1}_{i} = \Pi_{X_{i}} [ x^{k}_{i} - \gamma_{k}\Tilde{g}_{i}(x^{k}, \xi_{i}^{k}) ]$;
    \EndFor
\EndFor
\State \textbf{Output}: $x^{R}$ as final estimate.
\end{algorithmic}
\end{algorithm}

Before examining the RSG scheme, we first turn to the convergence measure. For the variational inequality $\mathrm{VI}\:(X, F)$ where $X \triangleq \prod_{i=1}^{N}X_{i}$ and $F(x) \triangleq (\nabla_{x_{i}}f_{i}(x))_{i=1}^{N}$, we consider the residual
\begin{equation*}
    G_{\gamma}(x) \triangleq \tfrac{1}{\gamma}(x-\Pi_{X}[x-\gamma F(x)]),~ \gamma>0.
\end{equation*}
It can be seen that $x^{\ast}$ solves $\mathrm{VI}\:(X,F)$ if and only if $G_{\gamma}(x^{*}) = 0$. In what follows, we consider two residuals at iteration $k$, defined as
\begin{align*}
    G_{\gamma_{k}}(x^{k}) &\triangleq \tfrac{1}{\gamma_{k}} (x^{k}-\Pi_{X}[x^{k}-\gamma_{k} F(x^{k})]), \\
    \Tilde{G}_{\gamma_{k}}(x^{k}, \xi^{k}) &\triangleq \tfrac{1}{\gamma_{k}} (x^{k}-\Pi_{X}[x^{k}-\gamma_{k} \Tilde{F}(x^{k}, \xi^{k})]).
\end{align*}
Now we are ready for the convergence analysis of the RSG scheme. It can seen that the line $7$ of RSG scheme can be compactly written as $x^{k+1} = \Pi_{X} [ x^{k} - \gamma_{k} \Tilde{F}(x^{k}, \xi^{k})]$, where $\Tilde{F}(x^{k}, \xi^{k}) \triangleq (\Tilde{g}_{i}(x^{k}, \xi_{i}^{k}))_{i=1}^{N}$. The proof of the following theorem is similar to that of \cite{ghadimi-lan-zhang-2016} and we defer it to Appendix \ref{proof-smooth-RSG}.

\begin{theorem}[Convergence of RSG]\label{RSG-convergence}
    Suppose that the stepsizes $\{\gamma_{k}\}_{k\geq 0}$ in the RSG scheme are chosen such that $0<\gamma_{k} \leq 1/L$ with $\gamma_{k} < 1/L$ for at least one $k$, and the probability mass function $P_{R}$ are chosen such that for any $k = 1, \dots, T$,
    \begin{equation*}
        P_{R}(k) \triangleq \emph{Prob}\{R = k\} = \frac{\gamma_{k}-L(\gamma_{k})^{2}}{\sum_{k=1}^{T}(\gamma_{k}-L(\gamma_{k})^{2})}.
    \end{equation*}
    Suppose that Assumption $\mathrm{A}$ holds. Then we have:\\
    \emph{(i)} For any $T\geq 1$, we have
    \begin{equation}\label{RSGR-thm-1}
        \mathbb{E}[\| \Tilde{G}_{\gamma_{R}}(x^{R}, \xi^{R}) \|^{2}] \leq \frac{LD^{2}+(\sigma^{2}N)\sum_{k=1}^{T}(\gamma_{k}/S_{k})}{\sum_{k=1}^{T}(\gamma_{k}-L(\gamma_{k})^{2})},
    \end{equation}
    where the expectation is taken with respect to $R$ and the history $\xi^{[k]} \triangleq (\xi^{0}, \dots, \xi^{k})$, and $D$ is given by
    \begin{equation}\label{D-definition}
        D \triangleq \left[ \frac{P_{\emph{max}} - P_{\emph{min}}}{L} \right]^{1/2}.
    \end{equation}
    \emph{(ii)} Suppose that we choose the stepsizes $\gamma_{k} = 1/(2L)$ and $P_{R}(k) = 1/T$ for all $k = 1, \dots, T$. Given a sufficiently large but fixed total number of calls $M$ to the stochastic first-order oracles \emph{(}$\mathcal{SFO}$\emph{)}. If the number of samples $S_{k}$ at iteration $k$ of the RSG algorithm is fixed as follows:
    \begin{equation}\label{S-selection}
        S_{k} = S \triangleq \left\lceil \frac{\sigma\sqrt{6M}}{4LD} \right\rceil,
    \end{equation}
    then $T\leq \lfloor M/(SN) \rfloor$ and we have that
    \begin{equation}\label{RSG-final-bound}
        \mathbb{E}[\| G_{\gamma_{R}}(x^{R}) \|^{2}] \leq \frac{16L^{2}D^{2}N}{M} + \frac{8\sqrt{6} LDN\sigma}{\sqrt{M}}.
    \end{equation}
\end{theorem}
\begin{proof}
    See Appendix \ref{proof-smooth-RSG}.
\end{proof}

\begin{remark}\em
    Several points deserve emphasis. First, it follows from \eqref{RSG-final-bound} that an $x_R$ satisfying $\mathbb{E}[\|G_{\gamma_R}(x^R)\|]\le \epsilon$ can be obtained within $\mathcal{O}(\epsilon^{-2})$ iterations and $\mathcal{O}(N^2\epsilon^{-4})$ sample evaluations. Unlike \cite{ghadimi-lan-zhang-2016}, our complexity bound must account for the number of players $N$, leading to a modified sample complexity estimate. Second, although our analysis relies on a constant batch size, adaptive batch-size strategies can be incorporated in implementation, see \cite{ghadimi-lan-zhang-2016} for further discussion. Third, RSG requires an estimate of $L$ to determine the quantities $P_R(k)$, $\gamma_k$, and $S_k$. This requirement may introduce additional difficulties in stochastic regime.
\end{remark}

One exciting result is that, when the potential function $P$ is smooth and pseudoconvex on $X$, i.e., $\nabla P(x)^{\top}(y-x) \geq 0 \implies P(y) \geq P(x)$ for any $x, y \in X$, a QNE is indeed an NE, and thus RSG converges to an NE despite the nonconvexity of the game. We formalize it in the following theorem and its proof follows from \cite[Proposition 3]{xiao-shanbhag-2025-gr} immediately.

\begin{corollary}
    Consider the smooth game \eqref{smooth-games} with a potential function $P$ that is pseudoconvex and $C^{1}$ on an open set $\mathcal{O}\supseteq X$. Then the RSG scheme converges to an NE in expectation.
\end{corollary}

\subsection{Biased Randomized Stochastic Gradient Scheme}\label{Sec-3.2}

The results we have established in subsection \ref{Sec-3.1} are based on crucial unbiased assumption $\mathrm{(A2)}$. However, in some classes of contemporary stochastic optimization problems, such as distributionally robust optimization \cite{kuhn-shafiee-wiesemann-2025} and stochastic bilevel optimization \cite{chen-sun-yin-2021, hong-wai-wang-yang-2023}, obtaining unbiased gradient estimators remains challenging. We will illustrate this fact in Section \ref{Sec-5} through a class of stochastic hierarchical games. Motivated by these challenges, we consider a biased extension of RSG. We impose Assumption $\mathrm{B}$ throughout this subsection.

\vspace{5pt}

\begin{custombox}
\noindent\textbf{Assumption B.}\\
(B1) (Potentiality) The smooth game \eqref{smooth-games} admits an $L$-smooth potential function $P$ attaining its maximum $P_{\max}$ and minimum $P_{\min}$ over $X$. \\
(B2) (Biasedness) For any $x^{k}$, we have $\|\mathbb{E}[ \nabla_{x_{i}}\Tilde{f}_{i}(x^{k}_{i}, x^{k}_{-i}, \xi)-\nabla_{x_{i}}f_{i}(x^{k}_{i}, x^{k}_{-i}) \mid x^{k}]\|\leq \mu_{k}$ for some $\mu_{k}>0$ and any $i\in [N]$.\\
(B3) (Bounded second moment) For any $x^{k}$, we have $\mathbb{E}[ \|\nabla_{x_{i}}\Tilde{f}_{i}(x^{k}_{i}, x^{k}_{-i}, \xi)-\nabla_{x_{i}}f_{i}(x^{k}_{i}, x^{k}_{-i})\|^{2} \mid x^{k}]\leq \sigma_{k}^{2}$ for some $\sigma_{k}>0$ and any $i\in [N]$.
\end{custombox}

\vspace{5pt}

Now we are ready to establish the convergence of the biased RSG (b-RSG) scheme. Note that the stepsizes requirement is slightly different from that in Theorem \ref{RSG-convergence}.

\begin{theorem}[Convergence of biased RSG]\label{biased-RSG-convergence}
    Consider the RSG scheme with biased gradient estimators. Suppose that the stepsizes $\{\gamma_{k}\}_{k\geq 0}$ are chosen such that $0 < \gamma_{k} \leq 1/(2L)$ for all $k\geq 0$, and the probability mass function $P_{R}$ are chosen such that for any $k = 1, \dots, T$,
    \begin{equation*}
        P_{R}(k) \triangleq \emph{Prob}\{R = k\} = \frac{\gamma_{k}-L(\gamma_{k})^{2}}{\sum_{k=1}^{T}(\gamma_{k}-L(\gamma_{k})^{2})}.
    \end{equation*}
    Suppose that Assumption $\mathrm{B}$ holds. Then we have:\\
    \emph{(i)} For any $T\geq 1$ and any $c > 0$, we have
    \begin{equation*}
        \mathbb{E}[\| \Tilde{G}_{\gamma_{R}}(x^{R}, \xi^{R}) \|^{2}] \leq \frac{2LD^{2}+(2+c^{2})N\sum_{k=1}^{T}\gamma_{k}((\sigma^{2}_{k}+(S_{k}-1)\mu^{2}_{k})/S_{k})}{2\sum_{k=1}^{T}(\gamma_{k}-L(\gamma_{k})^{2})} + \frac{1}{c^{2}}\mathbb{E}[\| G_{\gamma_{R}}(x^{R}) \|^{2}],
    \end{equation*}
    where $D$ is defined in \eqref{D-definition}.\\
    \emph{(ii)} Suppose that we choose the stepsizes $\gamma_{k} = 1/(2L)$ and $P_{R}(k) = 1/T$ for all $k = 1, \dots, T$. Assume that $\sigma_{k} \leq \sigma$ for some $\sigma>0$ for all $k$. Given a sufficiently large but fixed total number of calls $M$ to the stochastic first-order oracles \emph{(}$\mathcal{SFO}$\emph{)}, if the number of samples $S_{k}$ at iteration $k$ is defined in \eqref{S-selection}, then we have
    \begin{equation}\label{biased-RSG-final-bound-0}
        \mathbb{E}[\| G_{\gamma_{R}}(x^{R}) \|^{2}] \leq \frac{16L^{2}D^{2}N}{(1-2/c^{2})M} + \frac{\tfrac{4}{3}\sqrt{6}(9+2c^{2})L^{2}D^{2}N\sigma + \sqrt{6}(3+c^{2})N^{2}\sigma\sum_{k=1}^{T}\mu^{2}_{k}}{(1-2/c^{2})LD\sqrt{M}}
    \end{equation}
    for any $c>\sqrt{2}$. In particular, if we take $c = \sqrt{3}$, we have
    \begin{equation}\label{biased-RSG-final-bound}
        \mathbb{E}[\| G_{\gamma_{R}}(x^{R}) \|^{2}] \leq \frac{48L^{2}D^{2}N}{M} + \frac{60\sqrt{6}L^{2}D^{2}N\sigma + 18\sqrt{6}N^{2}\sigma\sum_{k=1}^{T}\mu^{2}_{k}}{LD\sqrt{M}}.
    \end{equation}
\end{theorem}
\begin{proof}
    (i) Let $\delta^{k} = \Tilde{F}(x^{k}, \xi^{k}) - F(x^{k})$ for any $k\geq 0$ with $i$th component $\delta^{k}_{i} = \Tilde{g}_{i}(x^{k}, \xi_{i}^{k}) - \nabla_{x_{i}}f_{i}(x^{k})$. Similar as the proof of Theorem \ref{RSG-convergence}, we may arrive
    \begin{align}\label{biased-RSG-eqn1}
        \sum_{k=1}^{T} (\gamma_{k}-L(\gamma_{k})^{2}) \| \Tilde{G}_{\gamma_{k}}(x^{k}, \xi^{k}) \|^{2}
        &\leq P_{\text{max}} - P_{\text{min}} + \sum_{k=1}^{T} \{ \gamma_{k} \langle \delta^{k}, G_{\gamma_{k}}(x^{k}) \rangle + \gamma_{k} \|\delta^{k}\|^{2} \} \notag \\
        &= P_{\text{max}} - P_{\text{min}} + \sum_{k=1}^{T} \left\{ \gamma_{k} \left\langle c \delta^{k}, \tfrac{1}{c} G_{\gamma_{k}}(x^{k}) \right\rangle + \gamma_{k} \|\delta^{k}\|^{2} \right\} \notag \\
        &\leq P_{\text{max}} - P_{\text{min}} + \sum_{k=1}^{T} \left\{ \gamma_{k} \Big( \tfrac{1}{2} \| c \delta^{k} \|^{2} + \tfrac{1}{2} \big\| \tfrac{1}{c} G_{\gamma_{k}}(x^{k}) \big\|^{2} \Big) + \gamma_{k} \|\delta^{k}\|^{2} \right\} 
    \end{align}
    for any $c > 0$. Under Assumption $\mathrm{B}$, we may claim that
    \begin{equation}\label{biased-RSG-eqn1'}
        \mathbb{E}[\|\delta^{k}\|^{2}]\leq N (\sigma^{2}_{k}+(S_{k}-1)\mu^{2}_{k})/S_{k}.
    \end{equation}
    Indeed, we have
    \begin{equation*}
        \begin{aligned}
            \mathbb{E}[\|\delta^{k}_{i}\|^{2} \mid x^{k}] &= \mathbb{E} \left[ \left\| \frac{1}{S_{k}}\sum_{l=1}^{S_{k}}\nabla_{x_{i}}\Tilde{f}_{i}(x^{k}_{i}, x^{k}_{-i}, \xi^{k}_{i, l}) - \frac{1}{S_{k}}\sum_{l=1}^{S_{k}} \nabla_{x_{i}}f_{i}(x^{k}) \right\|^{2} \,\middle\vert\, x^{k} \right] \\
            &= \frac{1}{S^{2}_{k}} \mathbb{E} \left[ \left\| \sum_{l=1}^{S_{k}} (\nabla_{x_{i}}\Tilde{f}_{i}(x^{k}_{i}, x^{k}_{-i}, \xi^{k}_{i, l})-\nabla_{x_{i}}f_{i}(x^{k})) \right\|^{2} \,\middle\vert\, x^{k} \right] \\
            &\overset{(*)}{\leq} \frac{1}{S^{2}_{k}} \left[ S_{k}\sigma^{2}_{k} + S_{k}(S_{k}-1)\mu^{2}_{k} \right] = \frac{1}{S_{k}} \left[ \sigma^{2}_{k} + (S_{k}-1)\mu^{2}_{k} \right],
        \end{aligned}
    \end{equation*}
    where $(*)$ is due to assumptions (B2) and (B3), together with the fact that $\{\xi^{k}_{i, l}\}_{l=1}^{S_{k}}$ are i.i.d. realizations. It leads to
    \begin{equation}\label{biased-RSG-eqn2}
        \mathbb{E}[\|\delta^{k}\|^{2}] = \mathbb{E} \left[\, \mathbb{E} \left[ \|\delta^{k}\|^{2} \,\middle\vert\, x^{k} \right] \,\right] = \mathbb{E} \left[\, \mathbb{E} \left[ \sum_{i=1}^{N}\|\delta_{i}^{k}\|^{2} \,\middle\vert\, x^{k} \right] \,\right] \leq \frac{N}{S_{k}} \left[ \sigma^{2}_{k} + (S_{k}-1)\mu^{2}_{k} \right].
    \end{equation}
    By taking unconditional expectation on both sides of \eqref{biased-RSG-eqn1}, we arrive that
    \begin{equation*}
        \sum_{k=1}^{T} (\gamma_{k}-L(\gamma_{k})^{2}) \mathbb{E}[\| \Tilde{G}_{\gamma_{k}}(x^{k}, \xi^{k}) \|^{2}] \leq P_{\text{max}} - P_{\text{min}} + \sum_{k=1}^{T} \gamma_{k} \frac{2+c^{2}}{2} \mathbb{E}[\| \delta^{k} \|^{2}] + \sum_{k=1}^{T} \frac{\gamma_{k}}{2c^{2}} \mathbb{E}[\| G_{\gamma_{k}}(x^{k})\|^{2}].
    \end{equation*}
    Since $0<\gamma_{k}\leq 1/(2L)$ hence we have $\tfrac{1}{2}\gamma_{k} \leq \gamma_{k} - L(\gamma_{k})^{2}$, together with \eqref{biased-RSG-eqn2}, it follows that
    \begin{equation*}
        \begin{aligned}
            &\sum_{k=1}^{T} (\gamma_{k}-L(\gamma_{k})^{2}) \mathbb{E}[\| \Tilde{G}_{\gamma_{k}}(x^{k}, \xi^{k}) \|^{2}] \\
            &\leq P_{\text{max}} - P_{\text{min}} + \frac{2+c^{2}}{2}N\sum_{k=1}^{T}\gamma_{k}\frac{\sigma^{2}_{k}+(S_{k}-1)\mu^{2}_{k}}{S_{k}} + \frac{1}{c^{2}} \sum_{k=1}^{T} (\gamma_{k} - L(\gamma_{k})^{2}) \mathbb{E}[\| G_{\gamma_{k}}(x^{k})\|^{2}].
        \end{aligned}
    \end{equation*}
    By dividing both sides by $\sum_{k=1}^{T} (\gamma_{k}-L(\gamma_{k})^{2}) > 0$, we know from the definition of $x^{R}$ that
    \begin{equation*}
        \mathbb{E}[\| \Tilde{G}_{\gamma_{R}}(x^{R}, \xi^{R}) \|^{2}] \leq \frac{2LD^{2}+(2+c^{2})N\sum_{k=1}^{T}\gamma_{k}((\sigma^{2}_{k}+(S_{k}-1)\mu^{2}_{k})/S_{k})}{2\sum_{k=1}^{T}(\gamma_{k}-L(\gamma_{k})^{2})} + \frac{1}{c^{2}}\mathbb{E}[\| G_{\gamma_{R}}(x^{R}) \|^{2}],
    \end{equation*}
    as desired.

    \noindent (ii) By plugging $\gamma_{k} = 1/(2L)$ and $S_{k} = S$ into the upper bound of $\mathbb{E}[\| \Tilde{G}_{\gamma_{R}}(x^{R}, \xi^{R}) \|^{2}]$, we may arrive
    \begin{equation*}
        \mathbb{E}[\| \Tilde{G}_{\gamma_{R}}(x^{R}, \xi^{R}) \|^{2}] \leq \frac{4L^{2}D^{2}}{T} + \frac{(2+c^{2})N}{ST}\sum_{k=1}^{T}(\sigma^{2}_{k} + (S-1)\mu^{2}_{k}) + \frac{1}{c^{2}}\mathbb{E}[\| G_{\gamma_{R}}(x^{R}) \|^{2}].
    \end{equation*}
    Therefore, we have that
    \begin{equation*}
        \begin{aligned}
            &\mathbb{E}[\| G_{\gamma_{R}}(x^{R}) \|^{2}] \leq 2\mathbb{E}[\| \Tilde{G}_{\gamma_{R}}(x^{R}, \xi^{R}) \|^{2}] + 2\mathbb{E}[\| G_{\gamma_{R}}(x^{R}) - \Tilde{G}_{\gamma_{R}}(x^{R}, \xi^{R}) \|^{2}] \\
            &\leq 2 \left( \frac{4L^{2}D^{2}}{T} + \frac{(2+c^{2})N}{ST}\sum_{k=1}^{T}(\sigma^{2}_{k} + (S-1)\mu^{2}_{k}) + \frac{1}{c^{2}}\mathbb{E}[\| G_{\gamma_{R}}(x^{R}) \|^{2}] \right) + 2\mathbb{E}[\| \delta^{R} \|^{2}] \\
            &= 2 \left( \frac{4L^{2}D^{2}}{T} + \frac{(2+c^{2})N}{ST}\sum_{k=1}^{T}(\sigma^{2}_{k} + (S-1)\mu^{2}_{k}) + \frac{1}{c^{2}}\mathbb{E}[\| G_{\gamma_{R}}(x^{R}) \|^{2}] \right) + \frac{2\sum_{k=1}^{T}(\gamma_{k}-L(\gamma_{k})^{2})\mathbb{E}[\|\delta^{k}\|^{2}]}{\sum_{k=1}^{T}(\gamma_{k}-L(\gamma_{k})^{2})},
        \end{aligned}
    \end{equation*}
    where the last equality is from the definition of $\delta^{R}$. By rearranging the terms, together with $\gamma_{k} = 1/(2L)$ and $S_{k} = S$, it follows from \eqref{biased-RSG-eqn2} that
    \begin{align}\label{biased-RS-RSG-useful}
        ( 1 - 2/c^{2} ) \mathbb{E}[\| G_{\gamma_{R}}(x^{R}) \|^{2}] &\leq \frac{8L^{2}D^{2}}{T} + \frac{(4+2c^{2})N}{ST}\sum_{k=1}^{T}(\sigma^{2}_{k} + (S-1)\mu^{2}_{k}) + \frac{2N}{ST}\sum_{k=1}^{T}(\sigma^{2}_{k} + (S-1)\mu^{2}_{k}) \notag \\
        &= \frac{8L^{2}D^{2}}{T} + \frac{(6+2c^{2})N\sum_{k=1}^{T}(\sigma^{2}_{k}+(S-1)\mu^{2}_{k})}{ST}.
    \end{align}
    By the fact that $T = \lfloor M/(SN) \rfloor \geq M/(2SN)$ and $\sigma_{k}\leq \sigma$, together with \eqref{S-selection}, we may establish that
    \begin{align}\label{biased-RSG-eqn3}
        & ( 1 - 2/c^{2} ) \mathbb{E}[\| G_{\gamma_{R}}(x^{R}) \|^{2}] \notag \\
        &\leq \frac{16L^{2}D^{2}SN}{M} + \frac{(12+4c^{2})N^{2}S\sum_{k=1}^{T}\mu^{2}_{k}}{M} + \frac{(12+4c^{2})N^{2}\sum_{k=1}^{T}(\sigma^{2}-\mu^{2}_{k})}{M} \notag \\
        &\leq \frac{16L^{2}D^{2}N}{M} \left( 1 + \frac{\sigma\sqrt{6M}}{4LD} \right) + \frac{(12+4c^{2})N^{2}\sum_{k=1}^{T}\mu^{2}_{k}}{M} \left( 1 + \frac{\sigma\sqrt{6M}}{4LD} \right) + \frac{(12+4c^{2})N^{2}\sum_{k=1}^{T}(\sigma^{2}-\mu^{2}_{k})}{M} \notag \\
        &= \frac{16L^{2}D^{2}N}{M} \left( 1 + \frac{\sigma\sqrt{6M}}{4LD} \right) + \frac{(12+4c^{2})N^{2}\sum_{k=1}^{T}\mu^{2}_{k}}{M} \left( \frac{\sigma\sqrt{6M}}{4LD} \right) + \frac{(12+4c^{2})N^{2}\sum_{k=1}^{T}\sigma^{2}}{M}.
    \end{align}
    Since $T\leq M/(SN)$, we have the upper bound
    \begin{equation}\label{biased-RSG-eqn4}
        \begin{aligned}
            \frac{(12+4c^{2})N^{2}\sum_{k=1}^{T}\sigma^{2}}{M} &\leq \frac{(12+4c^{2})N\sigma^{2}}{S} \leq \frac{\tfrac{4}{3}\sqrt{6}(6+2c^{2})LDN\sigma}{\sqrt{M}}.
        \end{aligned}
    \end{equation}
    Consequently, by plugging \eqref{biased-RSG-eqn4} into \eqref{biased-RSG-eqn3}, we arrive that
    \begin{equation*}
        \begin{aligned}
            (1-2/c^{2}) \mathbb{E}[\| G_{\gamma_{R}}(x^{R}) \|^{2}] \leq \frac{16L^{2}D^{2}N}{M} + \frac{\tfrac{4}{3}\sqrt{6}(9+2c^{2})L^{2}D^{2}N\sigma + \sqrt{6}(3+c^{2})N^{2}\sigma\sum_{k=1}^{T}\mu^{2}_{k}}{LD\sqrt{M}},
        \end{aligned}
    \end{equation*}
    implying that
    \begin{equation*}
        \mathbb{E}[\| G_{\gamma_{R}}(x^{R}) \|^{2}] \leq \frac{16L^{2}D^{2}N}{(1-2/c^{2})M} + \frac{\tfrac{4}{3}\sqrt{6}(9+2c^{2})L^{2}D^{2}N\sigma + \sqrt{6}(3+c^{2})N^{2}\sigma\sum_{k=1}^{T}\mu^{2}_{k}}{(1-2/c^{2})LD\sqrt{M}}.
    \end{equation*}
    When $c = \sqrt{3}$, we have
    \begin{equation*}
        \mathbb{E}[\| G_{\gamma_{R}}(x^{R}) \|^{2}] \leq \frac{48L^{2}D^{2}N}{M} + \frac{60\sqrt{6}L^{2}D^{2}N\sigma + 18\sqrt{6}N^{2}\sigma\sum_{k=1}^{T}\mu^{2}_{k}}{LD\sqrt{M}},
    \end{equation*}
    which completes the proof.
\end{proof}

\begin{remark}\em
    We know from \eqref{biased-RSG-final-bound-0} that if the bias sequence $\{ \mu_{k} \}_{k\geq 0}$ is square summable, the iteration and sample complexities of the b-RSG scheme are $\mathcal{O}(N\epsilon^{-2})$ and $\mathcal{O}(N^{4}\epsilon^{-4})$, respectively. It can be seen that the b-RSG scheme exhibits a stronger dependence on $N$ than the standard RSG. Note that when the problem has a bilevel structure and the biases come from inexact lower-level solutions, we may increase the number of lower-level iterations to make the bias sequence decreasing and square summable. We will discuss this in more detail in Section \ref{Sec-5}.
\end{remark}
 
\section{Stochastic Nonconvex Nonsmooth Potential Games}\label{Sec-4}

Building on the convergence results of RSG, we are ready to examine the randomized smoothed RSG (RS-RSG) scheme for stochastic potential nonconvex nonsmooth games. In this section, we consider the stochastic $N$-player Lipschitz continuous game
\begin{equation}\label{Lipschitz-continuous-games}
    \min_{x_{i}\in X_{i}} ~ f_{i}(x_{i}, x_{-i}) \triangleq \underbrace{\mathbb{E} \Big[ \Tilde{h}_{i}(x_{i}, \xi) \Big]}_{\triangleq \, h_{i}(x_{i})} \, + \, \underbrace{\mathbb{E} \Big[ \Tilde{m}_{i}(x_{i}, x_{-i}, \xi) \Big]}_{\triangleq \, m_{i}(x_{i}, x_{-i})}, ~ \forall i\in [N], \tag{$\mathrm{G^{LC}}$}
\end{equation}
where each $X_{i} \subseteq \mathbb{R}^{n_{i}}$ is convex and closed, each $\Tilde{h}_{i}(\bullet, \xi)$ is $L_{i}$-Lipschitz continuous over $X_{i}$ for any $\xi$, and each $\Tilde{m}_{i}(\bullet, x_{-i}, \xi)$ is continuously differentiable for any $x_{-i}$ and any $\xi$. We denote the randomized smoothed counterpart of \eqref{Lipschitz-continuous-games} by
\begin{equation}\label{RS-games}
    \min_{x_{i}\in X_{i}} ~ f^{\eta}_{i}(x_{i}, x_{-i}) \triangleq h^{\eta}_{i}(x_{i}) + m_{i}(x_{i}, x_{-i}), ~ \forall i\in [N], \tag{$\mathrm{G^{RS}}$}
\end{equation}
where $h^{\eta}_{i}(x_{i})$ is the randomized smoothing of $h_{i}(x_{i})$, i.e., $h^{\eta}_{i}(x_{i}) = \mathbb{E}_{u_{i}\in \mathbb{B}_{n_{i}}}[h_{i}(x_{i} + \eta u_{i})]$. Throughout this section, we always assume that \eqref{Lipschitz-continuous-games} admits a potential function $P(x)$.

\begin{remark}\label{xi-assumption}\em
    Alternatively, we may require each $\Tilde{h}_{i}(\bullet,\xi)$ to be $\Tilde{L}_{i}(\xi)$-Lipschitz continuous. Then under a suitable integrability condition on $\Tilde{L}_{i}(\xi)$, an analog of the above assumption can be obtained. For purpose of simplicity, we directly assume that each $\Tilde{h}_{i}(\bullet, \xi)$ is $L_{i}$-Lipschitz continuous for any $\xi \in \Xi$.
\end{remark}

The potentiality of \eqref{Lipschitz-continuous-games} implies that of \eqref{RS-games}, as captured by the following proposition.

\begin{proposition}\label{potentiality-implies-potentiality}
    Suppose that the stochastic Lipschitz continuous game \eqref{Lipschitz-continuous-games} is a potential game and $P$ denotes its potential function. Then its randomized smoothing counterpart \eqref{RS-games} is also a potential game with a smooth potential function $P^{\eta}$, defined as
    \begin{equation}\label{new-potential-function}
        P^{\eta}(x) \triangleq \underbrace{P(x) - \sum_{i=1}^{N} h_{i}(x_{i})}_{\triangleq P_{m}(x)} + \sum_{i=1}^{N} h^{\eta}_{i}(x_{i}).
    \end{equation}
\end{proposition}
\begin{proof}
    The proof follows from Proposition \ref{potentiality-integrability} and the potential function definition \eqref{potential-function} immediately. We omit it here.
\end{proof}

\begin{remark}\em
    It can be shown that $P_{m}(x) = P(x) - \sum_{i=1}^{N} h_{i}(x_{i})$ is a potential function of the smooth $N$-player game where the $i$th player solves $\min_{x_{i}\in X_{i}} m_{i}(x_{i}, x_{-i})$. Since each $m_{i}(\bullet, x_{-i})$ is smooth, we know from Proposition \ref{potentiality-integrability} that $P_{m}(x)$ is continuously differentiable.
\end{remark}

\subsection{Randomized Smoothed RSG Scheme}\label{Sec-4.1}

We impose Assumption $\mathrm{C}$ throughout this section. In particular, the potentiality assumption $\mathrm{(C1)}$ is justified by Proposition \ref{potentiality-implies-potentiality}. In Lemma \ref{how-to-evaluate-Leta}, we will illustrate how to effectively evaluate $L^{\eta}$ from Lipschitz smoothness of $P_{m}(x)$ and Lipschitz continuity of $h_{i}(x_{i})$.

\vspace{5pt}

\begin{custombox}
\noindent\textbf{Assumption C.}\\
(C1) (Potentiality) The randomized smoothed game \eqref{RS-games} admits an $L^{\eta}$-smooth potential function $P^{\eta}(x)$ attaining its maximum $P^{\eta}_{\max}$ and minimum $P^{\eta}_{\min}$ over $X$. \\
(C2) (Unbiasedness) For any $x^{k}$, we have $\mathbb{E}[\nabla_{x_{i}}\Tilde{m}_{i}(x^{k}_{i}, x^{k}_{-i}, \xi) \mid x^{k}] = \nabla_{x_{i}}m_{i}(x^{k}_{i}, x^{k}_{-i})$ for any $i\in [N]$.\\
(C3) (Bounded second moment) For any $x^{k}$, we have $\mathbb{E}[ \|\nabla_{x_{i}}\Tilde{m}_{i}(x^{k}_{i}, x^{k}_{-i}, \xi)-\nabla_{x_{i}}m_{i}(x^{k}_{i}, x^{k}_{-i})\|^{2} \mid x^{k}]\leq \sigma_{m}^{2}$ for some $\sigma_{m}>0$ and any $i\in [N]$.
\end{custombox}

\vspace{2pt}

\begin{lemma}\label{how-to-evaluate-Leta}
    Consider the stochastic $N$-player game \eqref{Lipschitz-continuous-games} and its randomized smoothed counterpart \eqref{RS-games}. Suppose that $P^{\eta}(x)$ defined in \eqref{new-potential-function} is a potential function of \eqref{RS-games}, where $P_{m}(x)$ is $L^{1}_{m}$-smooth. Then $P^{\eta}(x)$ is $L^{\eta}$-smooth with $L^{\eta} \triangleq L^{1}_{m} + \eta^{-1}L_{\max}\sqrt{N n_{\max}}$, where $L_{\max} \triangleq \max_{i\in [N]} L_{i}$ and $n_{\max} \triangleq \max_{i\in [N]} n_{i}$.
\end{lemma}
\begin{proof}
    By \eqref{new-potential-function}, we have that for any $x, y\in X$,
    \begin{align*}
        \| \nabla P^{\eta}(x) - \nabla P^{\eta}(y) \| &\leq \| \nabla P_{m}(x) - \nabla P_{m}(y) \| + \left\| \nabla \left( \sum_{i=1}^{N} h^{\eta}_{i}(x_{i}) \right) - \nabla \left( \sum_{i=1}^{N} h^{\eta}_{i}(y_{i}) \right) \right\|.
    \end{align*}
    By Lemma \ref{RS-property-grad}-(iv), it follows that
    \begin{align*}
        &\left\| \nabla \left( \sum_{i=1}^{N} h^{\eta}_{i}(x_{i}) \right) - \nabla \left( \sum_{i=1}^{N} h^{\eta}_{i}(y_{i}) \right) \right\|^{2} = \left\| \sum_{i=1}^{N} ( \nabla h^{\eta}_{i}(x_{i}) - \nabla h^{\eta}_{i}(y_{i}) ) \right\|^{2} \\
        &\leq N \sum_{i=1}^{N} \| \nabla h^{\eta}_{i}(x_{i}) - \nabla h^{\eta}_{i}(y_{i}) \|^{2} = N \sum_{i=1}^{N} \| \nabla_{x_{i}} h^{\eta}_{i}(x_{i}) - \nabla_{x_{i}} h^{\eta}_{i}(y_{i}) \|^{2} \\
        &\leq N \sum_{i=1}^{N} \frac{L^{2}_{i}n_{i}}{\eta^{2}} \| x_{i} - y_{i} \|^{2} \leq \frac{NL^{2}_{\max}n_{\max}}{\eta^{2}} \sum_{i=1}^{N} \| x_{i} - y_{i} \|^{2} = \frac{NL^{2}_{\max}n_{\max}}{\eta^{2}} \| x - y \|^{2},
    \end{align*}
    where $L_{\max} \triangleq \max_{i\in [N]} L_{i}$ and $n_{\max} \triangleq \max_{i\in [N]} n_{i}$. Since $P_{m}(x)$ is $L^{1}_{m}$-smooth, we deduce that
    \begin{equation*}
        \| \nabla P^{\eta}(x) - \nabla P^{\eta}(y) \| \leq \underbrace{\left( L^{1}_{m} + \frac{L_{\max}\sqrt{N n_{\max}}}{\eta} \right)}_{\triangleq L^{\eta}} \| x - y \|,
    \end{equation*}
    which completes the proof.
\end{proof}

By Lemma \ref{RS-property-grad}, the exact gradient of $h^{\eta}_{i}(x_{i})$ is given by
\begin{equation*}
    \nabla_{x_{i}} h^{\eta}_{i}(x_{i}) = \left( \frac{n_{i}}{2\eta} \right) \mathbb{E}_{v_{i}\in \eta \mathbb{S}_{n_{i}}} \left[ (h_{i}(x_{i} + v_{i}) - h_{i}(x_{i} - v_{i})) \frac{v_{i}}{\|v_{i}\|} \right].
\end{equation*}
Therefore, at $k$th iteration we may employ the update rule
\begin{equation*}
    x^{k+1}_{i} = \Pi_{X_{i}} [ x^{k}_{i} - \gamma_{k} \Delta^{k, f}_{i, S_{k}}], ~ \forall i\in [N],
\end{equation*}
where $\gamma_{k}$ is the stepsize and $\Delta^{k, f}_{i, S_{k}}$ is the mini-batch gradient with batch size $S_{k}$, defined by $\Delta^{k, f}_{i, S_{k}} \triangleq \Delta^{k, h}_{i, S_{k}} + \Delta^{k, m}_{i, S_{k}}$ with $\Delta^{k, h}_{i, S_{k}} \triangleq \frac{1}{S_{k}}\sum_{l=1}^{S_{k}} \Delta^{k, h}_{i, l}$ and $\Delta^{k, m}_{i, S_{k}} \triangleq \frac{1}{S_{k}}\sum_{l=1}^{S_{k}} \Delta^{k, m}_{i, l}$, where $\Delta^{k, h}_{i, l}$ and $\Delta^{k, m}_{i, l}$ are
\begin{equation}\label{RS-RSG-minibatch-gradient}
    \Delta^{k, h}_{i, l} \triangleq \frac{n_{i}}{2\eta}(\Tilde{h}_{i}(x^{k}_{i}+v^{k}_{i,l}, \xi^{k}_{i,l})-\Tilde{h}_{i}(x^{k}_{i}-v^{k}_{i,l}, \xi^{k}_{i,l}))\frac{v^{k}_{i,l}}{\|v^{k}_{i,l}\|} \quad\text{and}\quad \Delta^{k, m}_{i, l} \triangleq \nabla_{x_{i}}\Tilde{m}_{i}(x^{k}_{i}, x^{k}_{-i}, \xi^{k}_{i, l}).
\end{equation}
Based on the above discussion, we state our RS-RSG scheme in Algorithm \ref{RS-RSG}.

\begin{algorithm}
\caption{~RS-RSG scheme} \label{RS-RSG}
\begin{algorithmic}[1]
\State \textbf{Input}: starting point $x^{0}\in X$, iteration limit $T$, stepsizes $\{\gamma_{k}\}_{k\geq 0}$, batch sizes $\{S_{k}\}_{k\geq 0}$, and probability mass function $P_{R}$ supported on $\{ 1, \dots, T \}$.
\State Let $R$ be a random variable with probability mass function $P_{R}$.
\For{$k = 0, 1, \dots, R-1$}
    \For{$i = 1, \dots, N$}
        \State (i) Pick $S_{k}$ i.i.d. realizations $\{\xi^{k}_{i, l}\}_{l=1}^{S_{k}}$ of $\xi$ and realizations $\{ v^{k}_{i,l} \}_{l=1}^{S_{k}}$ of $v_{i}\in \eta \mathbb{S}_{n_{i}}$;
        \State (ii) Compute $\Delta^{k, h}_{i, S_{k}}$, $\Delta^{k, m}_{i, S_{k}}$, and the compute $\Delta^{k, f}_{i, S_{k}}$;
        \State (iii) Update $x^{k+1}_{i} = \Pi_{X_{i}} [ x^{k}_{i} - \gamma_{k}\Delta^{k, f}_{i, S_{k}} ]$;
    \EndFor
\EndFor
\State \textbf{Output}: $x^{R}$ as final estimate.
\end{algorithmic}
\end{algorithm}

Before proceeding, we show the following unbiasedness and bounded second moment properties.

\begin{lemma}\label{zo-gradient-unbiased-second-moment}
    For the mini-batch gradient estimator $\Delta^{k, f}_{i, l} \triangleq \Delta^{k, h}_{i, l} + \Delta^{k, m}_{i, l}$, we have\\
    \emph{(i)} Suppose assumption $\mathrm{(C2)}$ holds. Then $\mathbb{E}[ \Delta^{k, f}_{i, l} \mid x^{k} ] = \nabla_{x_{i}} f^{\eta}_{i}(x^{k}_{i}, x^{k}_{-i})$ holds for any $i\in [N]$. \\
    \emph{(ii)} Suppose assumption $\mathrm{(C3)}$ holds. Then $\mathbb{E}[\| \Delta^{k, f}_{i, l} - \nabla_{x_{i}} f^{\eta}_{i}(x^{k}_{i}, x^{k}_{-i}) \|^{2} \mid x^{k}] \leq \sigma^{2} \triangleq 32\sqrt{2\pi} L^{2}_{\max} n_{\max} + 2\sigma^{2}_{m}$ holds for any $i\in [N]$, where $L_{\max} = \max_{i\in [N]} L_{i}$, and $n_{\max} = \max_{i\in [N]} n_{i}$.
\end{lemma}
\begin{proof}
    The unbiasedness property (i) follows immediately from assumption $\mathrm{(C2)}$ and the fact that $\xi^{k}_{i, l}$ and $v^{k}_{i, l}$ are i.i.d. realizations of $\xi$ and $v_{i}$, respectively. For the second moment property (ii), we have
    \begin{equation*}
        \begin{aligned}
            & \mathbb{E}[\| \Delta^{k, f}_{i, l} - \nabla_{x_{i}} f^{\eta}_{i}(x^{k}_{i}, x^{k}_{-i}) \|^{2} \mid x^{k}] \\
            &= \mathbb{E}[\| \Delta^{k, h}_{i, l} + \Delta^{k, m}_{i, l} - \nabla_{x_{i}} h^{\eta}_{i}(x^{k}_{i}) - \nabla_{x_{i}} m_{i}(x^{k}_{i}, x^{k}_{-i}) \|^{2} \mid x^{k}] \\
            &\leq 2\mathbb{E}[\| \Delta^{k, h}_{i, l} - \nabla_{x_{i}} h^{\eta}_{i}(x^{k}_{i}) \|^{2} \mid x^{k}] + 2\mathbb{E}[\| \Delta^{k, m}_{i, l} - \nabla_{x_{i}} m_{i}(x^{k}_{i}, x^{k}_{-i}) \|^{2} \mid x^{k}] \\
            &\leq 2\mathbb{E}[\| \Delta^{k, h}_{i, l} \|^{2}] + 2\sigma^{2}_{m} \leq 32\sqrt{2\pi} L^{2}_{\max} n_{\max} + 2\sigma^{2}_{m} \triangleq \sigma^{2},
        \end{aligned}
    \end{equation*}
    where the penultimate inequality is due to the unbiasedness $\mathbb{E}[\Delta^{k, h}_{i, l} \mid x^{k}_{i}] = \nabla_{x_{i}} h^{\eta}_{i}(x^{k}_{i})$ and the last inequality is from \cite[Lemma 2.4-(v)]{marrinan-shanbhag-yousefian-2026}.
\end{proof}

Based on the above lemma, we may establish the convergence of RS-RSG. In contrast to Theorem \ref{RSG-convergence}, where we only need to consider stochastic first-order $(\mathcal{SFO})$ oracles, Theorem \ref{RS-RSG-main-theorem} also requires accounting for stochastic zeroth-order $(\mathcal{SZO})$ oracles.

\begin{theorem}[Convergence of RS-RSG]\label{RS-RSG-main-theorem}
    Consider the stochastic $N$-player Lipschitz continuous game \eqref{Lipschitz-continuous-games} and its randomized smoothed counterpart \eqref{RS-games}. Suppose that Assumption $\mathrm{C}$ holds. Suppose that we choose $\gamma_{k} = 1/(2L^{\eta})$ and $P_{R}(k) = 1/T$ for all $k = 1, \dots, T$. Given a sufficiently large number $M_{0}$ of $\mathcal{SZO}$ calls and a sufficiently large number $M_{1}$ of $\mathcal{SFO}$ calls. We define $M \triangleq \min\{ M_{0}/2, M_{1} \}$. If the batch size $S_{k}$ is fixed as
    \begin{equation*}
        S_{k} = S \triangleq \left\lceil \frac{\sigma\sqrt{6M}}{4L^{\eta}D^{\eta}} \right\rceil,
    \end{equation*}
    where $D^{\eta} \triangleq ((P^{\eta}_{\max}-P^{\eta}_{\min})/L^{\eta})^{1/2}$ and $\sigma^{2}$ is defined in Lemma \ref{zo-gradient-unbiased-second-moment}-(ii). Then we have
    \begin{equation}\label{RS-RSG-main-theorem-final-bound}
        \mathbb{E}[\| G^{\eta}_{\gamma_{R}}(x^{R}) \|^{2}] \leq \frac{16(L^{\eta})^{2}(D^{\eta})^{2}N}{M} + \frac{8\sqrt{6}L^{\eta}D^{\eta} N \sigma}{\sqrt{M}},
    \end{equation}
    where the residual is defined as $G^{\eta}_{\gamma}(x) \triangleq \tfrac{1}{\gamma} (x-\Pi_{X}[x-\gamma F^{\eta}(x)])$ with $F^{\eta}(x) \triangleq (\nabla_{x_{i}} f^{\eta}_{i}(x_{i}, x_{-i}))_{i=1}^{N}$.
\end{theorem}
\begin{proof}
    We can see that RS-RSG can perform at most $T = \min\{ \lfloor (M_{0}/2)/(SN) \rfloor, \lfloor M_{1}/(SN) \rfloor \}$ iterations. Since $M_{0}$ and $M_{1}$ are sufficiently large, we have the lower bound\footnote{If $a \leq b$ and $c \leq d$, the inequality $\min\{a, c\} \leq \min\{ b, d \}$ holds.}
    \begin{equation*}
        T \geq \min \{ (M_{0}/2)/(2SN), M_{1}/(2SN) \} = M/(2SN).
    \end{equation*}
    Therefore, we have the same lower bound on $T$ as Theorem \ref{RSG-convergence}. The reminder of the proof proceeds similarly.
\end{proof}

\begin{theorem}[Complexities of RS-RSG]\label{RS-RSG-complexity}
    Consider the same setting as Theorem \ref{RS-RSG-main-theorem}. Suppose that $P^{\eta}(x)$ defined in \eqref{new-potential-function} is a potential function of \eqref{RS-games}, where $P_{m}(x)$ is $L^{0}_{m}$-Lipschitz and $L^{1}_{m}$-smooth. Suppose that $X$ is compact with diameter $D_{X} > 0$. We define $L_{\max} = \max_{i\in [N]} L_{i}$ and $n_{\max} = \max_{i\in [N]} n_{i}$. Then, to reach a point $x^{R}$ satisfying $\mathbb{E}[\| G^{\eta}_{\gamma_{R}}(x^{R}) \|]\leq \epsilon$, we have \\
    \emph{(i)} The required sample complexity is $\mathcal{O}(L^{4}_{\max}n^{3/2}_{\max}N^{3}\eta^{-1}\epsilon^{-4})$. \\
    \emph{(ii)} The required iteration complexity is $\mathcal{O}(L^{3}_{\max} n_{\max} N \eta^{-1} \epsilon^{-2})$.
\end{theorem}
\begin{proof}
   (i) Since $M$ is sufficiently large, we know from Theorem \ref{RS-RSG-main-theorem} that
   \begin{equation}\label{RS-RSG-complexity-eqn1}
       \mathbb{E}[\| G^{\eta}_{\gamma_{R}}(x^{R}) \|^{2}] \leq \mathcal{O} \left( \frac{L^{\eta}D^{\eta} N \sigma}{\sqrt{M}} \right) = \mathcal{O} \left( \frac{(L^{\eta})^{1/2} (P^{\eta}_{\max} - P^{\eta}_{\min})^{1/2} N \sigma}{\sqrt{M}} \right).
   \end{equation}
   By Lemma \ref{zo-gradient-unbiased-second-moment}, we know that
   \begin{equation}\label{RS-RSG-complexity-eqn2}
       \sigma \leq \mathcal{O} \left( L_{\max} n^{1/2}_{\max} \right).
   \end{equation}
   Additionally by Lemma \ref{RS-property-grad}-(iv), we may obtain
   \begin{equation*}
       \begin{aligned}
           &\| \nabla P^{\eta}(x) - P^{\eta}(y) \| \leq \| \nabla P^{\eta}_{m}(x) - P^{\eta}_{m}(y) \| + \sum_{i=1}^{N} \| \nabla_{x_{i}} h^{\eta}_{i}(x_{i}) - \nabla_{x_{i}} h^{\eta}_{i}(y_{i}) \| \\
           &\leq L^{1}_{m} \| x - y \| + \sum_{i=1}^{N} \frac{L_{i}\sqrt{n_{i}}}{\eta} \| x_{i} - y_{i} \| \leq L^{1}_{m} \| x - y \| + \frac{L_{\max}\sqrt{n_{\max}}\sqrt{N}}{\eta} \| x - y \|,
       \end{aligned}
   \end{equation*}
   where the last inequality is due to $\sum_{i=1}^{N} \| x_{i} - y_{i} \| \leq \sqrt{N} \| x - y \|$. Therefore, we may derive that
   \begin{equation}\label{RS-RSG-complexity-eqn3}
       L^{\eta} = L^{1}_{m} + \frac{L_{\max}\sqrt{n_{\max}}\sqrt{N}}{\eta} \implies (L^{\eta})^{1/2} \leq \mathcal{O} \left( \frac{L^{1/2}_{\max} n^{1/4}_{\max} N^{1/4}}{\eta^{1/2}} \right).
   \end{equation}
   Suppose that $P^{\eta}$ attains its maximum and minimum over $X$ at $x^{\max}$ and $x^{\min}$, respectively. Similarly, by Lemma \ref{RS-property-grad}-(ii), it follows that
   \begin{equation}\label{RS-RSG-complexity-eqn3.5}
       \begin{aligned}
           P^{\eta}_{\max} - P^{\eta}_{\min} &= (P_{m}(x^{\max}) - P_{m}(x^{\min})) + \sum_{i=1}^{N} \left( h^{\eta}_{i}(x^{\max}_{i}) - h^{\eta}_{i}(x^{\min}_{i}) \right) \\
           &\leq (L^{0}_{m} + L_{\max}\sqrt{N} ) \| x^{\max} - x^{\min} \| \leq (L^{0}_{m} + L_{\max}\sqrt{N} ) D_{X},
       \end{aligned}
   \end{equation}
   which implies that
   \begin{equation}\label{RS-RSG-complexity-eqn4}
       ( P^{\eta}_{\max} - P^{\eta}_{\min} )^{1/2} \leq \mathcal{O} \left( L^{1/2}_{\max} N^{1/4} \right).
   \end{equation}
   By plugging \eqref{RS-RSG-complexity-eqn2}, \eqref{RS-RSG-complexity-eqn3}, and \eqref{RS-RSG-complexity-eqn4} into \eqref{RS-RSG-complexity-eqn1}, we may arrive that
   \begin{equation*}
       \mathbb{E}[\| G^{\eta}_{\gamma_{R}}(x^{R}) \|^{2}] \leq \mathcal{O} \left( \frac{L^{2}_{\max} n^{3/4}_{\max} N^{3/2}}{\eta^{1/2} M^{1/2}} \right).
   \end{equation*}
   In order to ensure that $\mathbb{E}[\| G^{\eta}_{\gamma_{R}}(x^{R}) \|] \leq \epsilon$ holds, we need
    \begin{equation}\label{RS-RSG-complexity-M}
        M = \mathcal{O} \left( \frac{ L^{4}_{\max} n^{3/2}_{\max} N^{3} }{ \eta \epsilon^{-4} } \right).
    \end{equation}
   By setting minimal $M_{0} = 2M$ and $M_{1} = M$, it can be seen that for any $\epsilon > 0$, we need at least $M_{0} + M_{1} = 3M = \mathcal{O}(L^{4}_{\max}n^{3/2}_{\max}N^{3}\eta^{-1}\epsilon^{-4})$ samples to ensure that $\mathbb{E}[\| G^{\eta}_{\gamma_{R}}(x^{R}) \|] \leq \epsilon$ holds.

   \noindent (ii) Since $T = \min\{ \lfloor (M_{0}/2)/(SN) \rfloor, \lfloor M_{1}/(SN) \rfloor \}$, it follows that
   \begin{equation*}
       T \leq \min\{ (M_{0}/2)/(SN), M_{1}/(SN) \} = M/(SN) \implies T = \mathcal{O} \left( \frac{ L^{\eta} D^{\eta} \sqrt{M} }{ \sigma N } \right).
   \end{equation*}
   By \eqref{RS-RSG-complexity-eqn3}-\eqref{RS-RSG-complexity-M}, together with the fact that $\sigma \geq \sigma_{m}$ from Lemma \ref{zo-gradient-unbiased-second-moment}-(ii), it leads to
   \begin{equation*}
       T = \mathcal{O} \left( \frac{ L^{3}_{\max} n_{\max} N }{ \eta \epsilon^{-2} } \right),
   \end{equation*}
   which completes the proof.
\end{proof}

\begin{remark}\em
    We make two remarks here. First, when specialized to the optimization setting with $N=1$, the above sample complexity is consistent with the result in \cite{lin-zheng-jordon-2022}, where a fixed batch size is also employed. Second, although the zeroth-order gradient estimator used in RS-RSG is unbiased, the resulting iteration and sample complexities exhibit a stronger dependence on $N$ than those of RSG.
\end{remark}

\subsection{Approximation of CNE}\label{Sec-4.2}

In this subsection, we analyze how the smoothed CNE of \eqref{RS-games} approximates the CNE of the original game \eqref{Lipschitz-continuous-games}. Before proceeding, we recall an important property of generalized variational inequalities. For $\mathrm{GVI}\:(X, \mathbf{F})$ where $\mathbf{F}: X \rightrightarrows \mathbb{R}^n$, we consider the set-valued map as
\begin{equation}\label{GVI-ncvx-ns}
    \mathbf{G}_{\gamma}(x) \triangleq \tfrac{1}{\gamma}(x-\Pi_{X}[x-\gamma \mathbf{F}(x)]),~ \gamma>0.
\end{equation}
We may show that $0\in \mathbf{G}_{\gamma}(x^{*})$ if and only if there exists $y^{\ast}\in \mathbf{F}(x^{\ast})$ such that the pair $(x^{*}, y^{*})$ solves $\mathrm{GVI}\:(X, \mathbf{F})$. Therefore, we are motivated to use the quantity $\text{dist}(0, \mathbf{G}_{\gamma}(x))^{2}$ as the residual measure.

~

Given two nonempty sets $A$ and $B$ in $\mathbb{R}^{n}$. Recall that the one-sided deviation of $A$ from $B$ is defined as $\mathbb{D}(A, B) \triangleq \sup_{x\in A} \mathrm{dist}\: (x, B)$. The following two lemmas are useful in our analysis.

\begin{lemma}\label{deviation-bound}
    Suppose that $A$ and $B$ are two nonempty subsets in $\mathbb{R}^{n}$. Suppose that $A \subseteq B + d\mathbb{B}$ holds for some $d>0$, where $\mathbb{B}$ is a closed unit ball. Then we have that $\mathbb{D}(A, B) \leq d$.
\end{lemma}
\begin{proof}
    The inclusion $A \subseteq B + d\mathbb{B}$ means that for any $x\in A$, there exists some $x_{B}\in B$ such that $x\in x_{B} + d\mathbb{B}$, implying that $\| x - x_{B} \|\leq d$ holds. Therefore, we have
    \begin{equation*}
        \mathrm{dist}\: (x, B) = \inf_{y\in B} \| y - x \| \leq \| x_{B} - x \| \leq d.
    \end{equation*}
    Taking the supremum over $x\in A$, it leads to
    \begin{equation*}
        \mathbb{D}(A, B) = \sup_{x\in A} \mathrm{dist}\: (x, B) \leq d,
    \end{equation*}
    which completes the proof.
\end{proof}

We note that neither convexity nor closedness is required in Lemma \ref{deviation-bound}. Based on it, we next show that under the Lipschitz continuity of $\partial^{C} h_{i}(x_{i})$, the one-sided deviation $\mathbb{D}( \partial^{C}_{\eta} h_{i}(x_{i}), \partial^{C} h_{i}(x_{i}) )$ can be upper bounded, where $\partial^{C}_{\eta} h_{i}(x_{i})$ is the $\eta$-Clarke subdifferential defined in \eqref{delta-Clarke-subdifferential}. Recall that a set-valued mapping $F: X \rightrightarrows Y$ is said to be $L$-Lipschitz on $X$ \cite[Definition 9.26]{rockafellar-wets-1998} if there exists $L > 0$ such that $F(x_{1}) \subseteq F(x_{2}) \,+\, L\|x_{1}-x_{2}\|\mathbb{B}_{Y}$ for any $x_{1}, x_{2}\in X$, where $\mathbb{B}_{Y}$ is the unit ball in $Y$.

\begin{lemma}\label{deviation-bound-2}
    Suppose that each Clarke subdifferential $\partial^{C} h_{i}(x_{i})$ is $\hat{L}_{i}$-Lipschitz. Then the one-sided deviation bound
    \begin{equation*}
        \mathbb{D}( \partial^{C}_{\eta} h_{i}(x_{i}), \partial^{C} h_{i}(x_{i}) ) \leq \eta \hat{L}_{i}
    \end{equation*}
    holds for any $i\in [N]$. 
\end{lemma}
\begin{proof}
    By invoking Lemma \ref{deviation-bound}, it suffices to show that
    \begin{equation}\label{lemma-LC-proof-eqn1}
        \partial^{C}_{\eta} h_{i}(x_{i}) \subseteq \partial^{C} h_{i}(x_{i}) + \eta \hat{L}_{i} \mathbb{B}_{n_{i}}.
    \end{equation}
    By the $\eta$-Clarke subdifferential definition \eqref{delta-Clarke-subdifferential}, we have that
    \begin{equation*}
        \partial^{C}_{\eta} h_{i}(x_{i}) = \mathrm{conv} \left( \bigcup\limits_{y_{i}\in \mathbb{B}_{n_{i}}(x_{i}; \eta)} \partial^{C} h_{i}(y_{i}) \right)
    \end{equation*}
    for any $i\in [N]$. For any $\zeta_{i}\in \partial^{C}_{\eta} h_{i}(x_{i})$, it follows from the convex combination expression of convex hull (see \cite[Proposition 1.30]{mordukhovich-nam-2023}) that
    \begin{equation}\label{lemma-LC-proof-eqn2}
        \zeta_{i} = \lambda_{1} a_{1} + \cdots + \lambda_{m} a_{m}, \; \sum_{j=1}^{m} \lambda_{j} = 1, \; \lambda_{j} \geq 0, \; a_{j}\in \bigcup\limits_{y_{i}\in \mathbb{B}_{n_{i}}(x_{i}; \eta)} \partial^{C} h_{i}(y_{i}).
    \end{equation}
    Without loss of generality, we assume that
    \begin{equation*}
        a_{j} \in \partial^{C} h_{i}(y_{i, j}) \;\;\text{where}\;\; y_{i,j} \in \mathbb{B}_{n_{i}}(x_{i}; \eta).
    \end{equation*}
    By our assumption on $\hat{L}_{i}$-Lipschitz continuity of $\partial^{C} h_{i}(x_{i})$, we deduce that
    \begin{align}\label{lemma-LC-proof-eqn3}
        a_{j} \in \partial^{C} h_{i}(y_{i, j}) \subseteq \partial^{C} h_{i}(x_{i}) + \hat{L}_{i} \| y_{i, j}-x_{i} \| \mathbb{B}_{n_{i}} \subseteq \partial^{C} h_{i}(x_{i}) + \eta \hat{L}_{i} \mathbb{B}_{n_{i}}.
    \end{align}
    Combining \eqref{lemma-LC-proof-eqn2} and \eqref{lemma-LC-proof-eqn3}, we can see that
    \begin{align*}
        \zeta_{i} &= \lambda_{1} a_{1} + \cdots + \lambda_{m} a_{m} \\
        &\in \lambda_{1} \left( \partial^{C} h_{i}(x_{i}) + \eta \hat{L}_{i} \mathbb{B}_{n_{i}} \right) + \cdots + \lambda_{m} \left( \partial^{C} h_{i}(x_{i}) + \eta \hat{L}_{i} \mathbb{B}_{n_{i}} \right) \\
        &\subseteq \partial^{C} h_{i}(x_{i}) + \eta \hat{L}_{i} \mathbb{B}_{n_{i}},
    \end{align*}
    where the last inclusion is due to the fact that both $\partial^{C} h_{i}(x_{i})$ and $\eta \hat{L}_{i} \mathbb{B}_{n_{i}}$ are convex sets, so is the sum. We have shown that \eqref{lemma-LC-proof-eqn1} holds. By Lemma \ref{deviation-bound}, we establish the bound $\mathbb{D}( \partial^{C}_{\eta} h_{i}(x_{i}), \partial^{C} h_{i}(x_{i}) ) \leq \eta \hat{L}_{i}$, which completes the proof.
\end{proof}

Next we prove the main approximation theorem. Recall that in the context of Lipschitz continuous games, the set-valued map in $\mathrm{GVI}\:(X, \mathbf{F})$ is given by the concatenation of Clarke subdifferentials.

\begin{theorem}[Approximation of CNE]\label{rs-approx-main-thm}
    Consider the stochastic $N$-player Lipschitz continuous game \eqref{Lipschitz-continuous-games} and its randomized smoothed counterpart \eqref{RS-games}. Suppose that each Clarke subdifferential $\partial^{C} h_{i}(x_{i})$ is $\hat{L}_{i}$-Lipschitz. If the output $x^{R}$ of RS-RSG satisfies $\mathbb{E}[\| G^{\eta}_{\gamma_{R}}(x^{R}) \|^{2}] \leq \epsilon^{2}$, we have
    \begin{equation}\label{LC-thm-main-result}
        \mathbb{E}[\mathrm{dist}\:(0, \mathbf{G}_{\gamma_{R}}(x^{R}))^{2}] \leq 2\eta^{2}\sum_{i=1}^{N}\hat{L}^{2}_{i} + 2\epsilon^{2},
    \end{equation}
    where $\mathbf{G}_{\gamma_{R}}(x^{R}) \triangleq \tfrac{1}{\gamma_{R}}(x^{R}-\Pi_{X}[x^{R}-\gamma_{R} \mathbf{F}(x^{R})])$ with $\mathbf{F}(x^{R}) \triangleq (\partial^{C} h_{i}(x^{R}_{i}))_{i=1}^{N} + (\nabla_{x_{i}} m_{i}(x^{R}_{i}, x^{R}_{-i}))_{i=1}^{N}$.
\end{theorem}
\begin{proof}
    Recall from \cite[Proposition 4.3.1]{cui-pang-2021} that the Clarke subdifferential is a convex compact set. Therefore, the projection of $\nabla_{x_{i}} h^{\eta}_{i}(x^{R}_{i})$ onto $\partial^{C} h_{i}(x^{R}_{i})$ is unique. We denote it by
    \begin{equation}\label{LC-smooth-thm-proj}
        u^{R}_{i} \triangleq \Pi_{\partial^{C} h_{i}(x^{R}_{i})}[\nabla_{x_{i}} h^{\eta}_{i}(x^{R}_{i})],~\forall i\in [N].
    \end{equation}
    We define $u^{R} \triangleq (u^{R}_{i})_{i=1}^{N}\in (\partial^{C} h_{i}(x^{R}_{i}))_{i=1}^{N}$ and
    \begin{equation*}
        g_{\gamma_{R}}(x^{R}) \triangleq \tfrac{1}{\gamma_{R}} (x^{R}-\Pi_{X}[x^{R}-\gamma_{R} ( u^{R} + (\nabla_{x_{i}} m_{i}(x^{R}_{i}, x^{R}_{-i}))_{i=1}^{N} )]).
    \end{equation*}
    It can be seen that $g_{\gamma_{R}}(x^{R})$ is a selection of the set-valued map $\mathbf{G}_{\gamma_{R}}(x^{R})$, implying that
    \begin{align}\label{LC-smooth-thm-eqn1}
        \mathrm{dist}\:(0, \mathbf{G}_{\gamma_{R}}(x^{R}))^{2} &\leq \| g_{\gamma_{R}}(x^{R}) \|^{2} \notag \\
        &\leq 2\| g_{\gamma_{R}}(x^{R})-G^{\eta}_{\gamma_{R}}(x^{R}) \|^{2} + 2\| G^{\eta}_{\gamma_{R}}(x^{R}) \|^{2} \notag \\
        &\leq 2\| u^{R} + (\nabla_{x_{i}} m_{i}(x^{R}_{i}, x^{R}_{-i}))_{i=1}^{N} - F^{\eta}(x^{R}) \|^{2} + 2\| G^{\eta}_{\gamma_{R}}(x^{R}) \|^{2} \notag \\
        &\leq 2\| u^{R} - (\nabla_{x_{i}} h^{\eta}_{i}(x^{R}_{i}))_{i=1}^{N} \|^{2} + 2\| G^{\eta}_{\gamma_{R}}(x^{R}) \|^{2},
    \end{align}
    where the penultimate inequality is implied by invoking \cite[Proposition 1]{ghadimi-lan-zhang-2016}. By noting the projection equality \eqref{LC-smooth-thm-proj}, we may bound the term $\| u^{R} - (\nabla_{x_{i}} h^{\eta}_{i}(x^{R}_{i}))_{i=1}^{N} \|^{2}$ in \eqref{LC-smooth-thm-eqn1} by
    \begin{align}\label{LC-smooth-thm-eqn2}
        \| u^{R} - (\nabla_{x_{i}} h^{\eta}_{i}(x^{R}_{i}))_{i=1}^{N} \|^{2} &= \sum_{i=1}^{N} \| u^{R}_{i} - \nabla_{x_{i}} h^{\eta}_{i}(x^{R}_{i}) \|^{2} \notag \\
        &= \sum_{i=1}^{N} \mathrm{dist}\: (\nabla_{x_{i}}h^{\eta}_{i}(x^{R}_{i}), \partial^{C} h_{i}(x^{R}_{i}))^{2} \notag \\
        &\overset{(*)}{\leq} \sum_{i=1}^{N} \left[ \sup_{z_{i} \in \partial^{C}_{\eta} h_{i}(x^{R}_{i})} \mathrm{dist}\: (z_{i}, \partial^{C} h_{i}(x^{R}_{i})) \right] ^{2} \notag \\
        &= \sum_{i=1}^{N} \mathbb{D}^{2}( \partial^{C}_{\eta} h_{i}(x_{i}), \partial^{C} h_{i}(x_{i}) ) \overset{\text{Lemma \ref{deviation-bound-2}}}{\leq} \eta^{2} \sum_{i=1}^{N} \hat{L}^{2}_{i},
    \end{align}
    where $(*)$ is due to Lemma \ref{RS-property-inclusion}-(i). By plugging \eqref{LC-smooth-thm-eqn2} into \eqref{LC-smooth-thm-eqn1} and taking the conditions on both sides, we may arrive the upper bound \eqref{LC-thm-main-result}.
\end{proof}

\begin{corollary}\label{lc-main-corl}
    Consider the stochastic $N$-player Lipschitz continuous game \eqref{Lipschitz-continuous-games} and its randomized smoothed counterpart \eqref{RS-games}. Consider the same setting as in Theorem \ref{rs-approx-main-thm}. Suppose that $x^{R}$ is a CNE of \eqref{RS-games}. Then we have the expected upper bound
    \begin{equation}\label{lc-main-corl-bound}
        \mathbb{E} \left[ \mathrm{dist}\:(0, \mathbf{G}_{\gamma_{R}}(x^{R}))^{2} \right] \leq 2\eta^{2}\sum_{i=1}^{N} \hat{L}^{2}_{i}.
    \end{equation}
\end{corollary}
\begin{proof}
    The proof follows immediately by noting that $x^{R}$ is a CNE of \eqref{RS-games} if and only if $\epsilon = 0$ in Theorem \ref{rs-approx-main-thm}.
\end{proof}

\begin{remark}\em
    Two remarks are in order. First, by assuming Lipschitz continuity of Clarke subdifferentials, we obtain an $\mathcal{O}(\eta^{2})$ approximation, improving the $\mathcal{O}(\eta)$ approximation in Lemma \ref{RS-property-inclusion}. Second, for ease of presentation in Section \ref{Sec-5}, the entire discussion in Section \ref{Sec-4} is based on the structured setting, i.e., each player’s objective is the sum of a nonconvex nonsmooth private term and a smooth coupling term. Our results can be generalized to general nonconvex nonsmooth objectives $f_{i}(\bullet, x_{-i})$. However, the potentiality of a general randomized smoothed game is not always valid. We leave the investigation of general stochastic nonconvex nonsmooth potential games for future work.
\end{remark}

\section{Stochastic Nonconvex Nonsmooth Potential Hierarchical Games}\label{Sec-5}

In this section, we consider the stochastic $N$-player nonconvex nonsmooth hierarchical game
\begin{equation}\label{hierarchical-games}
    \min_{x_i \in X_i} ~ f_{\mathrm{H}, i}(x_{i}, x_{-i}) \triangleq \underbrace{\mathbb{E} \Big[ \Tilde{h}_{i}(x_{i}, y_{i}(x_{i}), \xi) \Big]}_{\triangleq \, h_{i}(x_{i}, y_{i}(x_{i}))} \, + \, \underbrace{\mathbb{E} \Big[ \Tilde{m}_{i}(x_{i}, x_{-i}, \xi) \Big]}_{\triangleq \, m_{i}(x_{i}, x_{-i})},~ \forall i\in [N], \tag{$\mathrm{G_{H}}$}
\end{equation}
where each $X_{i} \subseteq \mathbb{R}^{n_{i}}$ is convex and compact, each $\Tilde{h}_{i}(\bullet, y_{i}(\bullet), \xi)$ is $L_{i}$-Lipschitz continuous over $X_{i}$ for any $\xi$, and each $\Tilde{m}_{i}(\bullet, x_{-i}, \xi)$ is continuously differentiable for any $x_{-i}$ and any $\xi$. The lower-level solution $y_{i}(x_{i}): X_{i}\to \mathbb{R}^{k_{i}}$ is characterized by a parameterized SVI, i.e.,
\begin{equation}\label{lower-SVI}
    y_{i}(x_{i}) \in \mathrm{SOL} \: ( \mathbb{E}[\Tilde{F}_{i}(x_{i}, \bullet, \xi)], Y_{i} ), \tag{SVI$^{\text{lower}}$}
\end{equation}
where each $Y_{i}\subseteq \mathbb{R}^{k_{i}}$ is convex and compact. Throughout this section, we always assume that \eqref{hierarchical-games} admits a potential function $P_{\mathrm{H}}(x)$.

In contrast to recent works on stochastic hierarchical games \cite{cui-shanbhag-2023, cui-shanbhag-staudigl-2025}, we do not require the convexity assumption. The challenges of solving such a stochastic hierarchical games are twofold. One lies in the essential nonconvex nonsmooth nature of the hierarchical function, which we handle by employing the implicit randomized smoothing, thereby obtaining the randomized smoothed hierarchical game
\begin{equation}\label{hierarchical-games-RS}
    \min_{x_i \in X_i} ~ f^{\eta}_{\mathrm{H}, i}(x_{i}, x_{-i}) \triangleq h^{\eta}_{i}(x_{i}, y_{i}(x_{i})) + m_{i}(x_{i}, x_{-i}), ~ \forall i\in [N]. \tag{$\mathrm{G^{RS}_{H}}$}
\end{equation}
By Proposition \ref{potentiality-implies-potentiality}, we know that the potentiality of \eqref{hierarchical-games} implies the potentiality of \eqref{hierarchical-games-RS}. The other lies in the fact that the exact lower-level solution $y_{i}(x_{i})$ is unavailable in finite time, hence introduces an undesirable bias. We develop a biased RS-RSG scheme for solving such a class of challenging games.

\subsection{Biased RS-RSG Scheme}\label{Sec-5.1}

We make Assumption $\mathrm{D}$ throughout this section. Again, assumption $\mathrm{(D1)}$ follows from Proposition \ref{potentiality-implies-potentiality}, which shows that the potentiality of \eqref{hierarchical-games} implies the potentiality of \eqref{hierarchical-games-RS}. Regarding assumptions $\mathrm{(D4)}$ and $\mathrm{(D5)}$, for ease of presentation we ignore the dependence on $\xi$, as noted in Remark \ref{xi-assumption}. Moreover, the strong monotonicity assumption $\mathrm{(D4)}$ yields the uniqueness of the lower-level solution $y_{i}(x_{i})$.

\vspace{5pt}

\begin{custombox}
\noindent\textbf{Assumption D.}\\
(D1) (Potentiality) The randomized smoothed hierarchical game \eqref{hierarchical-games-RS} admits an $L^{\eta}_{\mathrm{H}}$-smooth potential function $P^{\eta}_{\mathrm{H}}(x)$ attaining its maximum $P^{\eta}_{\mathrm{H}, \max}$ and minimum $P^{\eta}_{\mathrm{H}, \min}$ over $X$. \\
(D2) (Unbiasedness) We have $\mathbb{E}[\nabla_{x_{i}}\Tilde{m}_{i}(x^{k}_{i}, x^{k}_{-i}, \xi) \mid x^{k}] = \nabla_{x_{i}}m_{i}(x^{k}_{i}, x^{k}_{-i})$ for any $k \geq 0$ and $i\in [N]$. \\
(D3) (Bounded second moment) For any $x^{k}$, we have $\mathbb{E}[ \|\nabla_{x_{i}}\Tilde{m}_{i}(x^{k}_{i}, x^{k}_{-i}, \xi)-\nabla_{x_{i}}m_{i}(x^{k}_{i}, x^{k}_{-i})\|^{2} \mid x^{k}]\leq \sigma_{m}^{2}$ for some $\sigma_{m}>0$ and any $i\in [N]$. \\
(D4) (Strong monotonicity) For given $x_{i}$ and $\xi$, $\tilde{F}_{i}(x_{i}, \bullet, \xi)$ is $\mu_{i}$-strongly monotone on $Y_{i}$. \\
(D5) (Lipschitz continuity) Each $\Tilde{h}_{i}(x_{i}, \bullet, \xi)$ is $L^{y}_{i}$-Lipschitz continuous over $Y_{i}$ for any $x_{i}$ and $\xi$. \\
(D6) (Lower-level SA) Consider Algorithm \ref{hierarchical-SA}. For any $k, t \geq 0$ and any $i\in [N]$, we have (i) $\mathbb{E}[\Tilde{F}_{i}(\hat{x}^{k}_{i}, y^{t}_{i}, \xi^{t}) \mid \hat{x}^{k}, y^{t}] = \mathbb{E}[\Tilde{F}(\hat{x}^{k}_{i}, y^{t}_{i}, \xi) \mid \hat{x}^{k}, y^{t}]$; (ii) $\mathbb{E}[ \|\Tilde{F}_{i}(\hat{x}^{k}_{i}, y^{t}_{i}, \xi^{t}) - \mathbb{E}[\Tilde{F}_{i}(\hat{x}^{k}_{i}, y^{t}_{i}, \xi)]\|^{2} \mid \hat{x}^{k}_{i},y^{t}_{i} ] \leq v^{2}_{i}$ holds a.s. for some $v_{i}>0$; and (iii) the bound $\| \mathbb{E}[\Tilde{F}_{i}(x_{i}, y_{i}, \xi)] \| \leq c_{F_{i}}$ holds for some $c_{F_{i}} > 0$ and for any $x_{i}\in X_{i}$ and any $y_{i}\in Y_{i}$.
\end{custombox}

\vspace{5pt}

Similar to the analysis in Section \ref{Sec-4.1}, we update $x^{k+1}$ as
\begin{equation*}
    x^{k+1}_{i} = \Pi_{X_{i}} [ x^{k}_{i} - \gamma_{k} \Delta^{k, f}_{i, S_{k}, \epsilon_{k}}],~\forall i\in [N],
\end{equation*}
where $\gamma_{k}$ is the stepsize and $\Delta^{k, f}_{i, S_{k}, \epsilon_{k}}$ is the inexact mini-batch gradient with batch size $S_{k}$, defined by $\Delta^{k, f}_{i, S_{k}, \epsilon_{k}} \triangleq \Delta^{k, h}_{i, S_{k}, \epsilon_{k}} + \Delta^{k, m}_{i, S_{k}}$ with $\Delta^{k, h}_{i, S_{k}, \epsilon_{k}} \triangleq \frac{1}{S_{k}}\sum_{l=1}^{S_{k}} \Delta^{k, h}_{i, l, \epsilon_{k}}$ and $\Delta^{k, m}_{i, S_{k}} \triangleq \frac{1}{S_{k}}\sum_{l=1}^{S_{k}} \Delta^{k, m}_{i, l}$, where $\Delta^{k, h}_{i, l, \epsilon_{k}}$ and $\Delta^{k, m}_{i, l}$ are defined as
\begin{equation}\label{Delta_m_Delta_h_inexact}
    \begin{aligned}
        \Delta^{k, h}_{i, l, \epsilon_{k}} &\triangleq \frac{n_{i}}{2\eta}(\Tilde{h}_{i}(x^{k}_{i}+v^{k}_{i,l}, y^{\epsilon_{k}}_{i}(x^{k}_{i}+v^{k}_{i,l}), \xi^{k}_{i,l})-\Tilde{h}_{i}(x^{k}_{i}-v^{k}_{i,l}, y^{\epsilon_{k}}_{i}(x^{k}_{i}-v^{k}_{i,l}), \xi^{k}_{i,l}))\frac{v^{k}_{i,l}}{\|v^{k}_{i,l}\|}, \\
        \Delta^{k, m}_{i, l} &\triangleq \nabla_{x_{i}}\Tilde{m}_{i}(x^{k}_{i}, x^{k}_{-i}, \xi^{k}_{i, l}),
    \end{aligned}
\end{equation}
where the inexact solutions $y^{\epsilon_{k}}_{i}(x^{k}_{i}+v^{k}_{i,l})$ and $y^{\epsilon_{k}}_{i}(x^{k}_{i}-v^{k}_{i,l})$ are obtained via SA scheme.

We state our biased RS-RSG scheme in Algorithm \ref{biased-RS-RSG}. We use the notation $\hat{x}^{k}_{i}$ to unify $x^{k}_{i}+v^{k}_{i,l}$ and $x^{k}_{i}-v^{k}_{i,l}$ in the lower-level SA scheme, i.e., when we say $\hat{x}^{k}_{i}$, it can be $x^{k}_{i}+v^{k}_{i,l}$ or $x^{k}_{i}-v^{k}_{i,l}$. The following Lemma \ref{hierarchical-lower-level-SA-lemma} provides the convergence analysis of the lower-level SA scheme.

\begin{algorithm}
\caption{~Biased RS-RSG scheme} \label{biased-RS-RSG}
\begin{algorithmic}[1]
\State \textbf{Input}: starting point $x^{0}\in X$, iteration limit $T$, stepsizes $\{\gamma_{k}\}_{k\geq 0}$, batch sizes $\{S_{k}\}_{k\geq 0}$, and probability mass function $P_{R}$ supported on $\{ 1, \dots, T \}$.
\State Let $R$ be a random variable with probability mass function $P_{R}$.
\For{$k = 0, 1, \dots, R-1$}
    \For{$i = 1, \dots, N$}
        \State (i) Pick $S_{k}$ i.i.d. realizations $\{\xi^{k}_{i, l}\}_{l=1}^{S_{k}}$ of $\xi$ and realizations $\{ v^{k}_{i,l} \}_{l=1}^{S_{k}}$ of $v_{i}\in \eta \mathbb{S}_{n_{i}}$;
        \State (ii) Call Algorithm \ref{hierarchical-SA} twice to obtain inexact solutions $y^{\epsilon_{k}}_{i}(x^{k}_{i}+v^{k}_{i,l})$ and $y^{\epsilon_{k}}_{i}(x^{k}_{i}-v^{k}_{i,l})$;
        \State (iii) Compute $\Delta^{k, h}_{i, S_{k}, \epsilon_{k}}$, $\Delta^{k, m}_{i, S_{k}}$, and then compute $\Delta^{k, f}_{i, S_{k}, \epsilon_{k}}$;
        \State (iv) Update $x^{k+1}_{i} = \Pi_{X_{i}} [ x^{k}_{i} - \gamma_{k} \Delta^{k, f}_{i, S_{k}, \epsilon_{k}} ]$;
    \EndFor
\EndFor
\State \textbf{Output}: $x^{R}$ as final estimate.
\end{algorithmic}
\end{algorithm}

\begin{algorithm}
\caption{~SA scheme} \label{hierarchical-SA}
\begin{algorithmic}[1]
\State \textbf{Input}: An arbitrary $y^{0}_{i}\in X_{i}$, $\hat{x}^{k}_{i}\in X_{i} + \eta \mathbb{B}_{n_{i}}$, stepsize sequence $\alpha_{t} = \tfrac{\alpha_{0}}{t+\Gamma}$ where $\alpha_{0} > \tfrac{1}{2\mu_{i}}$ and $\Gamma > 0$, and iteration times $t_{k}$.
\For{$t = 0, 1, \dots, t_{k}-1$}
    \State (i) Generate $\Tilde{F}_{i}(\hat{x}^{k}_{i}, y^{t}_{i}, \xi^{t})$ where $\xi^{t}$ denotes the i.i.d. random realization of $\xi$.
    \State (ii) Update $y^{t+1}_{i} = \Pi_{Y_{i}} [ y^{t}_{i} - \alpha_{t}\Tilde{F}_{i}(\hat{x}^{k}_{i}, y^{t}_{i}, \xi^{t}) ]$;
\EndFor
\State \textbf{Output}: $y_{i}^{\epsilon_{k}}(\hat{x}^{k}_{i}) \triangleq y^{t_{k}}_{i}$ as final estimate.
\end{algorithmic}
\end{algorithm}

\begin{lemma}\label{hierarchical-lower-level-SA-lemma}
    Consider SA scheme (Algorithm \ref{hierarchical-SA}) for solving the lower-level SVI problem \eqref{lower-SVI}. Suppose that assumption $\mathrm{(D6)}$ holds. Given $k > 0$ and $\hat{x}^{k}_{i} \in X_{i}$, let $y_{i}(\hat{x}^{k}_{i})$ denote the unique exact solution of \eqref{lower-SVI} for given $\hat{x}^{k}_{i}\in X_{i} + \eta \mathbb{B}_{n_{i}}$ and $y^{\epsilon_{k}}_{i}(\hat{x}^{k}_{i})$ be the output of Algorithm \ref{hierarchical-SA}. Then for any $k\geq 0$, we have that
    \begin{equation}\label{SA-inexactness}
        \mathbb{E}[\| y^{\epsilon_{k}}_{i}(\hat{x}^{k}_{i}) - y_{i}(\hat{x}^{k}_{i}) \|^{2} \mid \hat{x}^{k}_{i}] \leq \epsilon_{k} \triangleq \frac{c_{\max}}{t_{k}+\Gamma},
    \end{equation}
    where $\Gamma>0$ and $c_{\max} \triangleq \max_{i\in [N]} \left\{ \max \left\{ \frac{(c^{2}_{F_{i}}+v^{2}_{i})\alpha^{2}_{0}}{2\mu_{i}\alpha_{0}-1}, \Gamma\sup_{y_{i}\in X_{i}}\| y^{0}_{i} - y_{i} \|^{2} \right\} \right\} > 0$.
\end{lemma}
\begin{proof}
    The proof can be found in \cite[Theorem 2]{cui-shanbhag-yousefian-2023}. We omit the proof here.
\end{proof}

To facilitate our subsequent analysis, we define the unbiased estimator $\Delta^{k, h}_{i, l}$ as
\begin{equation}\label{exact-two-point-hi}
    \Delta^{k, h}_{i, l} \triangleq \frac{n_{i}}{2\eta}(\Tilde{h}_{i}(x^{k}_{i}+v^{k}_{i,l}, y_{i}(x^{k}_{i}+v^{k}_{i,l}), \xi^{k}_{i,l})-\Tilde{h}_{i}(x^{k}_{i}-v^{k}_{i,l}, y_{i}(x^{k}_{i}-v^{k}_{i,l}), \xi^{k}_{i,l}))\frac{v^{k}_{i,l}}{\|v^{k}_{i,l}\|}.
\end{equation}
Before proceeding, we show the following biasedness and bounded second moment properties.

\begin{lemma}\label{RS-RSG-bound-biasedness-second-moment}
    Suppose that assumptions $\mathrm{(D2)}$-$\mathrm{(D5)}$ hold. For the inexact mini-batch gradient estimator $\Delta^{k, f}_{i, l, \epsilon_{k}} \triangleq \Delta^{k, h}_{i, l, \epsilon_{k}} + \Delta^{k, m}_{i, l}$, we have the following: \\
    \emph{(i)} We have $\| \mathbb{E}[ \Delta^{k, f}_{i, l, \epsilon_{k}} - \nabla_{x_{i}} f^{\eta}_{\mathrm{H}, i}(x^{k}_{i}, x^{k}_{-i}) \mid x^{k}] \| \leq \mu_{\mathrm{H}, k} \triangleq \eta^{-1} n_{\max} L^{y}_{\max} \sqrt{\epsilon_{k}}$, where $\epsilon_{k}$ is defined in \eqref{SA-inexactness}. \\
    \emph{(ii)}  We have that $\mathbb{E}[ \| \Delta^{k, f}_{i, l, \epsilon_{k}} - \nabla_{x_{i}} f^{\eta}_{\mathrm{H}, i}(x^{k}_{i}, x^{k}_{-i}) \|^{2} \mid x^{k}] \leq (\sigma_{\mathrm{H}, k})^{2} \triangleq \frac{4n^{2}_{\max}(L^{y}_{\max})^{2}\epsilon_{k}}{\eta^{2}} + 64\sqrt{2\pi}(L_{\max})^{2}n_{\max} + 2\sigma_{m}^{2}$, where $\epsilon_{k}$ is defined in \eqref{SA-inexactness}.
\end{lemma}
\begin{proof}
    (i) Since $\{ \xi^{k}_{i,l} \}_{l=1}^{S_{k}}$ and $\{ v^{k}_{i,l} \}_{l=1}^{S_{k}}$ are i.i.d. realizations, we have that
    \begin{align}\label{RS-RSG-bound-biasedness-second-moment-eqn1}
        & \| \mathbb{E} [ \Delta^{k, f}_{i, l, \epsilon_{k}} - \nabla_{x_{i}} f^{\eta}_{\mathrm{H}, i}(x^{k}_{i}, x^{k}_{-i}) \mid x^{k} ] \| \notag \\
        &= \| \mathbb{E} [ \Delta^{k,h}_{i,l,\epsilon_{k}} + \Delta^{k,m}_{i,l} - \nabla_{x_{i}} h^{\eta}_{i}(x^{k}_{i}, y_{i}(x^{k}_{i})) - \nabla_{x_{i}} m_{i}(x^{k}_{i}, x^{k}_{-i}) \mid x^{k} ] \| \notag \\
        &= \| \mathbb{E} [ \Delta^{k,h}_{i,l,\epsilon_{k}} - \nabla_{x_{i}} h^{\eta}_{i}(x^{k}_{i}, y_{i}(x^{k}_{i})) \mid x^{k} ] \| \notag \\
        &= \| \mathbb{E} [ \Delta^{k,h}_{i,l,\epsilon_{k}} - \Delta^{k,h}_{i,l} + \Delta^{k,h}_{i,l} - \nabla_{x_{i}} h^{\eta}_{i}(x^{k}_{i}, y_{i}(x^{k}_{i})) \mid x^{k} ] \| = \| \mathbb{E} [ \Delta^{k,h}_{i,l,\epsilon_{k}} - \Delta^{k,h}_{i,l} \mid x^{k} ] \|,
    \end{align}
    where the exact unbiased estimator $\Delta^{k,h}_{i,l}$ is defined in \eqref{exact-two-point-hi}. By invoking Jensen's inequality and Lipschitz continuity assumption $\mathrm{(D5)}$, we may deduce that
    \begin{align}\label{RS-RSG-bound-biasedness-second-moment-eqn2}
        & \| \mathbb{E} [ \Delta^{k,h}_{i,l,\epsilon_{k}} - \Delta^{k,h}_{i,l} \mid x^{k} ] \| \leq \mathbb{E} [ \|\Delta^{k,h}_{i,l,\epsilon_{k}} - \Delta^{k,h}_{i,l} \| \mid x^{k} ] \notag \\
        &\overset{\mathrm{(D5)}}{\leq} \left( \frac{n_{i}}{2\eta} \right) \left( L^{y}_{i} \mathbb{E}[ \|y^{\epsilon_{k}}_{i}(x^{k}_{i}+v^{k}_{i,l}) - y_{i}(x^{k}_{i}+v^{k}_{i,l})\| \mid x^{k}] + L^{y}_{i} \mathbb{E}[ \|y^{\epsilon_{k}}_{i}(x^{k}_{i}-v^{k}_{i,l}) - y_{i}(x^{k}_{i}-v^{k}_{i,l})\| \mid x^{k}] \right) \notag \\
        &\overset{\text{Lemma }\ref{hierarchical-lower-level-SA-lemma}}{\leq} \left( \frac{n_{i}}{2\eta} \right) ( L^{y}_{i}\sqrt{\epsilon_{k}} + L^{y}_{i}\sqrt{\epsilon_{k}}) = \left( \frac{n_{i}}{\eta} \right) L^{y}_{i}\sqrt{\epsilon_{k}} \leq \left( \frac{n_{\max}}{\eta} \right) L^{y}_{\max}\sqrt{\epsilon_{k}},
    \end{align}
    where $L^{y}_{\max} = \max\limits_{i\in [N]} L^{y}_{i}$ and $n_{\max} = \max\limits_{i\in [N]} n_{i}$. By combining \eqref{RS-RSG-bound-biasedness-second-moment-eqn1} with \eqref{RS-RSG-bound-biasedness-second-moment-eqn2}, we complete the proof.

    \noindent (ii) We have that
    \begin{align}\label{RS-RSG-bound-biasedness-second-moment-eqn3}
        &\mathbb{E}[ \| \Delta^{k,h}_{i,l,\epsilon_{k}} + \Delta^{k,m}_{i,l} - \nabla_{x_{i}} f^{\eta}_{\mathrm{H}, i}(x^{k}_{i}, x^{k}_{-i}) \|^{2} \mid x^{k}] \notag \\
        &= \mathbb{E}[ \| \Delta^{k,h}_{i,l,\epsilon_{k}} + \Delta^{k,m}_{i,l} - \nabla_{x_{i}} h^{\eta}_{i}(x^{k}_{i}, y_{i}(x^{k}_{i})) - \nabla_{x_{i}} m_{i}(x^{k}_{i}, x^{k}_{-i}) \|^{2} \mid x^{k}] \notag \\
        &\leq 2\mathbb{E}[ \| \Delta^{k,h}_{i,l,\epsilon_{k}} - \nabla_{x_{i}} h^{\eta}_{i}(x^{k}_{i}, y_{i}(x^{k}_{i})) \|^{2} \mid x^{k}] + 2\mathbb{E}[ \| \Delta^{k,m}_{i,l} - \nabla_{x_{i}} m_{i}(x^{k}_{i}, x^{k}_{-i}) \|^{2} \mid x^{k}] \notag \\
        &\overset{\mathrm{(D3)}}{\leq} \underbrace{4\mathbb{E}[ \| \Delta^{k,h}_{i,l,\epsilon_{k}} - \Delta^{k,h}_{i,l} \|^{2} \mid x^{k}]}_{\mathrm{(i)}} + \underbrace{4\mathbb{E}[ \| \Delta^{k,h}_{i,l} - \nabla_{x_{i}} h^{\eta}_{i}(x^{k}_{i}, y_{i}(x^{k}_{i})) \|^{2} \mid x^{k}]}_{\mathrm{(ii)}} + 2\sigma_{m}^{2}.
    \end{align}
    By the Lipschitz continuity assumption $\mathrm{(D5)}$, we have
    \begin{align}\label{RS-RSG-bound-biasedness-second-moment-eqn4}
        \mathrm{(i)} &\leq \frac{2n^{2}_{i}}{\eta^{2}} \Big[ (L^{y}_{i})^{2}\mathbb{E}[\|y^{\epsilon_{k}}_{i}(x^{k}_{i}+v^{k}_{i,l})-y_{i}(x^{k}_{i}+v^{k}_{i,l})\|^{2} \mid x^{k}] + (L^{y}_{i})^{2}\mathbb{E}[\|y^{\epsilon_{k}}_{i}(x^{k}_{i}-v^{k}_{i,l})-y_{i}(x^{k}_{i}-v^{k}_{i,l})\|^{2} \mid x^{k}] \Big] \notag \\
        &\overset{\text{Lemma }\ref{hierarchical-lower-level-SA-lemma}}{\leq} \frac{2n^{2}_{i}}{\eta^{2}} (2(L^{y}_{i})^{2}\epsilon_{k}) = \frac{4n^{2}_{i}(L^{y}_{i})^{2}\epsilon_{k}}{\eta^{2}} \leq \frac{4n^{2}_{\max}(L^{y}_{\max})^{2}\epsilon_{k}}{\eta^{2}}.
    \end{align}
    By \cite[Lemma 2.4-(v)]{marrinan-shanbhag-yousefian-2026} and the fact that $\mathbb{E}[ \Delta^{k,h}_{i,l} \mid x^{k}] = \nabla_{x_{i}}\mathbb{E}[\Tilde{h}^{\eta}_{i}(x^{k}_{i}, y_{i}(x^{k}_{i}), \xi)]$, we obtain
    \begin{align}\label{RS-RSG-bound-biasedness-second-moment-eqn5}
        \mathrm{(ii)} &= 4\mathbb{E}[\| \Delta^{k,h}_{i,l} \|^{2} - 2(\Delta^{k,h}_{i,l})^{\top}\nabla_{x_{i}} h^{\eta}_{i}(x^{k}_{i}, y_{i}(x^{k}_{i})) + \|\nabla_{x_{i}} h^{\eta}_{i}(x^{k}_{i}, y_{i}(x^{k}_{i}))\|^{2} \mid x^{k}] \notag \\
        &\leq 4\mathbb{E}[\| \Delta^{k,h}_{i,l} \|^{2} \mid x^{k}] \leq 64\sqrt{2\pi}(L_{i})^{2}n_{i} \leq 64\sqrt{2\pi}L^{2}_{\max}n_{\max},
    \end{align}
    where $L_{\max} \triangleq \max_{i\in [N]} L_{i}$. By combining \eqref{RS-RSG-bound-biasedness-second-moment-eqn3}, \eqref{RS-RSG-bound-biasedness-second-moment-eqn4}, and \eqref{RS-RSG-bound-biasedness-second-moment-eqn5}, we arrive the desired conclusion.
\end{proof}

With Lemma \ref{RS-RSG-bound-biasedness-second-moment} and Theorem \ref{biased-RSG-convergence} proven, we are ready to establish the convergence of the biased RS-RSG scheme. Similar to Theorem \ref{RS-RSG-main-theorem}, we need to consider both $\mathcal{SZO}$ and $\mathcal{SFO}$. If the $k$th inexactness $\epsilon_{k}$ defined in \eqref{SA-inexactness} satisfies $\epsilon_{k} \leq \epsilon^{\mathrm{up}}$ for some $\epsilon^{\mathrm{up}} > 0$, then $\sigma_{\mathrm{H}, k}$ in Lemma \ref{RS-RSG-bound-biasedness-second-moment}-(ii) can be upper bounded by $\sigma_{\mathrm{H}}$, defined as
\begin{equation}\label{b-RS-RSG-sigma}
    \sigma_{\mathrm{H}} \triangleq \sqrt{ \frac{4n^{2}_{\max}(L^{y}_{\max})^{2}\epsilon^{\mathrm{up}}}{\eta^{2}} + 64\sqrt{2\pi}L^{2}_{\max}n_{\max} + 2\sigma_{m}^{2} }.
\end{equation}

\begin{theorem}[Convergence of biased RS-RSG]\label{biased-RS-RSG-main-theorem}
    Consider the stochastic $N$-player hierarchical game \eqref{hierarchical-games} and its randomized smoothed game \eqref{hierarchical-games-RS}. Suppose that Assumption $\mathrm{D}$ holds. Suppose that we choose $\gamma_{k} = 1/(2L^{\eta})$ and $P_{R} = 1/T$ for all $k = 1, \dots, T$. Given a sufficiently large number $M_{0}$ of $\mathcal{SZO}$ calls and a sufficiently large number $M_{1}$ of $\mathcal{SFO}$ calls. We define $M \triangleq \min\{ M_{0}/2, M_{1} \}$. If the batch size $S_{k}$ is fixed as
    \begin{equation}\label{b-RS-RSG-Sk}
        S_{k} = S \triangleq \left\lceil \frac{\sigma_{\mathrm{H}}\sqrt{6M}}{4L^{\eta}_{\mathrm{H}}D^{\eta}_{\mathrm{H}}} \right\rceil,
    \end{equation}
    where $D^{\eta}_{\mathrm{H}} = ((P^{\eta}_{\mathrm{H}, \max}-P^{\eta}_{\mathrm{H}, \min})/L^{\eta}_{\mathrm{H}})^{1/2}$ and $\sigma_{\mathrm{H}}$ is defined in \eqref{b-RS-RSG-sigma}. Then we have
    \begin{equation}\label{biased-RS-RSG-main-theorem-eqn3}
        \begin{aligned}
            \mathbb{E}[\| G^{\eta}_{\mathrm{H}, \gamma_{R}}(x^{R}) \|^{2}] &\leq \frac{48(L^{\eta}_{\mathrm{H}})^{2}(D^{\eta}_{\mathrm{H}})^{2}N}{M} + \frac{60\sqrt{6}(L^{\eta}_{\mathrm{H}})^{2}(D^{\eta}_{\mathrm{H}})^{2}N\sigma_{\mathrm{H}} + 18\sqrt{6}N^{2}\sigma_{\mathrm{H}}\sum_{k=1}^{T}(\mu_{\mathrm{H}, k})^{2}}{L^{\eta}_{\mathrm{H}} D^{\eta}_{\mathrm{H}}\sqrt{M}},
        \end{aligned}
    \end{equation}
    where the residual is defined as $G^{\eta}_{\mathrm{H}, \gamma}(x) \triangleq \tfrac{1}{\gamma} (x-\Pi_{X}[x-\gamma F^{\eta}_{\mathrm{H}}(x)])$ with $F^{\eta}_{\mathrm{H}}(x) \triangleq (\nabla_{x_{i}} f^{\eta}_{\mathrm{H}, i}(x_{i}, x_{-i}))_{i=1}^{N}$.
\end{theorem}
\begin{proof}
     The proof proceeds similarly to that of Theorem \ref{biased-RSG-convergence}. Consequently, we may obtain \eqref{biased-RS-RSG-useful} with $L$, $D$, $\sigma_k^2$, and $\mu_k^2$ replaced by $L^{\eta}_{\mathrm{H}}$, $D^{\eta}_{\mathrm{H}}$, $\sigma^{2}_{\mathrm{H}, k}$, and $\mu^{2}_{\mathrm{H}, k}$, respectively. We can see that the biased RS-RSG algorithm can perform at most $T = \min \{ \lfloor (M_{0}/2)/(SN) \rfloor, \lfloor M_{1}/(SN) \rfloor \}$ iterations. Similar to the proof of Theorem \ref{RS-RSG-main-theorem}, we have the lower bound
     \begin{equation*}
         T \geq \min \{ (M_{0}/2)/(2SN), M_{1}/(2SN) \} = M/(2SN)
     \end{equation*}
     and the upper bound
     \begin{equation*}
         T \leq \min \{ (M_{0}/2)/(SN), M_{1}/(SN) \} = M/(SN).
     \end{equation*}
     Therefore, we have the same lower and upper bounds on $T$ as Theorem \ref{biased-RSG-convergence}. The remainder of the proof proceeds similarly by additionally invoking Lemma \ref{RS-RSG-bound-biasedness-second-moment}.
\end{proof}

Under the Lipschitz continuity on each Clarke subdifferential 
$\partial^{C}_{x_i} h_i(x_i,y_i(x_i))$, we obtain an $\mathcal{O}(\eta^{2})$ CNE approximation of \eqref{hierarchical-games}, similar to the result in Theorem \ref{rs-approx-main-thm}. Based on Theorem \ref{biased-RS-RSG-main-theorem}, we next derive the complexity guarantees for biased RS-RSG.

\subsection{Complexity Analysis of Biased RS-RSG}\label{Sec-5.2}

In this subsection, we derive the overall iteration ans sample complexities of the biased RS-RSG scheme based on Theorem \ref{biased-RS-RSG-main-theorem}.

\begin{theorem}[Overall complexities of biased RS-RSG]\label{overall-complexities-b-RS-RSG}
    Consider the biased RS-RSG scheme for solving the stochastic $N$-player hierarchical game \eqref{hierarchical-games}. Consider the same setting as in Theorem \ref{biased-RS-RSG-main-theorem}. At the $k$th upper-level iteration, let the number of lower-level SA iterations be $t_{k} = \lceil (k+1)^{1+\delta} \rceil$ for any arbitrarily small $\delta > 0$. Suppose that $\eta\in (0, 1]$. Then the complexities to reach a point $x^{R}$ such that $\mathbb{E}[\| G^{\eta}_{\mathrm{H}, \gamma_{R}}(x^{R}) \|] \leq \epsilon$ are as follows: \\
    \emph{(i)} The upper-level sample complexity is $\mathcal{O}(L^{4}_{\max} n^{13/2}_{\max} N^{5}\eta^{-7}\epsilon^{-4})$. \\
    \emph{(ii)} The upper-level iteration complexity is $\mathcal{O}(L^{3}_{\max}n^{7/2}_{\max}N^{2}\eta^{-4}\epsilon^{-2})$. \\
    \emph{(iii)} The lower-level sample and iteration complexities are
    $\mathcal{O}(L^{17/2+3\delta}_{\max} n^{45/4+7\delta/2}_{\max}$ $N^{29/4+2\delta}\eta^{-(12+4\delta)}\epsilon^{-(6+2\delta)})$.
\end{theorem}
\begin{proof}
    (i) We first show that by setting $t_{k} = \lceil k^{1+\delta} \rceil$, it yields $\sum_{k=1}^{T} \epsilon_{k} < c_{\max}(1+1/\delta)$, where $c_{\max}$ is defined in Lemma \ref{hierarchical-lower-level-SA-lemma}. Indeed, we have
    \begin{equation}\label{b-RS-RSG-complexity-eqn1}
        \sum_{k=1}^{T} \epsilon_{k} < \sum_{k=1}^{\infty} \frac{c_{\max}}{k^{1+\delta}} \leq c_{\max} \left( 1 + \int_{1}^{\infty} \frac{1}{x^{1+\delta}} dx \right) = c_{\max} ( 1 + 1/\delta ).
    \end{equation}
    It follows that $\sum_{k=1}^{T}(\mu_{\mathrm{H}, k})^{2} < \infty$ and we have
    \begin{align}\label{b-RS-RSG-complexity-eqn2}
        \mathbb{E}[\| G^{\eta}_{\mathrm{H}, \gamma_{R}}(x^{R}) \|^{2}] \leq \mathcal{O} \left( \frac{ (L^{\eta}_{\mathrm{H}})^{2} (D^{\eta}_{\mathrm{H}})^{2} N^{2} \sigma_{\mathrm{H}} }{L^{\eta}_{\mathrm{H}} D^{\eta}_{\mathrm{H}} \sqrt{M}} \right) = \mathcal{O} \left( \frac{ L^{\eta}_{\mathrm{H}} D^{\eta}_{\mathrm{H}} N^{2} \sigma_{\mathrm{H}} }{\sqrt{M}} \right).
    \end{align}
    It follows from \eqref{b-RS-RSG-sigma} that
    \begin{equation}\label{b-RS-RSG-complexity-eqn3}
        \sigma_{\mathrm{H}} \leq \mathcal{O} \left( \frac{L^{y}_{\max}n_{\max}}{\eta} + L_{\max} n^{1/2}_{\max} \right) \leq \mathcal{O} \left( \frac{n_{\max}}{\eta} (L^{y}_{\max} + L_{\max}) \right)
    \end{equation}
    since $\eta\in (0, 1]$ and $n_{\max}\geq 1$. By Lemma \ref{RS-RSG-bound-biasedness-second-moment}-(i) and \eqref{b-RS-RSG-complexity-eqn1}, we may obtain
    \begin{equation}\label{b-RS-RSG-complexity-eqn4}
        \sum_{k=1}^{T}(\mu_{\mathrm{H}, k})^{2} \leq \mathcal{O} \left( \frac{ n^{2}_{\max} (L^{y}_{\max})^{2} }{\eta^{2}} \right).
    \end{equation}
    For the product $L^{\eta}_{\mathrm{H}} D^{\eta}_{\mathrm{H}}$, similar to the proof of Theorem \ref{RS-RSG-complexity}, we have
    \begin{equation}\label{b-RS-RSG-complexity-eqn5}
        L^{\eta}_{\mathrm{H}} D^{\eta}_{\mathrm{H}} \leq \mathcal{O} \left( \frac{ L_{\max} n^{1/4}_{\max} N^{1/2} }{\eta^{1/2}} \right).
    \end{equation} 
    By plugging \eqref{b-RS-RSG-complexity-eqn3}, \eqref{b-RS-RSG-complexity-eqn4}, and \eqref{b-RS-RSG-complexity-eqn5} into \eqref{b-RS-RSG-complexity-eqn2}, we have the following bound
    \begin{equation*}
         \mathbb{E}[\| G^{\eta}_{\mathrm{H}, \gamma_{R}}(x^{R}) \|^{2}] \leq \mathcal{O} \left( \frac{ L_{\max} (L^{y}_{\max})^{2} (L^{y}_{\max}+L_{\max}) n^{13/4}_{\max} N^{5/2} }{ \eta^{7/2} \sqrt{M} } \right) = \mathcal{O} \left( \frac{ L^{2}_{\max} n^{13/4}_{\max} N^{5/2} }{ \eta^{7/2} \sqrt{M} } \right).
    \end{equation*}
    In order to ensure that $\mathbb{E}[\| G^{\mathrm{HI}, \eta}_{\gamma_{R}}(x^{R}) \|] \leq \epsilon$ holds, we need
    \begin{equation}\label{b-RS-RSG-complexity-eqn6}
        M = \mathcal{O} \left( \frac{ L^{4}_{\max} n^{13/2}_{\max} N^{5} }{ \eta^{7}\epsilon^{-4} } \right).
    \end{equation}
    By setting minimal $M_{0} = 2M$ and $M_{1} = M$, it can be seen that for any $\epsilon > 0$, we need at least $M_{0} + M_{1} = 3M = \mathcal{O}(L^{4}_{\max} n^{13/2}_{\max} N^{5}\eta^{-7}\epsilon^{-4})$ samples.

    \noindent (ii) From the upper and lower bounds on $T$ derived in the proof of Theorem \ref{biased-RS-RSG-main-theorem} and the choice of $S$ in \eqref{b-RS-RSG-Sk}, we may obtain
    \begin{equation*}
        M/(2SN) \leq T \leq M/(SN) \implies T = \mathcal{O} \left( \frac{M}{SN} \right) = \mathcal{O} \left( \frac{ L^{\eta}_{\mathrm{H}} D^{\eta}_{\mathrm{H}} \sqrt{M} }{\sigma_{\mathrm{H}} N} \right).
    \end{equation*}
    By invoking \eqref{b-RS-RSG-complexity-eqn5}, \eqref{b-RS-RSG-complexity-eqn6}, together with the fact that $\sigma_{\mathrm{H}}\geq \sigma_{m}$ from \eqref{b-RS-RSG-sigma}, we may derive that
    \begin{equation}\label{b-RS-RSG-complexity-eqn6.1}
        T = \mathcal{O} \left( \frac{ L^{3}_{\max} n^{7/2}_{\max} N^{2} }{ \eta^{4} \epsilon^{-2} } \right).
    \end{equation}

    \noindent (iii) We know from Algorithm \ref{hierarchical-SA} that the lower-level sample and iteration complexities for solving \eqref{lower-SVI} are the same. Both of them are given by
    \begin{align}\label{b-RS-RSG-complexity-eqn6.2}
        & 2 N S \sum_{k=0}^{T-1} t_{k} = 2 N S \sum_{k=0}^{T-1} \lceil (k+1)^{1+\delta} \rceil = 2 N S \sum_{k=1}^{T} \lceil k^{1+\delta} \rceil \leq 4 N S \sum_{k=1}^{T} k^{1+\delta} \leq 4 N S \int_{1}^{T+1} x^{1+\delta} dx \notag \\
        &= 4 N S \left( \frac{1}{2+\delta} x^{2+\delta} \right) \Big\vert^{T+1}_{1} \leq \frac{4NS}{2+\delta} (T+1)^{2+\delta} \leq \frac{4NS}{2+\delta} (2T)^{2+\delta}.
    \end{align}
    We claim that
    \begin{equation}\label{b-RS-RSG-complexity-eqn7}
        S = \left\lceil \frac{\sigma_{\mathrm{H}}\sqrt{6M}}{4L^{\eta}_{\mathrm{H}}D^{\eta}_{\mathrm{H}}} \right\rceil  \leq \mathcal{O} \left( \frac{L^{4}_{\max}n^{17/4}_{\max}N^{9/4}}{(L_{\max})^{3/2}(L^{y}_{\max})^{2}\eta^{4}\epsilon^{2}} \right) = \mathcal{O} \left( \frac{L^{5/2}_{\max}n^{17/4}_{\max}N^{9/4}}{\eta^{4}\epsilon^{2}} \right).
    \end{equation}
    Indeed, similar to the proof of Theorem \ref{RS-RSG-complexity}, we may derive from \eqref{RS-RSG-complexity-eqn3} and \eqref{RS-RSG-complexity-eqn3.5} that
    \begin{equation*}
        L^{\eta}_{\mathrm{H}}D^{\eta}_{\mathrm{H}} = (L^{\eta}_{\mathrm{H}})^{1/2} (P^{\eta}_{\mathrm{H}, \max} - P^{\eta}_{\mathrm{H}, \min}) \geq \frac{L^{1/2}_{\max}N^{1/4}}{\eta^{1/2}} (P_{\mathrm{H}, m}(x^{\max}) - P_{\mathrm{H}, m}(x^{\min})),
    \end{equation*}
    where $P_{\mathrm{H}, m}(x) \triangleq P_{\mathrm{H}}(x) - \sum_{i=1}^{N} h_{i}(x_{i}, y_{i}(x_{i}))$, implying that
    \begin{equation}\label{b-RS-RSG-complexity-eqn8}
        \frac{1}{L^{\eta}_{\mathrm{H}}D^{\eta}_{\mathrm{H}}} \leq \mathcal{O} \left( \frac{\eta^{1/2}}{L^{1/2}_{\max}N^{1/4}} \right).
    \end{equation}
    By combing \eqref{b-RS-RSG-complexity-eqn3}, \eqref{b-RS-RSG-complexity-eqn6}, and \eqref{b-RS-RSG-complexity-eqn8}, it leads to \eqref{b-RS-RSG-complexity-eqn7}. By plugging \eqref{b-RS-RSG-complexity-eqn6.1} and \eqref{b-RS-RSG-complexity-eqn7} into \eqref{b-RS-RSG-complexity-eqn6.2}, it follows that
    \begin{equation*}
        2 N S \sum_{k=0}^{T-1} t_{k} \leq \mathcal{O} \left( \frac{ L^{17/2+3\delta}_{\max} n^{45/4+7\delta/2}_{\max} N^{29/4+2\delta} }{\eta^{12+4\delta}\epsilon^{6+2\delta}} \right),
    \end{equation*}
    which completes the proof.
\end{proof}

\section{Numerical Experiments}\label{Sec-6}

In this section, we evaluate the performances of RS-RSG and biased RS-RSG schemes through two illustrative examples introduced in subsection \ref{two-motivating-examples}.

\subsection{Stochastic Nonconvex Nonsmooth Potential Cournot Games}

In this subsection, we consider the stochastic $N$-player Nash-Cournot game \eqref{NC-i} which admits the following potential function
\begin{equation*}
    P(x) = \sum_{i=1}^{N} \mathbb{E}[\Tilde{c}_{i}(\xi)]g_{i}(x_{i}) - \mathbb{E}[a(\xi)]\sum_{i=1}^{N}x_{i} + \mathbb{E}[b(\xi)]\sum_{i=1}^{N}x^{2}_{i} + \mathbb{E}[b(\xi)]\sum_{1\leq i< j\leq N}x_{i}x_{j}.
\end{equation*}
In this example, we consider $N = 6$ players. Suppose that each $X_{i} = [0, 12]$, $\xi \sim \mathrm{U}\:[0, 1]$, $\Tilde{c}_{i}(\xi) = (5 + i/(8N))\xi$ for any $i\in [N]$. The random variables $a(\xi)$ and $b(\xi)$ are given by $a(\xi) = 4\xi$ and $b(\xi) = 0.02\xi$, respectively. Suppose that $g_{i}(x_{i})$ is an increasing piecewise-linear concave function defined by $g_{i}(x_{i}) = \min\{ x_{i}, \frac{1}{2}x_{i} + 2 \}$. We set the initialization $x^{0} = 12\mathbf{e}_{6}$ and sample budget $M = 1\mathbf{e}8$. 

In the numerical experiment, we test three smoothing parameters $\eta_{1} = 0.3$, $\eta_{2} = 0.5$, and $\eta_{3} = 0.8$. We adopt the expected square residual $\mathbb{E}[\| G^{\eta}_{\gamma_{R}}(x^{R}) \|^{2}]$ as the convergence measure. The convergence in expectation is shown in Figure \ref{RS-RSG-figure}, averaged over $10$ sample paths. Tables \ref{RS-RSG-table} illustrates the required iterations and samples to achieve prescribed residual levels for different $\eta$. We may observe that, although a smaller $\eta > 0$ results in a better approximation (see Theorem \ref{rs-approx-main-thm}), it also requires more iterations and samples.

\begin{figure}[H]
    \centering
    \begin{subfigure}{0.45\textwidth}
        \centering
        \includegraphics[width=\textwidth]{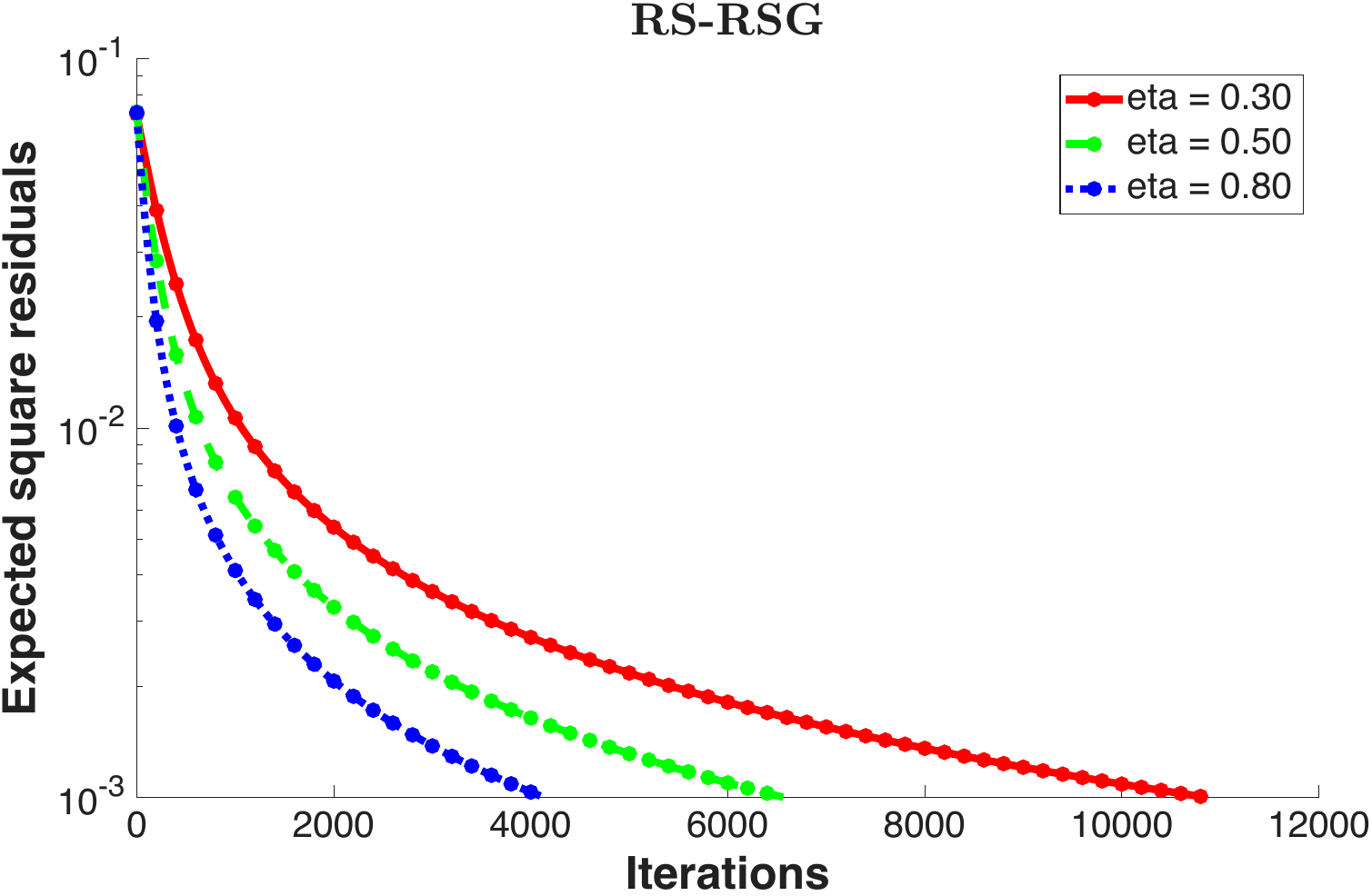}
        \caption{RS-RSG}
        \label{RS-RSG-figure}
    \end{subfigure}
    \hfill
    \begin{subfigure}{0.45\textwidth}
        \centering
        \includegraphics[width=\textwidth]{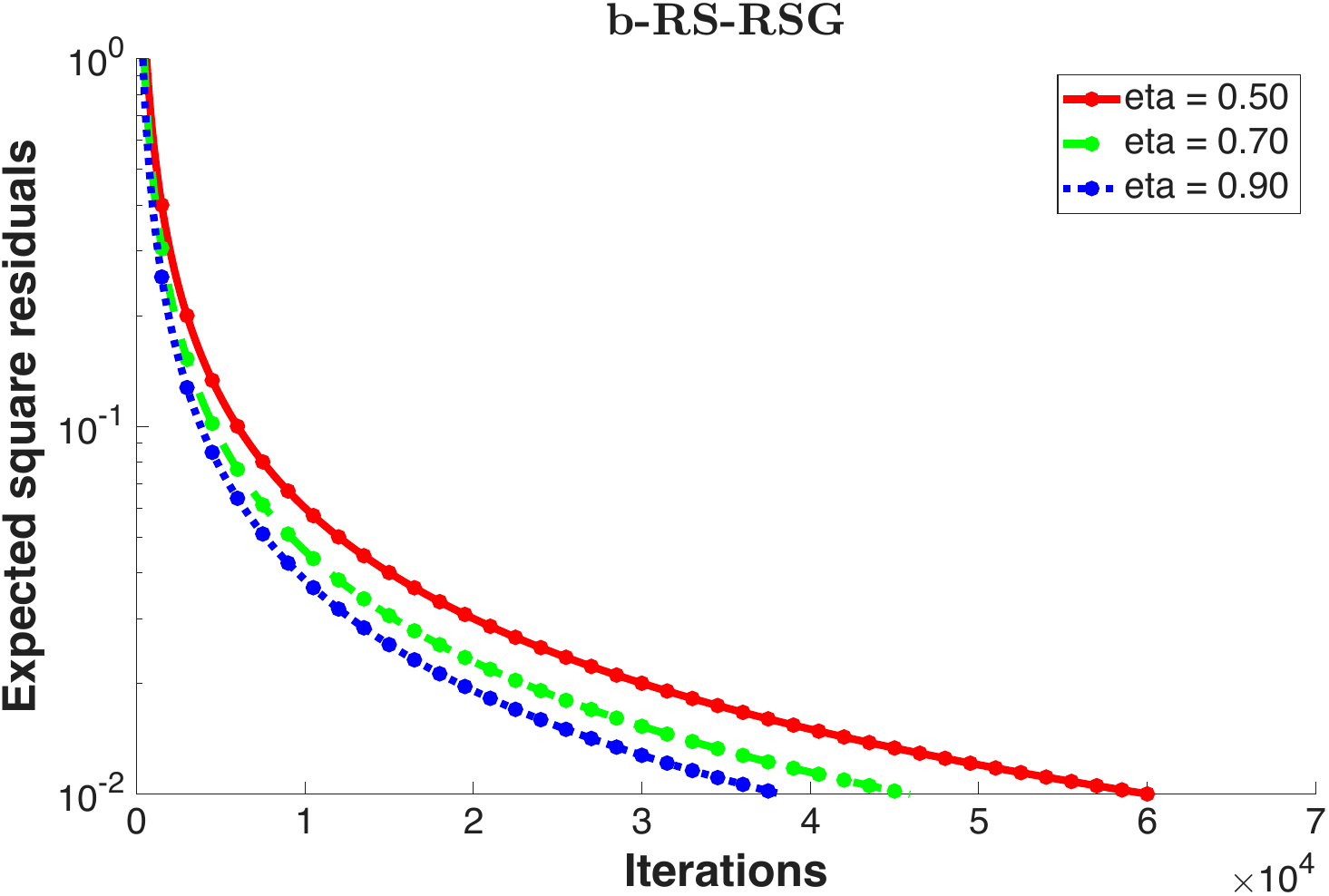}
        \caption{b-RS-RSG}
        \label{b-RS-RSG-figure}
    \end{subfigure}
    \caption{The convergences of RS-RSG and b-RS-RSG.}
    \label{RS-RSG-comparison-figure}
\end{figure}

\begin{table}
\centering
\scriptsize
{\renewcommand{\arraystretch}{1.8}
\begin{tabular}{@{\hspace{10pt}}c@{\hspace{22pt}}c@{\hspace{22pt}}c@{\hspace{22pt}}c@{\hspace{22pt}}c@{\hspace{22pt}}c@{\hspace{10pt}}}
\toprule
\toprule
\small{$\mathbb{E}[\| G^{\eta}_{\gamma_{R}}(x^{R}) \|^{2}]$} & \small{$1.0\times 10^{-2}$} & \small{$7.5\times 10^{-3}$} & \small{$5.0\times 10^{-3}$} & \small{$2.5\times 10^{-3}$} & \small{$1.0\times 10^{-3}$} \\
\midrule
\midrule
\small{$T\:(\eta_{1} = 0.3)$} & \small{$1.07\mathbf{e}3$} & \small{$1.43\mathbf{e}3$} & \small{$2.16\mathbf{e}3$} & \small{$4.34\mathbf{e}3$} & \small{$1.09\mathbf{e}4$} \\
\small{$T\:(\eta_{2} = 0.5)$} & \small{$6.45\mathbf{e}2$} & \small{$8.65\mathbf{e}2$} & \small{$1.30\mathbf{e}3$} & \small{$2.62\mathbf{e}3$} & \small{$6.56\mathbf{e}3$} \\
\small{$T\:(\eta_{3} = 0.8)$} & \small{$4.07\mathbf{e}2$} & \small{$5.45\mathbf{e}2$} & \small{$8.22\mathbf{e}2$} & \small{$1.65\mathbf{e}3$} & \small{$4.13\mathbf{e}3$} \\
\midrule
\midrule
\small{$M\:(\eta_{1} = 0.3)$} & \small{$9.09\mathbf{e}6$} & \small{$1.22\mathbf{e}7$} & \small{$1.84\mathbf{e}7$} & \small{$3.69\mathbf{e}7$} & \small{$9.24\mathbf{e}7$} \\
\small{$M\:(\eta_{2} = 0.5)$} & \small{$7.07\mathbf{e}6$} & \small{$9.48\mathbf{e}6$} & \small{$1.43\mathbf{e}7$} & \small{$2.87\mathbf{e}7$} & \small{$7.19\mathbf{e}7$} \\
\small{$M\:(\eta_{3} = 0.8)$} & \small{$5.62\mathbf{e}6$} & \small{$7.53\mathbf{e}6$} & \small{$1.14\mathbf{e}7$} & \small{$2.28\mathbf{e}7$} & \small{$5.71\mathbf{e}7$} \\
\bottomrule
\bottomrule
\end{tabular}
}
\caption{The required iteration and sample complexities to achieve prescribed residual levels.}
\label{RS-RSG-table}
\end{table}

\subsection{Stochastic Nonconvex Nonsmooth Potential Hierarchical Games}

In this subsection, we consider a special class of stochastic $N$-player two-stage hierarchical game inspired by \cite{cui-shanbhag-staudigl-2025}, but slightly modified to account for uncertainty in the lower-level problem. Suppose that the $i$th leader's problem is given by
\begin{equation*}
    \min_{x_{i}\in X_{i}} ~ C_{i}(x_{i}) - \mathbb{E}[ \Tilde{p}(x_{i} + X_{-i} + y_{i}(x_{i}), \xi) ], ~ \forall i\in [N],
\end{equation*}
and the corresponding follower's problem is parameterized by the leader's decision $x_{i}$, given by
\begin{equation*}
    \min_{y_{i}\in Y_{i}} ~ c_{i}(y_{i}) - \mathbb{E}[\Tilde{p}(x_{i} + y_{i}, \xi)]y_{i},
\end{equation*}
where $C_{i}(\bullet)$ and $c_{i}(\bullet)$ are cost functions of leader $i$ and follower $i$, respectively. For given $\xi$, the linear inverse demand function is defined as $\Tilde{p}(u, \xi) = a(\xi) - b(\xi) u$ for random variables $a(\xi)$ and $b(\xi)$. We can show that the above two-stage stochastic hierarchical game admits the potential function
\begin{equation*}
    P(x) = \sum_{i=1}^{N} C_{i}(x_{i}) + \mathbb{E}[b(\xi)]\sum_{i=1}^{N} x_{i}y_{i}(x_{i}) - \mathbb{E}[a(\xi)]\sum_{i=1}^{N}x_{i} + \mathbb{E}[b(\xi)]\sum_{i=1}^{N}x^{2}_{i} + \mathbb{E}[b(\xi)]\sum_{1\leq i< j\leq N}x_{i}x_{j}.
\end{equation*}
In this example, we consider $N = 4$ leader players. Suppose that each $X_{i} = [0, 20]$, $Y_{i} = [0, 200]$, $C_{i}(x_{i}) \triangleq \mathbb{E}[5+\xi] \log{(x_{i}+1)}$, and $c_{i}(y_{i}) \triangleq \mathbb{E}[1+0.2\xi] y_{i}$ for any $i\in [N]$. Further, we assume that $\xi \sim \mathrm{U}\:[-1, 1]$, $a(\xi) = 2\xi + 8$, and $b(\xi) = 0.01\xi + 0.02$. We set the initialization $x^{0} = 19\mathbf{e}_{4}$ and sample budget $M = 6\mathbf{e}9$. 

In the numerical experiment, we test three smoothing parameters $\eta_{1} = 0.5$, $\eta_{2} = 0.7$, and $\eta_{3} = 0.9$. We adopt the expected square residual $\mathbb{E}[\| G^{\eta}_{\mathrm{H}, \gamma_{R}}(x^{R}) \|^{2}]$ as the convergence measure. We set the number of lower-level iterations to be $t_{k} = \lceil (k+1)^{1+\delta}\rceil$, as in Theorem~\ref{overall-complexities-b-RS-RSG}. The convergence in expectation is shown in Figure \ref{b-RS-RSG-figure}, averaged over $10$ sample paths. Table \ref{b-RS-RSG-table} illustrates the required iteration and sample complexities of biased RS-RSG scheme. We may see that a smaller $\eta$ leads to more iterations and samples again. At the same time, the large Lipschitz continuity constant of this problem necessitates a small stepsize, and consequently the scheme requires on the order of $10^{4}$ iterations to attain a satisfactory level of accuracy.

\begin{table}
\centering
\scriptsize
{\renewcommand{\arraystretch}{1.8}
\begin{tabular}{@{\hspace{10pt}}c@{\hspace{22pt}}c@{\hspace{22pt}}c@{\hspace{22pt}}c@{\hspace{22pt}}c@{\hspace{22pt}}c@{\hspace{10pt}}}
\toprule
\toprule
\small{$\mathbb{E}[\| G^{\mathrm{HI}, \eta}_{\gamma_{R}}(x^{R}) \|^{2}]$} & \small{$1.0\times 10^{-1}$} & \small{$7.5\times 10^{-2}$} & \small{$5.0\times 10^{-2}$} & \small{$2.5\times 10^{-2}$} & \small{$1.0\times 10^{-2}$} \\
\midrule
\midrule
\small{$T\:(\eta_{1} = 0.5)$} & \small{$5.99\mathbf{e}3$} & \small{$7.99\mathbf{e}3$} & \small{$1.20\mathbf{e}4$} & \small{$2.40\mathbf{e}4$} & \small{$5.99\mathbf{e}4$} \\
\small{$T\:(\eta_{2} = 0.7)$} & \small{$4.61\mathbf{e}3$} & \small{$6.14\mathbf{e}3$} & \small{$9.21\mathbf{e}3$} & \small{$1.84\mathbf{e}4$} & \small{$4.61\mathbf{e}4$} \\
\small{$T\:(\eta_{3} = 0.9)$} & \small{$3.81\mathbf{e}3$} & \small{$5.08\mathbf{e}3$} & \small{$7.62\mathbf{e}3$} & \small{$1.52\mathbf{e}4$} & \small{$3.81\mathbf{e}4$} \\
\midrule
\midrule
\small{$M^{\mathrm{up}}\:(\eta_{1} = 0.5)$} & \small{$5.39\mathbf{e}8$} & \small{$7.18\mathbf{e}8$} & \small{$1.08\mathbf{e}9$} & \small{$2.16\mathbf{e}9$} & \small{$5.39\mathbf{e}9$} \\
\small{$M^{\mathrm{up}}\:(\eta_{2} = 0.7)$} & \small{$4.80\mathbf{e}8$} & \small{$5.97\mathbf{e}8$} & \small{$8.96\mathbf{e}8$} & \small{$1.79\mathbf{e}9$} & \small{$4.48\mathbf{e}9$} \\
\small{$M^{\mathrm{up}}\:(\eta_{3} = 0.9)$} & \small{$3.97\mathbf{e}8$} & \small{$5.29\mathbf{e}8$} & \small{$7.94\mathbf{e}8$} & \small{$1.59\mathbf{e}9$} & \small{$3.97\mathbf{e}9$} \\
\midrule
\midrule
\small{$M^{\mathrm{low}}\:(\eta_{1} = 0.5)$} & \small{$3.23\mathbf{e}12$} & \small{$5.74\mathbf{e}12$} & \small{$1.29\mathbf{e}13$} & \small{$5.16\mathbf{e}13$} & \small{$3.23\mathbf{e}14$} \\
\small{$M^{\mathrm{low}}\:(\eta_{2} = 0.7)$} & \small{$2.06\mathbf{e}12$} & \small{$3.67\mathbf{e}12$} & \small{$8.25\mathbf{e}12$} & \small{$3.30\mathbf{e}13$} & \small{$2.06\mathbf{e}14$} \\
\small{$M^{\mathrm{low}}\:(\eta_{3} = 0.9)$} & \small{$1.51\mathbf{e}12$} & \small{$2.69\mathbf{e}12$} & \small{$6.05\mathbf{e}12$} & \small{$2.42\mathbf{e}13$} & \small{$1.51\mathbf{e}14$} \\
\bottomrule
\bottomrule
\end{tabular}
}
\caption{The required iteration complexities to achieve prescribed residual levels.}
\label{b-RS-RSG-table}
\end{table}

\section{Conclusions}\label{Sec-7}

Efficient algorithms for solving stochastic nonconvex nonsmooth games remain largely unexplored. This paper proposes a randomized stochastic gradient scheme, together with its smoothing-enabled and biased variants, for solving stochastic nonconvex nonsmooth potential games. To the best of our knowledge, these findings are novel in the computation of equilibria in stochastic nonconvex nonsmooth games. Several open questions remain for future investigation, such as the development of asynchronous RSG scheme and its biased variant. We hope this work will stimulate researchers' interest in a broader paradigm of stochastic noncooperative games beyond the classical convex and smooth regimes.

\appendix

\section{Proof of Theorem \ref{RSG-convergence}}\label{proof-smooth-RSG}

\begin{proof}
    (i) We define $\delta^{k} \triangleq \Tilde{F}(x^{k}, \xi^{k}) - F(x^{k})$ for any $k\geq 0$ with $i$th component $\delta^{k}_{i} = \Tilde{g}_{i}(x^{k}, \xi_{i}^{k}) - \nabla_{x_{i}}f_{i}(x^{k})$. Since $P$ is $L$-smooth, it follows that
    \begin{align*}
        P(x^{k+1}) &\leq P(x^{k}) + \langle \nabla P(x^{k}), x^{k+1}-x^{k} \rangle + \frac{L}{2}\| x^{k+1}-x^{k} \|^{2} \\
        &= P(x^{k}) - \gamma_{k} \langle \nabla P(x^{k}), \Tilde{G}_{\gamma_{k}}(x^{k}, \xi^{k}) \rangle + \frac{L}{2}(\gamma_{k})^{2} \| \Tilde{G}_{\gamma_{k}}(x^{k}, \xi^{k}) \|^{2} \\
        &= P(x^{k}) - \gamma_{k} \langle F(x^{k}), \Tilde{G}_{\gamma_{k}}(x^{k}, \xi^{k}) \rangle + \frac{L}{2}(\gamma_{k})^{2} \| \Tilde{G}_{\gamma_{k}}(x^{k}, \xi^{k}) \|^{2} \quad \text{(by Proposition \ref{potentiality-integrability})} \\
        &= P(x^{k}) - \gamma_{k} \langle \Tilde{F}(x^{k}, \xi^{k}), \Tilde{G}_{\gamma_{k}}(x^{k}, \xi^{k}) \rangle + \frac{L}{2}(\gamma_{k})^{2} \| \Tilde{G}_{\gamma_{k}}(x^{k}, \xi^{k}) \|^{2} + \gamma_{k} \langle \delta^{k}, \Tilde{G}_{\gamma_{k}}(x^{k}, \xi^{k}) \rangle.
    \end{align*}
    By \cite[Lemma 1]{ghadimi-lan-zhang-2016}, we have
    \begin{equation*}
        \langle \Tilde{F}(x^{k}, \xi^{k}), \Tilde{G}_{\gamma_{k}}(x^{k}, \xi^{k}) \rangle \geq \| \Tilde{G}_{\gamma_{k}}(x^{k}, \xi^{k}) \|^{2}.
    \end{equation*}
    Therefore we obtain that
    \begin{align*}
        P(x^{k+1}) &\leq P(x^{k}) - \gamma_{k} \| \Tilde{G}_{\gamma_{k}}(x^{k}, \xi^{k}) \|^{2} + \frac{L}{2}(\gamma_{k})^{2} \| \Tilde{G}_{\gamma_{k}}(x^{k}, \xi^{k}) \|^{2} + \gamma_{k} \langle \delta^{k}, G_{\gamma_{k}}(x^{k}) \rangle + \gamma_{k} \langle \delta^{k}, \Tilde{G}_{\gamma_{k}}(x^{k}, \xi^{k}) - G_{\gamma_{k}}(x^{k}) \rangle \\
        &\leq P(x^{k}) - (\gamma_{k}-\frac{L}{2}(\gamma_{k})^{2}) \| \Tilde{G}_{\gamma_{k}}(x^{k}, \xi^{k}) \|^{2} + \gamma_{k} \langle \delta^{k}, G_{\gamma_{k}}(x^{k}) \rangle + \gamma_{k} \|\delta^{k}\| \| \Tilde{G}_{\gamma_{k}}(x^{k}, \xi^{k}) - G_{\gamma_{k}}(x^{k}) \| \\
        &\leq P(x^{k}) - (\gamma_{k}-\frac{L}{2}(\gamma_{k})^{2}) \| \Tilde{G}_{\gamma_{k}}(x^{k}, \xi^{k}) \|^{2} + \gamma_{k} \langle \delta^{k}, G_{\gamma_{k}}(x^{k}) \rangle + \gamma_{k} \|\delta^{k}\|^{2},
    \end{align*}
    where the last inequality follows from \cite[Proposition 1]{ghadimi-lan-zhang-2016}. Summing up the above inequalities for $k = 1, \dots, T$ and noting that $0 < \gamma_{k}\leq 1/L < 2/L$ with $0 < \gamma_{k} < 1/L$ for at least one $k$, we obtain
    \begin{equation}\label{RSGR-proof-1}
        \begin{aligned}
            \sum_{k=1}^{T} (\gamma_{k}-L(\gamma_{k})^{2}) \| \Tilde{G}_{\gamma_{k}}(x^{k}, \xi^{k}) \|^{2} &\leq \sum_{k=1}^{T} (\gamma_{k}-\frac{L}{2}(\gamma_{k})^{2}) \| \Tilde{G}_{\gamma_{k}}(x^{k}, \xi^{k}) \|^{2} \\
            &\leq P(x^{1}) - P(x^{T+1}) + \sum_{k=1}^{T} \{ \gamma_{k} \langle \delta^{k}, G_{\gamma_{k}}(x^{k}) \rangle + \gamma_{k} \|\delta^{k}\|^{2} \} \\
            &\leq P_{\text{max}} - P_{\text{min}} + \sum_{k=1}^{T} \{ \gamma_{k} \langle \delta^{k}, G_{\gamma_{k}}(x^{k}) \rangle + \gamma_{k} \|\delta^{k}\|^{2} \},
        \end{aligned}
    \end{equation}
    where $P_{\text{max}}$ and $P_{\text{min}}$ are the maximum and the minimum of $P$ over $X$, respectively. By the unbiasedness assumption $\mathrm{(A2)}$, we have $\mathbb{E}[\langle \delta^{k}, G_{\gamma_{k}}(x^{k}) \rangle \mid \xi^{[k-1]}] = 0$. In addition, denoting $\delta^{k}_{i,l} = \nabla_{x_{i}}\Tilde{f}_{i}(x^{k}_{i}, x^{k}_{-i}, \xi^{k}_{i,l}) - \nabla_{x_{i}}f_{i}(x^{k})$, $i = 1, \dots, N$, $l = 1, \dots, S_{k}$, $k = 1, \dots, T$, $\Lambda^{k}_{i, j} = \sum_{l=1}^{j} \delta^{k}_{i,l}$, $j = 1, \dots, S_{k}$, and $\Lambda^{k}_{i, 0} = 0$, and noting that for any $i\in [N]$, we have that
    \begin{equation*}
        \mathbb{E}[ \langle \Lambda^{k}_{i, l-1}, \delta^{k}_{i, l} \rangle \mid \Lambda^{k}_{i, l-1} ] = 0,~ \forall l = 1, \dots, S_{k}.
    \end{equation*}
    Then we have
    \begin{align*}
        \mathbb{E}[ \|\Lambda^{k}_{i, S_{k}}\|^{2} \mid \Lambda^{k}_{i, S_{k}-1} ] &= \mathbb{E}[ \|\Lambda^{k}_{i, S_{k}-1}\|^{2} + 2\langle \Lambda^{k}_{i, S_{k}-1}, \delta^{k}_{i, S_{k}} \rangle + \|\delta^{k}_{i, S_{k}}\|^{2} \mid \Lambda^{k}_{i, S_{k}-1} ] \\
        &= \mathbb{E}[ \|\Lambda^{k}_{i, S_{k}-1}\|^{2} \mid \Lambda^{k}_{i, S_{k}-1} ] + \mathbb{E}[ \|\delta^{k}_{i, S_{k}}\|^{2} \mid \Lambda^{k}_{i, S_{k}-1} ].
    \end{align*}
    By taking unconditional expectations on both sides, it follows that
    \begin{align*}
        \mathbb{E}[ \|\Lambda^{k}_{i, S_{k}}\|^{2} ] &= \mathbb{E}[ \|\Lambda^{k}_{i, S_{k}-1}\|^{2} ] + \mathbb{E}[ \|\delta^{k}_{i, S_{k}}\|^{2} ] \\ 
        &= \mathbb{E}[ \|\Lambda^{k}_{i, S_{k}-2}\|^{2} ] + \mathbb{E}[\|\delta^{k}_{i, S_{k-1}}\|^{2} ] + \mathbb{E}[\|\delta^{k}_{i, S_{k}}\|^{2} ] = \cdots = \sum_{l=1}^{S_{k}} \mathbb{E}[ \|\delta^{k}_{i, l}\|^{2} ].
    \end{align*}
    By the fact that $\Tilde{g}_{i}(x^{k}, \xi_{i}^{k}) = \frac{1}{S_{k}}\sum_{l=1}^{S_{k}} \nabla_{x_{i}}\Tilde{f}_{i}(x^{k}_{i}, x^{k}_{-i}, \xi^{k}_{i, l})$, it leads to
    \begin{align*}
        \Tilde{g}_{i}(x^{k}, \xi_{i}^{k}) - \nabla_{x_{i}}f_{i}(x^{k}) = \frac{1}{S_{k}}\sum_{l=1}^{S_{k}} (\nabla_{x_{i}}\Tilde{f}_{i}(x^{k}_{i}, x^{k}_{-i}, \xi^{k}_{i, l}) - \nabla_{x_{i}}f_{i}(x^{k})) \implies \delta^{k}_{i} = \frac{1}{S_{k}}\sum_{l=1}^{S_{k}} \delta^{k}_{i, l}.
    \end{align*}
    Since $\mathbb{E}[\|\delta^{k}_{i, l}\|^{2}] \leq \sigma^{2}$ for any $l = 1, \dots, S_{k}$ by assumption $\mathrm{(A3)}$, it follows that
    \begin{equation}\label{RSGR-proof-2}
        \mathbb{E}[\|\delta^{k}_{i}\|^{2}] = \mathbb{E}\left[ \left\| \frac{1}{S_{k}}\sum_{l=1}^{S_{k}} \delta^{k}_{i, l} \right\|^{2} \right] = \frac{1}{S^{2}_{k}} \mathbb{E}[ \| \Lambda^{k}_{i, S_{k}} \|^{2} ] = \frac{1}{S^{2}_{k}} \sum_{l=1}^{S_{k}} \mathbb{E}[ \|\delta^{k}_{i, l}\|^{2} ] \leq \frac{\sigma^{2}}{S_{k}} \implies \mathbb{E}[\|\delta^{k}\|^{2}]\leq \frac{\sigma^{2}N}{S_{k}}.
    \end{equation}
    We sequentially take the conditional expectations $\mathbb{E}[\bullet \mid \xi^{[T-1]}]$, $\mathbb{E}[\bullet \mid \xi^{[T-2]}]$, \dots, $\mathbb{E}[\bullet \mid \xi^{[0]}]$ and finally take the unconditional expectation on both sides of \eqref{RSGR-proof-1}, by tower property we have
    \begin{align*}
        \sum_{k=1}^{T}(\gamma_{k}-L(\gamma_{k})^{2})\mathbb{E}[\| \Tilde{G}_{\gamma_{k}}(x^{k}, \xi^{k}) \|^{2}] \leq P_{\text{max}} - P_{\text{min}} + \sum_{k=1}^{T}\gamma_{k}\mathbb{E}[\|\delta^{k}\|^{2}] \leq P_{\text{max}} - P_{\text{min}} + \sigma^{2}N \sum_{k=1}^{T} \frac{\gamma_{k}}{S_{k}}.
    \end{align*}
    Then, since $\sum_{k=1}^{T}(\gamma_{k}-L(\gamma_{k})^{2})>0$ by our assumption, dividing both sides of the above inequality by $\sum_{k=1}^{T}(\gamma_{k}-L(\gamma_{k})^{2})$ and noting that
    \begin{equation*}
        \mathbb{E}[\| \Tilde{G}_{\gamma_{R}}(x^{R}, \xi^{R}) \|^{2}] = \frac{\sum_{k=1}^{T}(\gamma_{k}-L(\gamma_{k})^{2})\mathbb{E}[\| \Tilde{G}_{\gamma_{k}}(x^{k}, \xi^{k}) \|^{2}]}{\sum_{k=1}^{T}(\gamma_{k}-L(\gamma_{k})^{2})} \leq \frac{LD^{2}+(\sigma^{2}N)\sum_{k=1}^{T}(\gamma_{k}/S_{k})}{\sum_{k=1}^{T}(\gamma_{k}-L(\gamma_{k})^{2})},
    \end{equation*}
    we establish (i).

    \noindent (ii) We first show that
    \begin{equation}\label{RSGR-proof-3}
        \mathbb{E}[\| G_{\gamma_{R}}(x^{R}) \|^{2}] \leq \frac{8L^{2}D^{2}}{T} + \frac{6\sigma^{2}N}{S}.
    \end{equation}
    By (i) and $S_{k} = S$, we have that
    \begin{equation*}
        \mathbb{E}[\| \Tilde{G}_{\gamma_{R}}(x^{R}, \xi^{R}) \|^{2}] \leq \frac{LD^{2}+\frac{\sigma^{2}N}{S}\sum_{k=1}^{T}\gamma_{k}}{\sum_{k=1}^{T}(\gamma_{k}-L(\gamma_{k})^{2})},
    \end{equation*}
    which together with $\gamma_{k} = 1/(2L)$ for all $k = 1, \dots, T$ imply that
    \begin{equation*}
        \mathbb{E}[\| \Tilde{G}_{\gamma_{R}}(x^{R}, \xi^{R}) \|^{2}] \leq \frac{LD^{2} + \frac{\sigma^{2}NT}{2SL}}{\frac{T}{4L}} = \frac{4L^{2}D^{2}}{T} + \frac{2\sigma^{2}N}{S}.
    \end{equation*}
    Then, by \cite[Proposition 1]{ghadimi-lan-zhang-2016}, we have from the above inequality and \eqref{RSGR-proof-2} that
    \begin{align*}
        \mathbb{E}[\| G_{\gamma_{R}}(x^{R}) \|^{2}] &\leq 2 \mathbb{E}[\| \Tilde{G}_{\gamma_{R}}(x^{R}, \xi^{R}) \|^{2}] + 2 \mathbb{E}[\| G_{\gamma_{R}}(x^{R}) - \Tilde{G}_{\gamma_{R}}(x^{R}, \xi^{R}) \|^{2}] \\
        &\leq 2 \left( \frac{4L^{2}D^{2}}{T} + \frac{2\sigma^{2}N}{S} \right) + 2\mathbb{E}[\| F(x^{R}) - \Tilde{F}(x^{R}, \xi^{R}) \|^{2}] \\
        &= 2 \left( \frac{4L^{2}D^{2}}{T} + \frac{2\sigma^{2}N}{S} \right) + 2\mathbb{E}[\|\delta^{R}\|^{2}] \\
        &= \frac{8L^{2}D^{2}}{T} + \frac{4\sigma^{2}N}{S} + 2\mathbb{E} \left[ \sum_{i=1}^{N}\|\delta^{R}_{i}\|^{2} \right] \\
        &\leq \frac{8L^{2}D^{2}}{T} + \frac{6\sigma^{2}N}{S},
    \end{align*}
    which completes the proof of \eqref{RSGR-proof-3}. We know that the RSG algorithm can perform at most $T = \lfloor M/(SN) \rfloor$ iterations. Obviously, we have $T \geq M/(2SN)$ since we assume that $M$ is sufficiently large. With this observation and \eqref{RSGR-proof-3}, together with \eqref{S-selection}, it follows that
    \begin{align*}
        \mathbb{E}[\| G_{\gamma_{R}}(x^{R}) \|^{2}] &\leq \frac{16 S L^{2}D^{2} N}{M} + \frac{6\sigma^{2}N}{S} \\
        &\leq \frac{16 L^{2}D^{2} N}{M} \left( 1 + \frac{\sigma \sqrt{6M}}{4LD} \right) + 6\sigma^{2}N \cdot \frac{4LD}{\sigma\sqrt{6M}} \\
        &= \frac{16 L^{2}D^{2} N}{M} + \frac{8\sqrt{6}LDN\sigma}{\sqrt{M}},
    \end{align*}
    which completes the proof.
\end{proof}

~

\noindent\textbf{Acknowledgments.} The author is grateful to Uday V. Shanbhag for providing helpful suggestions.

\bibliographystyle{plain}
\bibliography{references}

@inproceedings{alacaoglu-kim-wright-2024,
  title={Revisiting inexact fixed-point iterations for min-max problems: stochasticity and structured nonconvexity},
  author={Alacaoglu, Ahmet and Kim, Donghwan and Wright, Stephen J},
  booktitle={International Conference on Machine Learning},
  pages={840--878},
  year={2024}
}

@article{aussel-svensson-2020,
  title={A short state of the art on multi-leader-follower games},
  author={Aussel, Didier and Svensson, Anton},
  journal={Bilevel Optimization: Advances and Next Challenges},
  pages={53--76},
  year={2020},
  publisher={Springer}
}

@article{arjevani-carmon-duchi-foster-srebro-woodworth-2023,
  title={Lower bounds for non-convex stochastic optimization},
  author={Arjevani, Yossi and Carmon, Yair and Duchi, John C. and Foster, Dylan J. and Srebro, Nathan and Woodworth, Blake},
  journal={Mathematical Programming},
  volume={199},
  number={1},
  pages={165--214},
  year={2023},
  publisher={Springer}
}

@article{burke-curtis-lewis-overton-simoes-2020,
  title={Gradient sampling methods for nonsmooth optimization},
  author={Burke, James V. and Curtis, Frank E. and Lewis, Adrian S. and Overton, Michael L. and Sim{\~o}es, Lucas E. A.},
  journal={Numerical Nonsmooth Optimization: State of The Art Algorithms},
  pages={201--225},
  year={2020},
  publisher={Springer}
}

@article{burke-lewis-overton-2005,
  title={A robust gradient sampling algorithm for nonsmooth, nonconvex optimization},
  author={Burke, James V. and Lewis, Adrian S. and Overton, Michael L.},
  journal={SIAM Journal on Optimization},
  volume={15},
  number={3},
  pages={751--779},
  year={2005},
  publisher={SIAM}
}

@inproceedings{cai-alacaoglu-diakonikolas-2023,
  title={Variance reduced halpern iteration for finite-sum monotone inclusions},
  author={Cai, Xufeng and Alacaoglu, Ahmet and Diakonikolas, Jelena},
  booktitle={International Conference on Learning Representations},
  volume={2024},
  pages={42693--42725},
  year={2023}
}

@inproceedings{cai-oikonomou-zheng-2022,
  title={Accelerated algorithms for monotone inclusion and constrained nonconvex-nonconcave min-max optimization},
  author={Cai, Yang and Oikonomou, Argyris and Zheng, Weiqiang},
  booktitle={OPT 2022: Optimization for Machine Learning (NeurIPS workshop)},
  year={2022}
}

@inproceedings{cai-zheng-2022,
  title={Accelerated single-call methods for constrained min-max optimization},
  author={Cai, Yang and Zheng, Weiqiang},
  booktitle={OPT 2022: Optimization for Machine Learning (NeurIPS workshop)},
  year={2022}
}

@inproceedings{chen-sun-yin-2021,
  title={Closing the gap: Tighter analysis of alternating stochastic gradient methods for bilevel problems},
  author={Chen, Tianyi and Sun, Yuejiao and Yin, Wotao},
  booktitle={Advances in Neural Information Processing Systems},
  volume={34},
  pages={25294--25307},
  year={2021}
}

@book{cui-pang-2021,
  title={Modern Nonconvex Nondifferentiable Optimization},
  author={Cui, Ying and Pang, Jong-Shi},
  year={2021},
  publisher={MOS-SIAM Series on Optimization}
}

@article{cui-shanbhag-2023,
  title={On the computation of equilibria in monotone and potential stochastic hierarchical games},
  author={Cui, Shisheng and Shanbhag, Uday V.},
  journal={Mathematical Programming},
  volume={198},
  number={2},
  pages={1227--1285},
  year={2023},
  publisher={Springer}
}

@article{cui-shanbhag-staudigl-2025,
  title={A regularized variance-reduced modified extragradient method for stochastic hierarchical games},
  author={Cui, Shisheng and Shanbhag, Uday V. and Staudigl, Mathias},
  journal={Journal of Optimization Theory and Applications},
  volume={206},
  number={1},
  pages={11},
  year={2025},
  publisher={Springer}
}

@article{cui-shanbhag-yousefian-2023,
  title={Complexity guarantees for an implicit smoothing-enabled method for stochastic {MPEC}s},
  author={Cui, Shisheng and Shanbhag, Uday V. and Yousefian, Farzad},
  journal={Mathematical Programming},
  volume={198},
  number={2},
  pages={1153--1225},
  year={2023},
  publisher={Springer}
}

@article{dang-lan-2015,
  title={On the convergence properties of non-euclidean extragradient methods for variational inequalities with generalized monotone operators},
  author={Dang, Cong D. and Lan, Guanghui},
  journal={Computational Optimization and Applications},
  volume={60},
  number={2},
  pages={277--310},
  year={2015},
  publisher={Springer}
}

@inproceedings{diakonikolas-daskalakis-jordan-2021,
  title={Efficient methods for structured nonconvex-nonconcave min-max optimization},
  author={Diakonikolas, Jelena and Daskalakis, Constantinos and Jordan, Michael I},
  booktitle={International Conference on Artificial Intelligence and Statistics},
  pages={2746--2754},
  year={2021},
  organization={PMLR}
}

@article{demiguel-xu-2009,
  title={A stochastic multiple-leader Stackelberg model: analysis, computation, and application},
  author={DeMiguel, Victor and Xu, Huifu},
  journal={Operations Research},
  volume={57},
  number={5},
  pages={1220--1235},
  year={2009},
  publisher={INFORMS}
}

@book{facchinei-pang-2003,
  title={Finite-Dimensional Variational Inequalities and Complementarity Problems},
  author={Facchinei, Francisco and Pang, Jong-Shi},
  year={2003},
  publisher={Springer}
}

@book{facchinei-pang-2009,
  title={Nash Equilibria: The Variational Approach},
  author={Facchinei, Francisco and Pang, Jong-Shi},
  year={2009},
  publisher={Convex Optimization in Signal Processing and Communications, Cambridge University Press}
}

@inproceedings{hsieh-iutzeler-malick-mertikopoulos-2020,
  title={Explore aggressively, update conservatively: Stochastic extragradient methods with variable stepsize scaling},
  author={Hsieh, Yu-Guan and Iutzeler, Franck and Malick, J{\'e}r{\^o}me and Mertikopoulos, Panayotis},
  booktitle={Advances in Neural Information Processing Systems},
  volume={33},
  pages={16223--16234},
  year={2020}
}

@article{fang-peterson-1982,
  title={Generalized variational inequalities},
  author={Fang, S.C. and Peterson, E.L.},
  journal={Journal of Optimization Theory and Applications},
  volume={38},
  number={3},
  pages={363--383},
  year={1982},
  publisher={Springer}
}

@article{ghadimi-lan-2013,
  title={Stochastic first-and zeroth-order methods for nonconvex stochastic programming},
  author={Ghadimi, Saeed and Lan, Guanghui},
  journal={SIAM Journal on Optimization},
  volume={23},
  number={4},
  pages={2341--2368},
  year={2013},
  publisher={SIAM}
}

@article{ghadimi-lan-zhang-2016,
  title={Mini-batch stochastic approximation methods for nonconvex stochastic composite optimization},
  author={Ghadimi, Saeed and Lan, Guanghui and Zhang, Hongchao},
  journal={Mathematical Programming},
  volume={155},
  number={1},
  pages={267--305},
  year={2016},
  publisher={Springer}
}

@article{goldstein-1977,
  title={Optimization of {L}ipschitz continuous functions},
  author={Goldstein, Allen A.},
  journal={Mathematical Programming},
  volume={13},
  number={1},
  pages={14--22},
  year={1977},
  publisher={Springer}
}

@article{hobbs-pang-2007,
  title={{N}ash-{C}ournot equilibria in electric power markets with piecewise linear demand functions and joint constraints},
  author={Hobbs, Benjamin F. and Pang, Jong-Shi},
  journal={Operations Research},
  volume={55},
  number={1},
  pages={113--127},
  year={2007},
  publisher={INFORMS}
}

@article{hong-wai-wang-yang-2023,
  title={A two-timescale stochastic algorithm framework for bilevel optimization: Complexity analysis and application to actor-critic},
  author={Hong, Mingyi and Wai, Hoi-To and Wang, Zhaoran and Yang, Zhuoran},
  journal={SIAM Journal on Optimization},
  volume={33},
  number={1},
  pages={147--180},
  year={2023},
  publisher={SIAM}
}

@article{huang-zhang-2024,
  title={Beyond monotone variational inequalities: Solution methods and iteration complexities},
  author={Huang, Kevin and Zhang, Shuzhong},
  journal={Pacific Journal of Optimization},
  volume={20},
  number={3},
  pages={403--428},
  year={2024},
  publisher={Yokohama}
}

@article{iusem-jofre-oliveira-thompson-2017,
  title={Extragradient method with variance reduction for stochastic variational inequalities},
  author={Iusem, Alfredo N. and Jofr{\'e}, Alejandro and Oliveira, Roberto Imbuzeiro and Thompson, Philip},
  journal={SIAM Journal on Optimization},
  volume={27},
  number={2},
  pages={686--724},
  year={2017},
  publisher={SIAM}
}

@article{iusem-jofre-oliveira-thompson-2019,
  title={Variance-based extragradient methods with line search for stochastic variational inequalities},
  author={Iusem, Alfredo N. and Jofr{\'e}, Alejandro and Oliveira, Roberto Imbuzeiro and Thompson, Philip},
  journal={SIAM Journal on Optimization},
  volume={29},
  number={1},
  pages={175--206},
  year={2019},
  publisher={SIAM}
}

@article{kannan-shanbhag-2012,
  title={Distributed computation of equilibria in monotone Nash games via iterative regularization techniques},
  author={Kannan, Aswin and Shanbhag, Uday V.},
  journal={SIAM Journal on Optimization},
  volume={22},
  number={4},
  pages={1177--1205},
  year={2012},
  publisher={SIAM}
}

@article{kannan-shanbhag-2019,
  title={Optimal stochastic extragradient schemes for pseudomonotone stochastic variational inequality problems and their variants},
  author={Kannan, Aswin and Shanbhag, Uday V.},
  journal={Computational Optimization and Applications},
  volume={74},
  number={3},
  pages={779--820},
  year={2019},
  publisher={Springer}
}

@article{kiwiel-2007,
  title={Convergence of the gradient sampling algorithm for nonsmooth nonconvex optimization},
  author={Kiwiel, Krzysztof C.},
  journal={SIAM Journal on Optimization},
  volume={18},
  number={2},
  pages={379--388},
  year={2007},
  publisher={SIAM}
}

@article{kiwiel-2010,
  title={A nonderivative version of the gradient sampling algorithm for nonsmooth nonconvex optimization},
  author={Kiwiel, Krzysztof C.},
  journal={SIAM Journal on Optimization},
  volume={20},
  number={4},
  pages={1983--1994},
  year={2010},
  publisher={SIAM}
}

@article{koshal-nedic-shanbhag-2013,
  title={Regularized iterative stochastic approximation methods for stochastic variational inequality problems},
  author={Koshal, Jayash and Nedi{\'c}, Angelia and Shanbhag, Uday V.},
  journal={IEEE Transactions on Automatic Control},
  volume={58},
  number={3},
  pages={594--609},
  year={2013},
  publisher={IEEE}
}

@article{koshal-nedic-shanbhag-2016,
  title={Distributed algorithms for aggregative games on graphs},
  author={Koshal, Jayash and Nedi{\'c}, Angelia and Shanbhag, Uday V.},
  journal={Operations Research},
  volume={64},
  number={3},
  pages={680--704},
  year={2016},
  publisher={INFORMS}
}

@article{kotsalis-lan-li-2022,
  title={Simple and optimal methods for stochastic variational inequalities, {I}: operator extrapolation},
  author={Kotsalis, Georgios and Lan, Guanghui and Li, Tianjiao},
  journal={SIAM Journal on Optimization},
  volume={32},
  number={3},
  pages={2041--2073},
  year={2022},
  publisher={SIAM}
}

@inproceedings{kovalev-gasnikov-2022,
  title={The first optimal algorithm for smooth and strongly-convex-strongly-concave minimax optimization},
  author={Kovalev, Dmitry and Gasnikov, Alexander},
  booktitle={Advances in Neural Information Processing Systems},
  volume={35},
  pages={14691--14703},
  year={2022}
}

@article{kuhn-shafiee-wiesemann-2025,
  title={Distributionally robust optimization},
  author={Kuhn, Daniel and Shafiee, Soroosh and Wiesemann, Wolfram},
  journal={Acta Numerica},
  volume={34},
  pages={579--804},
  year={2025},
  publisher={Cambridge University Press}
}

@inproceedings{lee-kim-2021,
  title={Fast extra gradient methods for smooth structured nonconvex-nonconcave minimax problems},
  author={Lee, Sucheol and Kim, Donghwan},
  booktitle={Advances in Neural Information Processing Systems},
  volume={34},
  pages={22588--22600},
  year={2021}
}

@article{lei-shanbhag-2020,
  title={Asynchronous schemes for stochastic and misspecified potential games and nonconvex optimization},
  author={Lei, Jinlong and Shanbhag, Uday V.},
  journal={Operations Research},
  volume={68},
  number={6},
  pages={1742--1766},
  year={2020},
  publisher={INFORMS}
}

@article{lei-shanbhag-2022,
  title={Distributed variable sample-size gradient-response and best-response schemes for stochastic Nash equilibrium problems},
  author={Lei, Jinlong and Shanbhag, Uday V.},
  journal={SIAM Journal on Optimization},
  volume={32},
  number={2},
  pages={573--603},
  year={2022},
  publisher={SIAM}
}

@article{lei-shanbhag-chen-2026,
  title={A Distributed Iterative Tikhonov Method for Networked Monotone Stochastic and Hierarchical Aggregative Games},
  author={Lei, Jinlong and Shanbhag, Uday V. and Chen, Jie},
  journal={Set-Valued and Variational Analysis},
  volume={34},
  number={2},
  pages={15},
  year={2026},
  publisher={Springer}
}

@article{lei-shanbhag-pang-sen-2020,
  title={On synchronous, asynchronous, and randomized best-response schemes for stochastic Nash games},
  author={Lei, Jinlong and Shanbhag, Uday V. and Pang, Jong-Shi and Sen, Suvrajeet},
  journal={Mathematics of Operations Research},
  volume={45},
  number={1},
  pages={157--190},
  year={2020},
  publisher={INFORMS}
}

@inproceedings{lin-jin-jordan-2020,
  title={Near-optimal algorithms for minimax optimization},
  author={Lin, Tianyi and Jin, Chi and Jordan, Michael I.},
  booktitle={Conference on Learning Theory},
  pages={2738--2779},
  year={2020},
  organization={PMLR}
}

@inproceedings{lin-zheng-jordon-2022,
  title={Gradient-free methods for deterministic and stochastic nonsmooth nonconvex optimization},
  author={Lin, Tianyi and Zheng, Zeyu and Jordan, Michael I.},
  booktitle={Advances in Neural Information Processing Systems},
  volume={35},
  pages={26160--26175},
  year={2022}
}

@article{marrinan-shanbhag-yousefian-2026,
  title={Zeroth-order gradient and quasi-newton methods for nonsmooth nonconvex stochastic optimization},
  author={Marrinan, Luke and Shanbhag, Uday V. and Yousefian, Farzad},
  journal={SIAM Journal on Optimization},
  volume={36},
  number={2},
  pages={564--596},
  year={2026},
  publisher={SIAM}
}

@article{metzler-hobbs-pang-2003,
  title={{N}ash-{C}ournot equilibria in power markets on a linearized DC network with arbitrage: Formulations and properties},
  author={Metzler, Carolyn and Hobbs, Benjamin F. and Pang, Jong-Shi},
  journal={Networks and Spatial Economics},
  volume={3},
  number={2},
  pages={123--150},
  year={2003},
  publisher={Springer}
}

@article{monderer-shapley-1996,
  title={Potential games},
  author={Monderer, Dov and Shapley, Lloyd S.},
  journal={Games and Economic Behavior},
  volume={14},
  number={1},
  pages={124--143},
  year={1996},
  publisher={Elsevier}
}

@book{mordukhovich-nam-2023,
  title={An Easy Path to Convex Analysis and Applications},
  author={Mordukhovich, Boris S. and Nam, Nguyen M.},
  year={2023},
  publisher={Springer}
}

@article{muu-nguyen-quy-2008,
  title={On {N}ash-{C}ournot oligopolistic market equilibrium models with concave cost functions},
  author={Muu, Le D. and Nguyen, V. H. and Quy, N. V.},
  journal={Journal of Global Optimization},
  volume={41},
  number={3},
  pages={351--364},
  year={2008},
  publisher={Springer}
}

@article{nash-1951,
  title={Non-cooperative games},
  author={Nash, John},
  journal={Annals of Mathematics},
  volume={54},
  number={2},
  pages={286-295},
  year={1951}
}

@article{nemirovski-juditsky-lan-shapiro-2009,
  title={Robust stochastic approximation approach to stochastic programming},
  author={Nemirovski, Arkadi and Juditsky, Anatoli and Lan, Guanghui and Shapiro, Alexander},
  journal={SIAM Journal on Optimization},
  volume={19},
  number={4},
  pages={1574--1609},
  year={2009},
  publisher={SIAM}
}

@book{osullivan-sheffrin-swan-2003,
  title={Economics: Principles in Action},
  author={O'sullivan, Arthur and Sheffrin, Steven M. and Swan, Kathy},
  year={2003}
}

@book{pang-razaviyayn-2016,
  title={A unified distributed algorithm for noncooperative games},
  author={Pang, Jong-Shi and Razaviyayn, Meisam},
  year={2016},
  publisher={Big Data over Networks, Cambridge University Press}
}

@article{pang-scutari-2011,
  title={Nonconvex games with side constraints},
  author={Pang, Jong-Shi and Scutari, Gesualdo},
  journal={SIAM Journal on Optimization},
  volume={21},
  number={4},
  pages={1491--1522},
  year={2011},
  publisher={SIAM}
}

@article{pang-scutari-2013,
  title={Joint sensing and power allocation in nonconvex cognitive radio games: Quasi-Nash equilibria},
  author={Pang, Jong-Shi and Scutari, Gesualdo},
  journal={IEEE Transactions on Signal Processing},
  volume={61},
  number={9},
  pages={2366--2382},
  year={2013},
  publisher={IEEE}
}

@article{patriksson-wynter-1999,
  title={Stochastic mathematical programs with equilibrium constraints},
  author={Patriksson, Michael and Wynter, Laura},
  journal={Operations Research Letters},
  volume={25},
  number={4},
  pages={159--167},
  year={1999},
  publisher={Elsevier}
}

@article{polyak-juditsky-1992,
  title={Acceleration of stochastic approximation by averaging},
  author={Polyak, Boris T. and Juditsky, Anatoli B.},
  journal={SIAM Journal on Control and Optimization},
  volume={30},
  number={4},
  pages={838--855},
  year={1992},
  publisher={SIAM}
}

@inproceedings{qiu-shanbhag-yousefian-2023,
  title={Zeroth-order methods for nondifferentiable, nonconvex, and hierarchical federated optimization},
  author={Qiu, Yuyang and Shanbhag, Uday V. and Yousefian, Farzad},
  booktitle={Advances in Neural Information Processing Systems},
  volume={36},
  pages={3425--3438},
  year={2023}
}

@article{ravat-shanbhag-2011,
  title={On the characterization of solution sets of smooth and nonsmooth convex stochastic {N}ash games},
  author={Ravat, Uma and Shanbhag, Uday V},
  journal={SIAM Journal on Optimization},
  volume={21},
  number={3},
  pages={1168--1199},
  year={2011},
  publisher={SIAM}
}

@phdthesis{razaviyayn-2014,
  title={Successive Convex Approximation: Analysis and Applications},
  author={Razaviyayn, Meisam},
  year={2014},
  school={University of Minnesota}
}

@article{robbins-monro-1951,
  title={A stochastic approximation method},
  author={Robbins, Herbert and Monro, Sutton},
  journal={The Annals of Mathematical Statistics},
  pages={400--407},
  year={1951},
  publisher={JSTOR}
}

@book{rockafellar-wets-1998,
  title={Variational Analysis},
  author={Rockafellar, R. Tyrrell and Wets, Roger J-B},
  year={1998},
  publisher={Springer}
}

@book{sandholm-2010,
  title={Population Games and Evolutionary Dynamics},
  author={Sandholm, William H.},
  year={2010},
  publisher={MIT press}
}

@article{scutari-palomar-facchinei-pang-2010,
  title={Convex optimization, game theory, and variational inequality theory},
  author={Scutari, Gesualdo and Palomar, Daniel P. and Facchinei, Francisco and Pang, Jong-Shi},
  journal={IEEE Signal Processing Magazine},
  volume={27},
  number={3},
  pages={35--49},
  year={2010},
  publisher={IEEE}
}

@article{scutari-palomar-pang-facchinei-2009,
  title={Flexible design of cognitive radio wireless systems},
  author={Scutari, Gesualdo and Palomar, Daniel P. and Pang, Jong-Shi and Facchinei, Francisco},
  journal={IEEE Signal Processing Magazine},
  volume={26},
  number={5},
  pages={107--123},
  year={2009},
  publisher={IEEE}
}

@article{steklov-1907,
  title={Sur les expressions asymptotiques de certaines fonctions définies par les equations differentielles du second ordre et leurs applications au probleme du developement dune fonction arbitraire en series procedant suivant les diverses fonctions},
  author={Steklov, V.A.},
  journal={Communications de la Société mathématique de Kharkow},
  volume={10},
  pages={97--199},
  year={1907}
}

@inproceedings{vankov-nedic-sankar-2023,
  title={Last iterate convergence of Popov method for non-monotone stochastic variational inequalities},
  author={Vankov, Daniil and Nedi{\'c}, Angelia and Sankar, Lalitha},
  booktitle={OPT2023: 15th Annual Workshop on Optimization for Machine Learning},
  pages={1--27},
  year={2023}
}

@inproceedings{wang-li-2020,
  title={Improved algorithms for convex-concave minimax optimization},
  author={Wang, Yuanhao and Li, Jian},
  booktitle={Advances in Neural Information Processing Systems},
  volume={33},
  pages={4800--4810},
  year={2020}
}

@article{xiao-shanbhag-2025-GR,
  title={Computing equilibria in stochastic nonconvex and non-monotone games via gradient-response schemes},
  author={Xiao, Zhuoyu and Shanbhag, Uday V.},
  journal={arXiv:2504.14056v3},
  year={2025}
}

@article{xiao-shanbhag-2026-BR,
  title={Equilibrium invariance, proximality, and surrogation: Moreau-smoothed best-response pathways in stochastic nonsmooth games},
  author={Xiao, Zhuoyu and Shanbhag, Uday V.},
  journal={arXiv:2603.00934},
  year={2026}
}

@article{yousefian-nedic-shanbhag-2016,
  title={Self-tuned stochastic approximation schemes for non-Lipschitzian stochastic multi-user optimization and Nash games},
  author={Yousefian, Farzad and Nedi{\'c}, Angelia and Shanbhag, Uday V.},
  journal={IEEE Transactions on Automatic Control},
  volume={61},
  number={7},
  pages={1753--1766},
  year={2016},
  publisher={IEEE}
}

@article{yousefian-nedic-shanbhag-2017,
  title={On smoothing, regularization, and averaging in stochastic approximation methods for stochastic variational inequality problems},
  author={Yousefian, Farzad and Nedi{\'c}, Angelia and Shanbhag, Uday V.},
  journal={Mathematical Programming},
  volume={165},
  number={1},
  pages={391--431},
  year={2017},
  publisher={Springer}
}
 
\end{document}